\documentclass[11pt]{article}

\usepackage[vmargin=1in, hmargin=1in]{geometry}
\usepackage{amsfonts}
\usepackage{amsthm,amsmath,amssymb}

\newtheorem{theorem}{Theorem}
\newtheorem{lemma}{Lemma}

\usepackage{tikz}

\newenvironment{comment}{\setbox0=\vbox\bgroup}{\egroup}

\newcommand{\Po}[1]{\tilde{#1}}             
\newcommand\Me[1]{ {#1^*} }     

\def\R{ {\mathbb{R}}}
\newcommand\ceil[1]{ \lceil {#1} \rceil}
\newcommand\floor[1]{ \lfloor {#1} \rfloor}
\newcommand\dee[1]{ {\;\mathrm{d}#1}}

\usepackage{todonotes}

\def\lintersect{\cap}

\def\Z{ {\mathbb{Z}}}

\def\parsec{\par\noindent}

\def\E{ {\mathbb{E}}}
\def\Var{ {\mathrm{Var}}}

\def\Normal{ {\mathcal{N}}}
\def\logicor{ {~\lor~}}

\def\P{ {\mathcal{P}}}

\def\Co{\mathcal{C}}
\def\Arg{\mathrm{Arg}}

\def\P{ {\frak{P}}}
\def\Bernoulli{ {\mathrm{Bernoulli}}}

\newcommand\spreadout[1]{ #1 }

\newcommand\fracpart[1]{ \{ #1 \} }
\newcommand\affilskip[1]{ #1 }

\newtheorem{example}{Example}

\begin{document}
\title{Asymmetric R\'enyi Problem}
\author
       {\spreadout{M. DRMOTA}{$^1$}%
\thanks{Research partially supported by the Austrian Science Foundation
FWF Grant~No. F50-02.}
and \spreadout{A. MAGNER}{$^2$}%
\thanks{Research supported by 
NSF Center for Science of Information (CSoI) Grant CCF-0939370.}
and \spreadout{W. SZPANKOWSKI}{$^§$}%
\thanks{Research partially supported by the
NSF Center for Science of Information (CSoI) Grant CCF-0939370, 
and in addition by the NSF Grants CCF-1524312, and  
NIH Grant 1U01CA198941-01.}  \\
\affilskip {$^1$} Institute for Discrete Mathematics and Geometry, TU Wien, Vienna, Austria, \\
\affilskip  A-1040 Wien, Wiedner Hauptstr. 8--10 \\
{$^2$} Coordinated Science Lab,  UIUC,  Champaign, \\ 
\affilskip  IL 61820, USA\\
{$^3$} Department of Computer Science, Purdue University, \\
\affilskip IN 47907, USA \\
michael.drmota@tuwien.ac.at, anmagner@illinois.edu, and szpan@purdue.edu}



\label{firstpage}
\maketitle

\begin{abstract}
In 1960 R\'enyi in his Michigan State University lectures asked for the number 
of random queries necessary to recover a hidden bijective labeling of  
$n$ distinct objects. In each query one selects a random subset of labels and asks, 
which objects have these labels?
We consider here an asymmetric version of
the problem in which in every query an object is chosen with probability 
$p > 1/2$ and we ignore ``inconclusive'' queries.
We study the number of queries needed to recover the labeling in 
its entirety ($H_n$), 
before at least one element is recovered ($F_n$),
and to recover a randomly chosen element $(D_n)$.
This problem exhibits several remarkable behaviors: 
$D_n$ converges in probability but not almost surely;
$H_n$ and 
$F_n$
exhibit phase transitions with respect
to $p$ in the second term. 
We prove that for $p>1/2$ with high probability (whp) we need
$
H_n=\log_{1/p} n +\frac 12 \log_{p/(1-p)}\log n +o(\log \log  n)
$
queries to recover the entire bijection. This should be compared to
its symmetric ($p=1/2$) counterpart established by Pittel and Rubin,
who proved that in this case one requires
$
H_n=\log_{2} n +\sqrt{2 \log_{2} n} +o(\sqrt{\log n})
$
queries. 
As a bonus, our analysis implies novel results for random PATRICIA tries, as 
the problem is probabilistically equivalent to
that of the height, fillup level, and typical depth  
of a PATRICIA trie built from $n$ independent
binary sequences generated by a biased($p$) memoryless source.
\end{abstract}

\section{Introduction}

In his lectures in the summer of 1960 at Michigan State University,
Alfred R\'enyi discussed several problems related to random sets \cite{renyi}.
Among them there was a problem regarding recovering a labeling of a set $X$ of $n$
distinct objects by asking random subset questions of the form ``which objects correspond to the labels in the (random) set $B$?''
For a given method of randomly selecting queries, R\'enyi's original problem asks for the typical behavior of the number of queries
necessary to recover the hidden labeling.

Formally, the unknown labeling of the set $X$ is a bijection $\phi$ from $X$ to a set $A$ of labels 
(necessarily with equal cardinality $n$), and a query takes the form of a subset $B \subseteq A$.  
The response to a query $B$ is $\phi^{-1}(B) \subseteq X$.  

Our contribution in this paper is a precise analysis of several parameters of R\'enyi's problem for a particular natural probabilistic model
on the query sequence.  In order to formulate this model precisely, it is convenient to first state a view of the process
that elucidates its tree-like structure.  In particular, 
a sequence of queries corresponds to a refinement of partitions of the set of objects, where two objects are in
different partition elements if they have been distinguished by some sequence of queries.
More precisely, the refinement works as follows: 
before any questions are asked, we have a trivial partition $\P_0 = X$ consisting of a single class (all objects).  Inductively,
if $\P_{j-1}$ corresponds to the partition induced by the first $j-1$ queries, then $\P_j$ is constructed from $\P_{j-1}$ by splitting
each element of $\P_{j-1}$ into at most two disjoint subsets: those objects that are contained in the preimage of the $j$th query set $B_j$ and
those that are not.  The hidden labeling is recovered precisely when 
the partition of $X$ consists only of singleton elements.  An instance of 
this process may be viewed as a rooted binary tree (which we call the 
\emph{partition refinement tree}) in which the $j$th level, for $j \geq 0$, corresponds
to the partition resulting from $j$ queries; a node in a given level corresponds to an element of the
partition associated with that level.  A right child corresponds to a subset of a parent partition element that is included
in the subsequent query, and a left child corresponds to a subset that is not included.  
See Example~\ref{QuerySequenceExample} for an illustration. 
\begin{example}[Demonstration of partition refinement]
	\label{QuerySequenceExample}
    Consider an instance of the problem where $X = [5] = \{1, ..., 5\}$, 
    with labels $(d, e, a, c, b)$ respectively (so $A = \{a, b, c, d, e\}$).
    Consider the following sequence of queries:
	    \tikzstyle{level 1}=[level distance=1.0cm, sibling distance=2.5cm]
	    \tikzstyle{level 2}=[level distance=1.0cm, sibling distance=1.5cm]
	    \tikzstyle{bag} = [rectangle, minimum width=3pt,inner sep=0pt]
	    \tikzstyle{end} = [circle, minimum width=3pt,fill, inner sep=0pt]
		
    \begin{minipage}{0.5\textwidth}
    \begin{enumerate}
    	\item
        	$B_1 = \{b, d\} \mapsto \{1, 5\}$
        \item
        	$B_2 = \{a, b, d\} \mapsto \{1, 3, 5\}$,
        \item
        	$B_3 = \{a, c, d\} \mapsto \{1, 3, 4\}$,
    \end{enumerate}
    \end{minipage}
    \vspace{20pt}
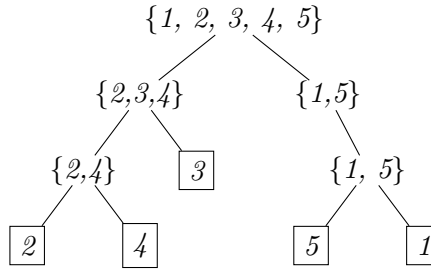
    \begin{minipage}{0.5\textwidth}
	    \begin{tikzpicture}[grow=down]
	    \node[bag]{\{1, 2, 3, 4, 5\}}
	    	child{
	            node[bag]{\{2,3,4\}}
	        	child{ 
	            	node[bag]{\{2,4\}}
                    child{ 
                        node[rectangle,draw]{2}
                    }
                    child{ 
                        node[rectangle,draw]{4}
                    }
	            }
	            child{ 
	            	node[rectangle,draw]{3}
	            }
	        }
	        child{
	            node[bag]{\{1,5\}}
                child{
	                node[bag,right]{\{1, 5\}}
			        	child{ 
			            	node[rectangle,draw]{5}
			                edge from parent
			                node[above]{}
			            }
			            child{ 
			            	node[rectangle,draw]{1}
			                edge from parent
			                node[above]{}
			            }
                }
	        };
	\end{tikzpicture}
	\label{QuerySequenceExampleDiagram}
    \end{minipage}
    Each level $j\geq 0$ of the tree depicts the partition $\P_j$, where a right child node 
    corresponds to the subset of objects in the parent set which are contained in the 
    response to the $j$th query.  Singletons are only explicitly depicted in the first
    level in which they appear.  We can determine the labels of all objects using the tree
    and the sequence of queries: for example, to determine the label of the object $3$, we
    traverse the tree until we reach the leaf corresponding to $3$.  This indicates that
    the label corresponding to $3$ is in the singleton set
    \[
        \neg B_1 \lintersect B_2 = \{a,c,e\} \lintersect \{a, b, d\} = \{a\}.
    \]
    Note that leaves of the tree always correspond to singleton sets.
\end{example}
In this work we consider a version of
the problem in which, in every query, each label is included independently
with probability $p > 1/2$ (the \emph{asymmetric case}) and we \emph{ignore inconclusive queries}.  In particular, if a candidate query fails to non trivially split some element 
of the previous partition, we modify the query by deciding again independently 
whether or not to include each label of that partition element with probability $p$.  
We perform this modification until
the resulting query splits every element of the previous partition non trivially.  
See Example~\ref{QuerySequenceIgnoreExample}.
\begin{example}[Ignoring inconclusive queries]
	\label{QuerySequenceIgnoreExample}
    Continuing Example~\ref{QuerySequenceExample}, the query $B_2$ fails to split the 
    partition element $\{1, 5\}$, so it is an example of an inconclusive query and would 
    be modified in our model to, say, $B_2' = \phi(\{1,3\})$.
    The resulting refinement of partitions is depicted as a tree here.  Note that the tree
    now does not contain non-branching paths and that $B_2$ is ignored in the final query
    sequence.
    
    \begin{minipage}{0.5\textwidth}
    \begin{enumerate}
    	\item
        	$B_1 = \{b, d\} \mapsto  \{1,5\}$
        \item
        	$B_2' = \{a,d\} \mapsto \{1,3\}$
        \item
        	$B_3 = \{a, c, d\} \mapsto \{1, 3, 4\}$.
    \end{enumerate}
    \end{minipage}
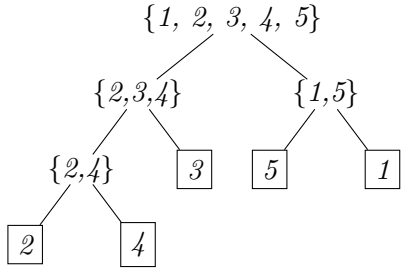
    \begin{minipage}{0.5\textwidth}
    	\vspace{12pt}
 	    \tikzstyle{level 1}=[level distance=1.0cm, sibling distance=2.5cm]
	    \tikzstyle{level 2}=[level distance=1.0cm, sibling distance=1.5cm]
	    \tikzstyle{bag} = [rectangle, minimum width=3pt,inner sep=0pt]
	    \tikzstyle{end} = [circle, minimum width=3pt,fill, inner sep=0pt]
 	    \begin{tikzpicture}[grow=down]
	    \node[bag]{\{1, 2, 3, 4, 5\}}
	    	child{
	            node[bag]{\{2,3,4\}}
	        	child{ 
	            	node[bag]{\{2,4\}}
                    child{ 
                    	node[rectangle,draw]{2}
                    }
                    child{ 
                    	node[rectangle,draw]{4}
                    }
	            }
	            child{ 
	            	node[rectangle,draw]{3}
	            }
	        }
	        child{
	            node[bag]{\{1,5\}}
	        	child{ 
	            	node[rectangle,draw]{5}
	                edge from parent
	                node[above]{}
	            }
	            child{ 
	            	node[rectangle,draw]{1}
	                edge from parent
	                node[above]{}
	            }
	        };
	\end{tikzpicture}   
    \end{minipage}
\end{example}
We study three parameters of this random process: $H_n$, the number of such queries
needed to recover the entire labeling; $F_n$, the number needed before at
least one element is recovered; 
and $D_n$, the number needed to recover an element selected uniformly
at random.  Our objective is to present precise probabilistic estimates 
of these parameters. 

The symmetric version (i.e., $p=1/2$) of the problem (with a variation) was discussed by Pittel 
and Rubin in \cite{pittelrubin1990}, where they analyzed the typical value of $H_n$.  
In their model, a query is constructed by deciding
whether or not to include each label from $A$ independently with probability $p=1/2$.
To make the problem more interesting, they added a constraint similar to ours: namely, a
query is, as in our model, admissible if and only if it splits every nontrivial element of the current
partition.  In contrast with our model, however, Pittel and Rubin completely discard inconclusive queries (rather than modifying their inconclusive subsets as we do).  
Despite this difference, the model considered in \cite{pittelrubin1990} is
probabilistically equivalent to ours for the symmetric case.  Our primary contribution
is the analysis of the problem in the asymmetric case ($p > 1/2$), but our methods
of proof allow us to recover the results of Pittel and Rubin. 

The question asked by R\'enyi brings some surprises.  For the
symmetric model ($p=1/2$) Pittel and Rubin \cite{pittelrubin1990} were
able to prove that the number of necessary queries is with high probability
(whp) (see Theorem~\ref{HeightTheorem}) 
\begin{eqnarray}
	\label{e1}
	H_n = \log_{2} n +\sqrt{2\log_{2} n} +o(\sqrt{\log n}).
\end{eqnarray}
In this paper, we develop a different method that could be used to re-establish this result
{and} prove that for $p > 1/2$ the number
of queries grows whp as
\begin{eqnarray}
	\label{e2}
    H_n = \log_{1/p} n +\frac{1}{2}\log_{p/q} \log n  +o(\log \log  n),
\end{eqnarray}
where $q:=1-p$.
Note a phase transition in the second term.  Moreover, this result is perhaps interesting in the sense that, for $p > 1/2$, $H_n$ exhibits the second-order
behavior that Pittel and Rubin stated that they fully expected but did not find in the $p=1/2$ case \cite{pittelrubin1990}.
We show that another  phase transition, also in the second term, occurs in the asymptotics for $F_n$
(see Theorem~\ref{FillupTheorem}):
\begin{eqnarray} 
\label{e3}
    F_n = \left\{ \begin{array}{ll}
           \log_{1/q} n - \log_{1/q}\log\log n + o(\log\log\log n)     &      p > q \\
        \log_{2} n - \log_2\log n + o(\log\log n)                   &      p = q = 1/2.
    \end{array}   \right. 
\end{eqnarray}
We also state in Theorem~\ref{DepthCorollary} some interesting probabilistic behaviors
of $D_n$.
We have $D_n/\log n  \to 1/h(p)$ (in probability)
where $h(p) := -p\log p - q\log q$, but we do not have almost sure
convergence.  

We establish these results in a novel way by considering first the \emph{external profile}
$B_{n,k}$, whose analysis was, until recently, an open problem of its own (the second 
and third authors gave a precise analysis of the external profile in an important range 
of parameters in \cite{magnerPhD2015,magnerspa2015}, but the present paper requires 
really nontrivial extensions).  
The external profile at level $k$ is
the number of bijection elements revealed by the $k$th query (one may also define the \emph{internal} profile
at level $k$ as the number of non-singleton elements of the partition immediately after the $k$th query).  Its study is motivated by
the fact that many other parameters, including all of those that we mention here, can
be written in terms of it. Indeed, $\Pr[D_n=k]=\E[B_{n,k}]/n$, 
$H_n = \max\{k : ~ B_{n,k} > 0\}$, and $F_n = \min\{k : ~ B_{n,k} > 0\} - 1$.

We now discuss our new results concerning the probabilistic  
behavior of the external profile.  
We establish in \cite{magnerspa2015,magnerPhD2015} precise asymptotic expressions for the 	
expected value and variance of $B_{n,k}$ in the \emph{central range}, that is, 
with $k \sim \alpha\log n$, where, for any fixed $\epsilon > 0$,
$\alpha \in (1/\log(1/q) + \epsilon, 1/\log(1/p) - \epsilon)$
(the left and right endpoints of this interval as $\epsilon \to 0$ are associated
with $F_n$ and $H_n$, respectively).  Specifically,
it was shown that both the mean and the variance
are of the same (explicit) polynomial order of growth
(with respect to $n$). 
More precisely, expected value and variance grow
for $k\sim\alpha \log n$ as
$$
H(\rho(\alpha), \log_{p/q}(p^kn)) ~  \frac{n^{\beta(\alpha)}}{\sqrt{C \log n}}
$$
where $\beta(\alpha)\leq 1$ and $\rho(\alpha)$ are 
complicated functions of $\alpha$, 
$C$ is an explicit constant, and $H(\rho, x)$ is
a function that is periodic in $x$. 
The oscillations come from infinitely many
regularly spaced saddle points that we observe when inverting the Mellin
transform of the Poisson generating function of $\E[B_{n,k}]$.  
Finally, in \cite{magnerspa2015} we prove a central limit theorem;
that is,
${(B_{n,k} - \E[B_{n,k}])}/{\sqrt{\Var[B_{n,k}]}} \to \Normal(0,1) $
where $\Normal(0,1)$ represents the standard normal distribution.  

In order to establish the most interesting results claimed in the present paper
for $H_n$ and $F_n$, the analysis sketched above does not suffice: we need to estimate the
mean and the variance of the external profile \emph{beyond} the range
$\alpha \in (1/\log(1/q) + \epsilon, 1/\log(1/p) - \epsilon)$; in particular,
for $F_n$ and $H_n$ we need expansions at the left and right side (as $\epsilon \to 0$), respectively, of this range.

Having described most of our main results, we mention an important equivalence
pointed out by Pittel and Rubin \cite{pittelrubin1990}.
They observed that their version of the R\'enyi process
resembles the construction of a digital tree
known as a PATRICIA trie\footnote{We recall that a trie is a binary digital tree, where
data that are represented by binary strings are stored at leaves of the tree according
to finite prefixes of the corresponding binary strings in a minimal way such that all
appearing prefixes are different. A PATRICIA trie is a trie in
which non-branching paths are \emph{compressed}; that is, there are no unary paths.}
\cite{knuth1998acp,szpa2001Book}. In fact,
the authors of \cite{pittelrubin1990} show that $H_n$ is probabilistically
equivalent to the height (longest path) of a PATRICIA trie built from
$n$ binary strings generated independently by a memoryless source with bias $p=1/2$
(that is, with a ``1'' generated with probability $p$; this is often called
the \emph{Bernoulli model with bias $p$}); the equivalence is true more generally, for $p \geq 1/2$.
It is easy to see that
$F_n$ is equivalent to the fillup level (depth of the deepest full level),
$D_n$ to the typical depth (depth of a randomly chosen leaf), and $B_{n,k}$
to the external profile of the tree (the number of leaves at level $k$; the internal
profile at level $k$ is similarly defined as the number of non-leaf nodes at that 
level).  We spell out this equivalence in the following simple claim.  
\begin{lemma}[Equivalence of the R\'enyi problem with those of PATRICIA tries]
	\label{EquivalenceLemma}
    Any parameter (in particular, $H_n, F_n, D_n$, and $B_{n,k}$) of the R\'enyi process with bias $p$ 
    that is a function of the partition refinement
    tree is equal in distribution to the same function of a random PATRICIA trie generated
    by $n$ independent infinite binary strings from a memoryless source with bias $p \geq 1/2$.
\end{lemma}    
\begin{proof}
    In a nutshell, we
    couple a random PATRICIA trie and the sequence of queries from the R\'enyi process
    by constructing both from the same sequence of binary strings from a memoryless source.
    We do this in such a way that the resulting PATRICIA trie and the partition refinement
    tree are isomorphic with probability $1$ (in fact, always isomorphic), so that parameters defined in terms of
    either tree structure are equal in distribution. 
    
    More precisely, we start with $n$ independent infinite binary strings $S_1, \ldots, S_n$ 
    generated according to a memoryless source with bias $p$, where each string corresponds, in a way to
    be made precise below, to a
    unique element of the set of labels (for simplicity, we assume that $A = [n]$, and
    $S_j$ is associated to the object $j$, for $j \in [n]$; intuitively, $S_j$ encodes the decision, for each
    query, of whether or not to include $j$).  These induce a PATRICIA trie $T$,
    and our goal is to show that we can simulate a R\'enyi process using these strings,
    such that the corresponding tree $T_R$ is isomorphic to $T$ as a rooted plane--
    oriented tree (see Example~\ref{QuerySequenceIgnoreExample}).  The basic idea is as follows: 
    we maintain for each string $S_j$ an
    index $k_j$, initially set to $1$.  Whenever the R\'enyi process demands
    that we make a decision about whether or not to include label $j$ in a query,
    we include it if and only if $S_{j,k_j} = 1$, and then increment $k_j$ by $1$.
    
    Clearly, this scheme induces the correct distribution on queries.  Furthermore,
    the resulting partition refinement tree (ignoring inconclusive queries) is easily 
    seen to be isomorphic to $T$.
    Since the trees are isomorphic, the parameters of interest are equal in each case.
\end{proof}
Thus, our results on these parameters for the R\'enyi problem directly lead
to novel results on PATRICIA tries, and vice versa.
In addition to their use as data structures, PATRICIA tries also arise as 
combinatorial structures which capture the behavior of various processes 
of interest in computer science and information theory
(e.g., in leader election processes without trivial splits 
\cite{jansonszpa1996} and in the solution to R\'{e}nyi's problem which we study here
\cite{pittelrubin1990, devroye1992}).

Similarly, the version of the R\'enyi problem that allows inconclusive queries corresponds
to results on tries built on $n$ binary strings from a memoryless source.  We thus discuss
them in the literature survey below.


Now we briefly review relevant facts about PATRICIA tries and other digital trees when built over
$n$ independent strings generated by a memoryless source.  
Profiles of tries in both the asymmetric and symmetric cases were studied 
extensively in \cite{park2008}.  The expected profiles
of digital search trees (another common digital tree with connections to Lempel-Ziv parsing) in both cases were analyzed in \cite{drmotaszpa2011}, and
the variance for the asymmetric case was treated
in \cite{kazemi2011}.  Some aspects of trie and PATRICIA trie profiles 
(in particular, the concentration of their distributions) were
studied using probabilistic methods in \cite{devroye2004, devroye2002}.
The depth in PATRICIA for the symmetric model was analyzed in 
\cite{devroye1992,knuth1998acp} while for the asymmetric model in 
\cite{szpa1990}. The leading asymptotics for the PATRICIA height
for the symmetric Bernoulli model was first analyzed by Pittel   
\cite{Pittel85} (see also \cite{szpa2001Book} for suffix trees).
The two-term expression for the height of PATRICIA for the symmetric model 
was first presented in \cite{pittelrubin1990} as discussed above
(see also \cite{devroye1992}).  To our knowledge, precise asymptotics
beyond the leading term for the height and fillup level have not been given
in the asymmetric case for either tries or digital search trees.
Finally, in \cite{magnerPhD2015,magnerspa2015}, the second two 
authors of the present paper presented a precise analysis of the external profile (including
its mean, variance, and limiting distribution) in the asymmetric case, for the range
in which the profile grows polynomially.  The present work relies on this previous analysis,
but the analyses for $H_n$ and $F_n$ involve a significant extension, since they rely on 
precise asymptotics for the external profile outside this central range.

Regarding methodology, the basic framework (which we use here) for analysis of digital tree recurrences by 
applying the Poisson transform to derive a functional equation, converting this to an algebraic
equation using the Mellin transform, and then inverting using the saddle point method/singularity
analysis followed by depoissonization, was worked out in \cite{park2008} and followed in \cite{drmotaszpa2011}.
While this basic chain is common, the challenges of applying it vary dramatically between the different
digital trees, and this is the case here.  As we discuss later (see (\ref{pg1}) and the surrounding text), this variation starts with the quite different
forms of the Poisson functional equations, which lead to unique analytic challenges.  

The plan for the paper is as follows. In the next section we formulate
more precisely our problem and present our main results regarding 
$B_{n,k}$, $H_n$, $F_n$, and $D_n$, along with sketches
of the derivations.  Complete proofs for $H_n$ (and a roadmap for 
the proof for $F_n$) are 
provided in Section~\ref{Proofs}. Section~\ref{sec:depo} provides some
background on the depoissonization step.  Finally, Section~\ref{secmiracle}
details a surprising series identity which arises in the analysis of $H_n$, 
leading to significant complications.

\section{Main Results}
\label{MainResults}

In this section, we formulate precisely R\'enyi's problem and 
present our main results.  
Our goal is to provide precise asymptotics for three natural parameters 
of the R\'enyi problem on $n$ objects with each label in a given 
query being included with probability $p \geq 1/2$:
the number $F_n$ of queries needed before at least a single element of the bijection can be identified, 
the number $H_n$ needed to recover the bijection 
in its entirety, and the number $D_n$ needed to recover an 
element of the bijection chosen uniformly at random from the $n$ objects.  
If one wishes to determine the label for a particular object, 
these quantities correspond to the best, worst, and average case performance, 
respectively, of the random subset strategy proposed by R\'enyi.

We recall that we can express 
$F_n$, $H_n$, and $D_n$ in terms of the {\it profile} $B_{n,k}$
(defined as the number of bijection elements revealed by the $k$th query)
\begin{equation}
    F_n = \min\{k : ~ B_{n,k} > 0\} - 1, \ 
    H_n = \max\{k : ~ B_{n,k} > 0\},  \ 
    \Pr[D_n = k] = \frac{\E[B_{n,k}]}n.  \label{eqrel2}
\end{equation}
Using the first and second moment methods, we can then obtain upper and lower bounds
on $H_n$ and $F_n$ in terms of the moments of $B_{n,k}$:
\begin{eqnarray}
\label{eq-h}
    \Pr[H_n>k] \leq \sum_{j>k} \E[B_{n,j}],  \ \ \  \ \ 
    \Pr[H_n<k] \leq \frac{\Var[B_{n,k}]}{\E[B_{n,k}]^2},
\end{eqnarray}
and
\begin{eqnarray}
\label{eq-f}
    \Pr[F_n > k] \leq \frac{\Var[B_{n,k}]}{\E[B_{n,k}]^2}, \quad
    \Pr[F_n < k] \leq \E[B_{n,k}].
\end{eqnarray}
The analysis of the distribution of $D_n$ reduces simply to that of $\E[B_{n,k}]$, 
as in (\ref{eqrel2}).

Having reduced the analyses of $F_n, H_n$, and $D_n$ to that of the moments of $B_{n,k}$,
we now explain our approach to the latter analysis, starting in Section~\ref{BasicFactsSection} with a review of the work
done in \cite{magnerPhD2015}. 
We will then show in 
Section~\ref{MainResultsSection} how the present paper 
requires extensions far beyond \cite{magnerPhD2015,magnerspa2015}
to give new results on the quantities of interest in the R\'enyi problem.
\subsection{Basic facts for the analysis of $B_{n,k}$}
\label{BasicFactsSection}
Here we recall some facts, worked out in detail in \cite{magnerPhD2015}, which will form the
starting point of the analysis in the present paper.  
In order to derive our main results, 
we need proper asymptotic information about $\E[B_{n,k}]$ and $\Var[B_{n,k}]$ at the boundaries of this region.

We start by deriving a recurrence for the average profile, 
which we denote by $\mu_{n,k} := \E[B_{n,k}]$. It satisfies
\begin{eqnarray}
    \label{muRecurrence}
    \mu_{n,k} = (p^n + q^n)\mu_{n,k} + \sum_{j=1}^{n-1} { n\choose j } p^j q^{n-j} (\mu_{j,k-1} + \mu_{n-j,k-1})
\end{eqnarray}
for $n\geq 2$ and $k \geq 1$, with some initial/boundary conditions; 
most importantly, $\mu_{n,k} = 0$ for $k \geq n$ and any $n$.
Moreover, $\mu_{n,k} \leq n$ for all $n$ and $k$ owing to the elimination of
inconclusive queries.  This recurrence 
arises from conditioning on the number $j$ of objects that are included 
in the first query.  If $1 \leq j \leq n-1$ objects are included,
then the conditional expectation is a sum of contributions from those 
objects that are included and those that aren't.  If, on
the other hand, all objects are included or all are excluded from the 
first potential query (which happens with probability $p^n + q^n$),
then the partition element splitting constraint on the queries applies, 
the potential query is ignored as {\it inconclusive}, 
and the contribution is $\mu_{n,k}$.

The tools that we use to solve this recurrence (for details see \cite{magnerPhD2015,magnerspa2015}) 
are similar to those of the analyses for digital trees \cite{szpa2001Book} 
such as tries and digital search trees (though the analytical details 
differ significantly).  We first derive a functional equation for the 
Poisson transform 
$\Po{G}_k(z) = \sum_{m \geq 0} \mu_{m,k}\frac{z^m}{m!}e^{-z}$ of
$\mu_{n,k}$, which gives 
\[
    \Po{G}_k(z) = \Po{G}_{k-1}(pz) + \Po{G}_{k-1}(qz) + e^{-pz}(\Po{G}_k - 
\Po{G}_{k-1})(qz) + e^{-qz}(\Po{G}_{k} - \Po{G}_{k-1})(pz).
\]
This we write as
\begin{eqnarray}
\label{pg1}
    \Po{G}_k(z) = \Po{G}_{k-1}(pz) + \Po{G}_{k-1}(qz) + \Po{W}_{k,G}(z), 
\end{eqnarray}
and at this point the goal is to determine asymptotics for 
$\Po{G}_k(z)$ as $z\to\infty$ in a cone around the positive real axis.
When solving (\ref{pg1}), $\Po{W}_{k,G}(z)$ significantly complicates the analysis because it has
no closed-form Mellin transform (see below).
Finally, depoissonization \cite{szpa2001Book} will allow us to directly transfer the
asymptotic expansion for $\Po{G}_k(z)$ back to one for $\mu_{n,k}$ since
$\mu_{n,k}$ is well approximated by $\Po{G}_k(n)$.

To convert (\ref{pg1}) to an equation that is easier to handle,
we use the \emph{Mellin transform} 
\cite{Flajolet95mellintransforms}, which, for a function 
$f:\R \to \R$ is given by
\[
    \Me{f}(s) = \int_{0}^\infty z^{s-1} f(z) \dee{z}.
\]
Using the Mellin transform identities and defining $T(s) = p^{-s} + q^{-s}$, we end up with an expression 
for the Mellin transform $\Me{G_k}(s)$ of $\Po{G}_k(z)$ of the form
\[
    \Me{G_k}(s) = \Gamma(s+1)A_k(s)( p^{-s} + q^{-s})^k=
    \Gamma(s+1)A_k(s) T(s)^k,
\]
where $A_k(s)$ is an infinite series arising from the contributions
coming from the function $\Po{W}_{k,G}(z)$, and the fundamental strip of 
$\Po{G}_k(z)$ contains
$(-k-1, \infty)$. It involves  unknown 
$\mu_{m,j} - \mu_{m,j-1}$ for various $m$ and $j$ 
(see  \cite{magnerPhD2015,magnerknesslszpa2014}), that is:
\begin{eqnarray}
    \label{A_kFormula}
    A_k(s)
    = \sum_{j=0}^k T(s)^{-j} \sum_{m \geq j} T(-m)(\mu_{m,j} - \mu_{m,j-1})\frac{\Gamma(m+s)}{\Gamma(s+1)\Gamma(m+1)}.
\end{eqnarray}
Locating and characterizing the singularities
of $\Me{G_k}(s)$ then becomes important. In \cite{magnerspa2015} it is
shown that   for any $k$, $A_k(s)$ is entire, with zeros at 
$s \in \Z \lintersect [-k, -1]$, so that $\Me{G_k}(s)$ is meromorphic,
with possible simple poles at the negative integers less than $-k$.  
The fundamental strip of $\Po{G}_k(z)$ then contains $(-k-1, \infty)$.

We then must asymptotically invert the Mellin transform to recover $\Po{G}_k(z)$.  The Mellin inversion
formula for $\Me{G_k}(s)$ is given by
\begin{eqnarray}
    \label{MellinInversionFormula}
    \Po{G}_k(z) = \frac{1}{2\pi i} \int_{\rho - i\infty}^{\rho + i\infty} 
z^{-s}\Me{G_k}(s)\dee{s}
    = \frac{1}{2\pi i} \int_{\rho - i\infty}^{\rho + i\infty} 
z^{-s}\Gamma(s+1)A_k(s)T(s)^k\dee{s},
\end{eqnarray}
where $\rho$ is any real number inside the fundamental strip 
associated with $\Po{G}_k(z)$.  

\subsection{Main results via extension of the analysis of $B_{n,k}$}
\label{MainResultsSection}
Having explained the relevant functional equations and the integral representation (\ref{MellinInversionFormula})
for $\Po{G}_k(z)$, we now move on to describe the main results of this paper. 
For Theorem~\ref{HeightTheorem} and \ref{FillupTheorem} we start with a sketch of the derivation
whereas the proof of Theorem~\ref{DepthCorollary} is given immediately. 
The complete proof of Theorem~\ref{HeightTheorem} and a roadmap for Theorem~\ref{FillupTheorem}, both for the case $p>q$,
 is given in Section~\ref{Proofs}.
 
\subsubsection{Result on $H_n$}
\label{HeightSketch}
Our first aim is to derive two-term expansions for the typical values of $H_n$ and $F_n$.  To do this for, e.g., $H_n$,
we define, for $p \geq q$,
\[
    k_* = \log_{1/p} n + \psi_*(n),
\]
where $\psi_*(n) = o(\log n)$ is a function to be determined.  We also define
\begin{eqnarray}
    \label{PsiLDefinition}
    \psi_L(n) = (1-\epsilon)\psi_*(n) &&        k_L = \log_{1/p} n + \psi_L(n) \\
    \psi_U(n) = (1+\epsilon)\psi_*(n) &&        k_U = \log_{1/p} n + \psi_U(n),
    \label{PsiUDefinition}
\end{eqnarray}
for arbitrarily small $\epsilon > 0$.
We require that $\psi_*(n)$ be such that
\begin{eqnarray}
    \label{PhaseTransitionExpr}
    \E[B_{n,k_L}] \to \infty,       &&
    \E[B_{n,k_U}] \to 0,
\end{eqnarray}
and a proper upper bound for $\Var[B_{n,k_L}]$ 
(see Lemma~\ref{VarianceLemma}). However, in order to make the 
following pre-analysis more transparent we will not dwell on the variance.

To determine a candidate for $\psi_*(n)$, we start with the inverse 
Mellin integral
representation for $\Po{G}_{k_*}(n)$:
\begin{eqnarray}
    \Po{G}_{k_*}(n)
    = \frac{1}{2\pi i} \int_{\rho - i\infty}^{\rho + i\infty} J_{k_*}(n, s)\dee{s},
    \label{InverseMellinIntegral}
\end{eqnarray}
where we define
\begin{eqnarray}
\nonumber
    J_k(n, s)
    &=& n^{-s}T(s)^k\Gamma(s+1)A_k(s) \\
    &=& \sum_{j=0}^k n^{-s}T(s)^{k-j}\sum_{m \geq j} T(-m)(\mu_{m,j} - \mu_{m,j-1
})\frac{\Gamma(m+s)}{\Gamma(m+1)}.
\label{eq-j}
\end{eqnarray}
Note that by depoissonization (see Section~\ref{sec:depo} and \cite{jacquetszpa1997}) 
we have
\[
\mu_{n,k_*} = \Po{G}_{k_*}(n) - \frac n2 \Po{G}''_{k_*}(n) + O(n^{-1+\epsilon}). 
\] 
Indeed, because of the exponential decay of $A_k(s)\Gamma(s+1)$ 
along vertical lines, the entire integral
is at most of the same order as the integrand on the real axis (we justify this more carefully in Section~\ref{HProof}). Furthermore, since the second derivative has an additional factor
$s(s+1)n^{-2}$ in the integrand we will get a similar bound for $\frac n2 \Po{G}''_{k_*}(n)$ which
is just $\rho^2/n$ times the corresponding bound for $\Po{G}_{k_*}(n)$ and, thus, negligible 
in comparison to $\Po{G}_{k_*}(n)$.

In this proof roadmap we focus on estimating the
integrand $J_{k_*}(n, \rho)$, $\rho \in \R$, as precisely as possible.  
Using Lemma~\ref{KnesslLemma}, we find (see (\ref{eqfirstesti}) 
in Section~\ref{HProof}) that the $j$th term in the representation (\ref{eq-j}) of $J_{k_*}(n,\rho)$ is of 
order 
\begin{eqnarray}\label{eqjthterm}
n^{-\rho} T(\rho)^{k_*-j} p^{j^2/2 + O(j\log j)},
\end{eqnarray}
where $\rho < 0$
and $T(\rho) = p^{-\rho} + q^{-\rho}$. Hence, by setting $j_0 = -\log_{1/p} T(\rho)$
we have
\begin{eqnarray}\label{eqJk*bound}
J_{k_*}(n, \rho) = O\left( n^{-\rho} T(\rho)^{k_*} p^{-j_0^2/2 + O(j_0\log j_0)} \right).
\end{eqnarray}
Next we have to choose $\rho \in \R_-$ that minimizes this upper bound.
Here we distinguish between the symmetric case $p = q = 1/2$ and the case $p> q$.

In the symmetric case we have $T(\rho) = 2^{\rho +1}$ and $j_0 = -\rho-1$ and, thus,
\[
J_{k_*}(n, \rho) = O\left( n^{-\rho} 2^{(\rho+1)(\log_2 n + 
\psi_*(n))+\rho^2/2 + O(|\rho|\log |\rho|)} \right).
\] 
Consequently by disregarding the error term $O(|\rho|\log |\rho|)$ the optimal choice 
of $\rho$ is $\rho = -\psi_*(n)$ which gives the upper bound
\[
J_{k_*}(n, \rho) = O\left( 2^{\log_2 n - \psi_*(n)^2/2 + O(|\psi_*(n)|\log |\psi_*(n)|)} \right).
\] 
Hence, the threshold for this upper bound is $\psi_*(n) = \sqrt{2\log_2 n}$. In particular it also follows that
\[
J_{k_U}(n, \rho) = O\left( 2^{-(2\epsilon+\epsilon^2)\log_2 n  + O(\sqrt{\log n} \log\log n )} \right),
\] 
where $k_U = \log_{1/p} n + (1+\epsilon)\sqrt{2\log_2 n}$. We also note that we get the same bound
if $\rho = -\psi_*(n) + O(1)$. 

In the case $p> q$ we have to be slightly more careful. Nevertheless we can start with the
upper bound (\ref{eqJk*bound}) and obtain
\[
J_{k_*}(n, \rho) = O\left( p^{ ( \rho- \log_{1/p} T(\rho))\log_{1/p}n - \psi_*(n) \log_{1/p} T(\rho) - (\log_{1/p} T(\rho))^2/2 
  + O(j_0\log j_0)  }   \right).
\]
 From the representation 
$T(\rho) = p^{-\rho}( 1+ (p/q)^\rho)$ we obtain
\[
\log_{1/p} T(\rho) = \rho + \frac{ (p/q)^\rho }{\log(1/p)} + O\left(  (p/q)^{2\rho} \right).
\]
It is clear that we have to choose $\rho < 0$ that tends to $-\infty$ if $n\to\infty$.
Hence, $\log_{1/p} T(\rho) = \rho + o(1)$ and consequently a proper choice for $\rho$ is
the solution of the equation
\begin{eqnarray*}
    \frac{\partial}{\partial \rho} \left(
    - \frac{ (p/q)^\rho }{\log(1/p)} \log_{1/p}n - \psi_*(n) \rho  -  
    \frac{\rho^2}2 \right) 
    = \frac{ (p/q)^\rho \log(p/q)}{\log(1/p)} \log_{1/p}n- \psi_*(n) -\rho = 0.
\end{eqnarray*}
Actually this gives $\rho < - \psi_*(n)$ and, thus, 
\[
\rho = - \log_{p/q}\log n + O(\log\log\log n).
\]
With this choice the upper bound for $J_{k_*}(n, \rho)$ writes as 
\begin{eqnarray}
J_{k_*}(n, \rho) = O\left( p^{ (\psi_*(n)+\rho)/\log(p/q) - \psi_*(n) \rho - \frac{\rho^2}2 
  + O(j_0\log j_0)  }   \right) 
= O\left( p^{ - \psi_*(n) \rho - \frac{\rho^2}2 
  + O(j_0\log j_0)  }   \right).
  \label{MyEqn}
\end{eqnarray}
This implies that the threshold for this upper bound is given by
\[
\psi_*(n) = -\frac \rho 2 = \frac 12 \log_{p/q} \log n + O(\log\log\log n).
\]
In particular, if we replace $\psi_*(n)$ by 
$\psi_U(n) = \frac 12(1+\epsilon)  \log_{p/q} \log n$ we obtain
\begin{eqnarray} \label{HeightMaxContribution0}
J_{k_U}(n, \rho) = O\left( p^{ \epsilon(\log_{p/q} \log n)^2/2 + 
O(\log\log n \log\log \log n) }   \right)
\end{eqnarray}
and for 
$\psi_L(n) = \left(1 -\epsilon\right) \frac 12 \log_{p/q} \log n$,
\begin{eqnarray} \label{HeightMaxContribution}
J_{k_L}(n, \rho) = O\left( p^{ -\epsilon(\log_{p/q} \log n)^2/2 + O(\log\log n \log\log \log n }   \right).
\end{eqnarray}

The above pre-analysis suggests asymptotic estimates for $\Po{G}_k(n)$ and, thus, by 
depoissonization estimates for $\mu_{n,k}$, which imply a two-term expansion for $H_n$.
The complete proof of this result is given in Section~\ref{HProof}.
In summary, we formulate below our first main result.

\begin{theorem}[Asymptotics for $H_n$]
    \label{HeightTheorem}
    With high probability,
    \begin{eqnarray*}
        H_n = \left\{ \begin{array}{ll}
                \log_{1/p} n + \frac 12\log_{p/q}\log n + o(\log\log n) &
                p > q \\
                \log_{2} n + \sqrt{2\log_2 n} + o(\sqrt{\log n}) &
                p = q
        \end{array}   \right.
    \end{eqnarray*}
for large $n$.
\end{theorem}

\subsubsection{Result on $F_n$}
We take a similar approach for the derivation of $F_n$, with some differences.  We set
\begin{eqnarray*}
    k_* = \log_{1/q} n + \phi_*(n)
\end{eqnarray*}
with
\begin{eqnarray*}
    \phi_L(n) = (1+\epsilon)\phi_*(n),      &&
    \phi_U(n) = (1-\epsilon)\phi_*(n), 
\end{eqnarray*}
and $k_L$ and $k_U$, respectively, defined with $\phi_L$ (respectively, $\phi_U$) in place of $\phi_*$.
The derivation of an estimate for the $j$th term of $J_{k_*}(n, \rho)$, $\rho\in\R$, is similar to that in Section~\ref{HeightSketch},
except now the asymptotics of $\Gamma(\rho+1)$ play a role (this is reflected in the proof, where $\Gamma(\rho+1)$ determines
the location of the saddle points of the integrand).  We find that the $j$th term is at most $q^{\lambda_j(n, \rho)}$,
where
\begin{eqnarray}
    \lambda_j(n, \rho) = \rho(j - \phi_*(n)) + (j - \phi_*(n) - \log_{1/q} n)\log_{1/q}(1 + (q/p)^{\rho}) - \rho\log_{1/q} \rho + O(\rho).
    \label{OptimalNu}
\end{eqnarray}
Optimizing over $j$ gives $j = 0$.
The behavior with respect to $\rho$ depends on whether or not $p=q$, because $\log_{1/q}(1 + (q/p)^\rho) = 1$ when $p=q$
and is dependent on $\rho$ otherwise.  Taking this into account and minimizing over all $\rho$ gives an optimal value of
\begin{eqnarray*}
     \rho =
     \left\{\begin{array}{ll}
        2^{-\phi_*(n) - 1/\log 2}       &  p=q=1/2, \\
        \log_{p/q}\log n                & p > 1/2.
     \end{array} \right.
\end{eqnarray*}
Note that this corresponds to the real part of the saddle points in the proof.
Plugging this into (\ref{OptimalNu}), setting the resulting expression equal to $0$, and solving for $\phi_*(n)$ gives
\begin{eqnarray*}
    \phi_*(n) =
    \left\{\begin{array}{ll}
        -\log_2\log n + O(1)    & p=q=1/2 \\
        -\log_{1/q}\log\log n   & p > 1/2.
    \end{array} \right.
\end{eqnarray*}

This heuristic derivation suggests that  the following theorem holds. 
More details are given in Section~\ref{FProof}.
\begin{theorem}[Asymptotics for $F_n$]
    \label{FillupTheorem}
    With high probability,
    \begin{eqnarray*}
        F_n =
        \left\{\begin{array}{ll}
            \log_{1/q} n - \log_{1/q}\log\log n + o(\log\log\log n)  &
            p > q \\
            \log_{2} n - \log_2\log n + o(\log\log n) &
            p = q 
        \end{array} \right. 
    \end{eqnarray*}
for large $n$.
\end{theorem}
\subsection{Result on $D_n$}
We move to  our results concerning $D_n$.  To state them, we first need to observe that there is a natural
way to define the sequence $\{D_n\}_{n\geq 0}$ on a single probability space, so that we may ask whether
or not $D_n$, properly normalized, converges almost surely, and to what limiting value.  This common space
is defined by appealing to the correspondence between the sequence of R\'enyi problem queries and the growth
of a random PATRICIA trie.  For each $n \geq 0$, we define a tree $T_n$ which is a PATRICIA trie constructed
on $n$ strings (equivalently, a terminating sequence of R\'enyi queries recovering a bijection between two sets
of $n$ elements): $T_0$ is an empty tree, and $T_{n+1}$ is constructed from $T_n$ by generating an independent
string of i.i.d. $\Bernoulli(p)$ random variables and inserting this string into $T_n$.  Then, for each $n$,
$D_n$ is the depth of a leaf chosen uniformly at random (and independent of everything else) from the leaves
of $T_{n}$.

With this construction in mind, we have the following result about the 
convergence of $D_n$.  Its proof combines known facts about the profile 
with the new ones proved here, as well as
a proof technique that was used before in, e.g., \cite{Pittel85}.

\begin{theorem}[Asymptotics of $D_n$]
    \label{DepthCorollary}
    For $p > 1/2$, the normalized depth $D_n/\log n$ converges in probability to $1/h(p)$ 
    where $h(p) := -p\log p - q\log q$ is the entropy of a $\Bernoulli(p)$ random variable, but not
    almost surely. In fact,  
    \begin{eqnarray}
        \label{D_nNoAlmostSure}
        \liminf_{n\to\infty} D_n/\log n = 1/\log(1/q), &&
        \limsup_{n\to\infty} D_n/\log n = 1/\log(1/p)
    \end{eqnarray}
almost surely.
\end{theorem}

\begin{proof}
The fact that $D_n/\log n$ converges in probability to $1/h(p)$ follows directly from the 
central limit theorem for $D_n$ given in \cite{szpa2001Book}.

Next we show that (\ref{D_nNoAlmostSure}) holds.
Clearly $F_n \le D_n \le H_n$. Now let us consider the following
    sequences of events: $A_n$ is the event that $D_n = F_n+1$, and $A'_n$ is the event that $D_n = H_n$.
    We note that all elements of the sequences are independent, and $\Pr[A_n] \ge 1/n$, $\Pr[A'_n] \geq 1/n$.  This
    implies that $\sum_{n=1}^\infty \Pr[A_n] = \sum_{n=1}^\infty \Pr[A'_n] = \infty$, so that the Borel-Cantelli
    lemma tells us that both $A_n$ and $A'_n$ occur infinitely often almost surely.  

In the next step we show that, almost surely, $F_n/\log n \to 1/\log(1/q)$
    and $H_n/\log n \to 1/\log(1/p)$. Then (\ref{D_nNoAlmostSure}) is proved.
    We cannot apply the Borel-Cantelli lemmas directly, because the relevant sums do not converge.  Instead, we apply the following trick: 
    we observe that both $(F_n)$ and $(H_n)$ are non-decreasing sequences.  Next, we show that, on some appropriately chosen
    subsequence, both of these sequences, when divided by $\log n$, converge almost surely to their respective
    limits.  Combining this with the observed monotonicity yields the claimed almost sure convergence, and,
    hence, the equalities in (\ref{D_nNoAlmostSure}).

    We illustrate this idea more precisely for $H_n$.  By our analysis above, we know that
    \[
        \Pr[|H_n/\log n - 1/\log(1/p)| > \epsilon] = O(e^{-\Theta(\log\log n)^2}).
    \]
    Then we fix $t$, and we define $n_{r,t} = 2^{t^2 2^{2r}}$.  
    On this subsequence, by the probability
    bound just stated, we can apply the Borel-Cantelli lemma to conclude 
    that $H_{n_{r,t}}/\log(n_{r,t}) \to 1/\log(1/p) \cdot (t+1)^2/t^2$ 
    almost surely.  Moreover, for every $n$, we can choose $r$ such 
    that $n_{r,t} \leq n \leq n_{r,t+1}$.
    Then
    \[
        H_n/\log n \leq H_{n_{r,t+1}}/\log n_{r,t},
    \]
    which implies
    \[
        \limsup_{n\to\infty} \frac{H_n}{\log n} \leq 
        \limsup_{r\to\infty} \frac{H_{n_{r,t+1}}}{\log n_{r,t+1}} 
        \frac{\log n_{r,t+1}}{\log n_{r,t}}
        = \frac{1}{\log(1/p)} \cdot \frac{(t+1)^2}{t^2}.
    \]
    Taking $t \to \infty$, this becomes $1/\log(1/p)$, as desired.  
    The argument for the $\liminf$ is similar,
    and this establishes the almost sure convergence of $H_n$.  
    The derivation is entirely similar for $F_n$.
  \end{proof}

\section{Proof of Theorems~\ref{HeightTheorem}  and  \ref{FillupTheorem}}
\label{Proofs}

We give a detailed proof of Theorem~\ref{HeightTheorem} and indicate the
main lines of the proof of Theorem~\ref{FillupTheorem}. 
We also concentrate just on the case $p>q$.
The proof of the symmetric case can be done by the same techniques 
(properly adapted) but it just reproves the result by Pittel and Rubin \cite{pittelrubin1990}.

\subsection{Proof of Theorem~\ref{HeightTheorem}}
\label{HProof}

\subsubsection{A-Priori Bounds for $\mu_{n,k}$}

For the analysis of the profile around the height level, we need precise information
about $\mu_{n,k}$ with $n\to\infty$ when $k$ close to $n$.  This is captured in the following lemma, which
first appeared in a similar form in \cite{magnerknesslszpa2014}.

We consider $\mu_{n,k}$ where $k$ is close to $n$, so we set $k = n-\ell$ 
and represent it as
\[
\mu_{n,k} = \mu_{n,n-\ell} = n! C_*(p) p^{(n-\ell)(n-\ell+1)^2/2}q^{n-\ell} \xi_{\ell}(n),
\]
where  
\[
        C_*(p) = \prod_{j=2}^\infty (1-p^j - q^j)^{-1} \cdot (1 + (q/p)^{j-2}),
\]
$\xi_1(1) = 1/C_*(p)$ and for $n > \ell \ge 1$
\begin{equation}\label{eqxirec}
        \xi_{\ell}(n)(1-p^n-q^n) = 
     \sum_{J=1}^\ell \frac{\xi_{\ell+1-J}(n-J)}{J!}q^{-1}p^{\ell-n}(p^{n-J}q^{J} + p^{J}q^{n-J}).
\end{equation}
Note that $\xi_\ell(n) = 0$ for $n\le \ell$.  The above formulas were first derived in \cite{magnerknesslszpa2014}.

\begin{lemma}[Asymptotics for $\mu_{n,k}$, $k\to\infty$ and $n$ near $k$]
    \label{KnesslLemma}
    \label{KNESSLLEMMA}         
    \label{mu_mjUpperBoundLemma}
    Let $p \geq q$.  
\parsec    
{\rm (i)} {\rm Precise estimate: }
    For every fixed $\ell \ge 1$ and $n\to\infty$ 
    \begin{eqnarray*}
        \mu_{n,n-\ell} \sim n! C_*(p) p^{(n-\ell)^2/2 + (n-\ell)/2}q^{n-\ell} \xi_{\ell},
    \end{eqnarray*}
    where the sequence $\xi_\ell$, $\ell\ge 1$ satisfies the recurrence
    \begin{equation}\label{eqexrec2}
      \xi_\ell
        = q^{-1}p^{\ell} \sum_{J=1}^\ell \frac{\xi_{\ell+1-J}}{J!} (q/p)^{J}
     \end{equation}
    with $\xi_1 = 1$. Furthermore we have (for some positive constant $C$)
 \begin{equation}\label{eqxiapprox}
|\xi_{\ell+1-J}(n-J)-\xi_{\ell+1-J}| \le C(p^{n-\ell-1} + (q/p)^{n-\ell-1})/(\ell-J)!,
\end{equation}
\parsec
{\rm (ii)} {\rm Upper bound: }
    We have $\xi_\ell(n) \le C_1/(\ell-1)!$ for some constant $C_1$ and, thus, for $1\le k< n$
    (and some constant $C$)
    \begin{eqnarray}
        \mu_{n,k}
        \leq C\frac{n!}{(n-k-1)!}p^{(k^2 + k)/2 }q^{k}.
        \label{mu_mjUpperBound}
    \end{eqnarray}
\end{lemma}

\begin{proof}
    From the recurrence (\ref{eqxirec}) it follows easily that for each 
$\ell\ge 1$ the limit $ \xi_\ell = \lim_{n\to\infty} \xi_\ell(n)$  
exists by (\ref{mu_mjUpperBound}), and in particular 
for $\ell =1$ we have $\xi_1 = 1$. 
   Clearly this limits satisfy the recurrence (\ref{eqexrec2}).

Next we show by induction a uniform upper bound of the form 
$\xi_\ell(n) \le C_1/(\ell-1) !$
The induction step for $n>\ell > \ell_1$ runs as follows (where $C_1$ and  $\ell_1$ is appropriately chosen such
that the upper bound is true for $\ell \le \ell_1$ and that 
$2/(q \ell_1(1-p^{\ell_1}-q^{\ell_1}) \le 1$):
\begin{eqnarray*}
\xi_\ell(n) &\le& \frac{C_1}{1-p^n-q^n} \left( \sum_{J=1}^{\ell} \frac{p^{\ell-J}q^{J-1}}{J!(\ell-J)!} + 
\sum_{J=1}^{\ell} \frac{p^{\ell+J-n}q^{n-J-1}}{J!(\ell-J)!}  \right) \\
&\le& \frac{C_1}{\ell!(1-p^n-q^n)} 
\left( \frac 1{q} \sum_{J=0}^{\ell} {\ell \choose J} p^{\ell-J}q^{J} + 
\frac{(q/p)^{n-\ell}}{q} \sum_{J=0}^{\ell}   {\ell\choose J} p^{J}q^{\ell-J}  \right) \\
&\le& \frac{C_1}{(\ell-1)!} \frac 1{\ell_1(1-p^{\ell_1}-q^{\ell_1})} 
\frac 2q 
\le \frac{C_1}{(\ell-1)!}.
\end{eqnarray*}
In a similar way we obtain the approximation estimate (\ref{eqxiapprox}). We give a full proof in Section~\ref{KnesslLemmaProof}.
\end{proof}

\subsubsection{Upper bound on $H_n$}
Now we set 
\begin{eqnarray}\label{eqkUdef}
k = k_U = \log_{1/p} n + \psi_U(n) = \log_{1/p} n + 
\frac 12\left(1 +\epsilon\right)\log_{p/q}\log n
\end{eqnarray}
just as in (\ref{PsiUDefinition}).
We will first estimate the value of $J_k(n, s)$ (which is defined in (\ref{eq-j}))
for $s = \rho' = - 2\psi(n) + O(1)\in \Z^{-} - 1/2$ (i.e., the set $\{ -3/2, -5/2, ... \}$),
as hinted at in Section~\ref{MainResults}.

\begin{lemma}\label{LeJest}
Suppose that $p>q$, that $\epsilon> 0$, that $k_U$ is given by (\ref{eqkUdef}), and 
that $\rho' = \lfloor \rho \rfloor + \frac 12$, where $\rho = - \log_{p/q} \log n + O(\log\log \log n)$ 
is the solution of the equation
\[
 \frac{ (p/q)^\rho \log(p/q)}{\log(1/p)} \log_{1/p}n+ \psi_U(n) +\rho = 0.
\]
Then we have for $k \ge k_U$
\begin{eqnarray}
    J_{k}(n, \rho') = O\left( T(\rho')^{k-k_U} p^{ \epsilon 
(\log_{p/q} \log n)^2/2 + O(\log\log n \cdot \log\log \log n) }   \right).
    \label{anm1}
\end{eqnarray}
\end{lemma}

\begin{proof}
First we observe that the assumption $\rho' \in \Z^{-} - 1/2$  with $|\rho'| \to \infty$ assures that
for all $m\ge 0$ we have $\left| \Gamma(m+\rho')/\Gamma(m+1) \right| \leq 1$.
Next by (\ref{mu_mjUpperBound}) of Lemma~\ref{mu_mjUpperBoundLemma}
we have $\mu_{m,j} = O\left( m^{j+1} p^{j^2/2 + O(j)}\right)$
which implies that 
\[
\sum_{m\ge j} T(-m) \mu_{m,j} = O\left( p^{j^2/2 + O(j\log j)} \right).
\]
Hence, the $j$th term in the representation  (\ref{eq-j}) of  $J_k(n, \rho')$ can be estimated by
\begin{eqnarray}
&& \left| n^{-\rho'}T(\rho')^{k-j} 
\sum_{m \geq j} T(-m)(\mu_{m,j} - \mu_{m,j-1})
\frac{\Gamma(m+\rho')}{\Gamma(m+1)}\right| \label{eqfirstesti}\\
&&\le n^{-\rho'}T(\rho')^{k-j} 
\sum_{m \geq j} T(-m)(\mu_{m,j} +\mu_{m,j-1}) 
= O\left(n^{-\rho'}T(\rho')^{k-j}  p^{j^2/2 + O(j\log j)}   \right).  \nonumber 
\end{eqnarray}
Thus, we have shown (\ref{eqJk*bound}) which implies (\ref{anm1}) 
for $k = k_U$ (see (\ref{HeightMaxContribution0})).
However, it is easy to extend it to larger $k$ (since equation (\ref{MyEqn})
holds for generic $k_* = k$ and the given choice of $\rho$).
Actually we get uniformly for $k\ge k_U$
\[
J_k(n, \rho') = O\left( T(\rho')^{k-k_U} 
p^{ \epsilon (\log_{p/q} \log n)^2/2 + O(\log\log n \log\log \log n) }   \right)
\]
for large $n$.
\end{proof}


Our next goal is to evaluate the integral (\ref{InverseMellinIntegral})
and to obtain a bound for $\mu_{n,k}$.
\begin{lemma}\label{LeJest2}
Suppose that $p>q$, that $\epsilon> 0$, and that $k_U$ and $\rho'$ 
are given as in Lemma~\ref{LeJest}.
Then we have (for some $\delta> 0$)
\begin{equation}
\mu_{n,k} = 
  O\left( T(\rho')^{k-k_U} p^{ \epsilon (\log_{p/q} \log n)^2/2 + O((\log\log n)^{1-\delta}) } \right)
  + O(n^{-1+\epsilon})
    \label{mu_nkUBound}
\end{equation}
uniformly for $k \ge k_U$.
\end{lemma}

\begin{proof}
Letting $\Co$ denote the vertical line $\Re(s) = \rho'$, we 
evaluate the integral (\ref{InverseMellinIntegral}) 
by splitting it into an inner region $\Co^I$
and outer tails $\Co^O$:
\[
    \Co^I = \{\rho' + it :~ |t| \leq e^{(\log\log n)^{2-\delta}}\}, \quad
    \Co^O = \{\rho' + it :~ |t| > e^{(\log\log n)^{2-\delta}}\},
\]
where $0<\delta < 1$ is some fixed real number.  The inner region we evaluate by showing that it is of the same order as
the integrand on the real axis, and the outer tails are shown to be negligible by the exponential decay of the $\Gamma$ function.

It is easily checked that 
$|n^{-s}T(s)^{k-j}\Gamma(m+s)| \leq n^{-\rho'}T(\rho')^{k-j} |\Gamma(m+\rho')|$
when $\Re(s) = \rho'$ (and any value for $\Im(s)$).  Thus,
\begin{eqnarray*}
    |J_k(n, s)|
    \leq  T(\rho')^{k-k_U}  \sum_{j=0}^k n^{-\rho'}T(\rho')^{k_U-j}\sum_{m\geq j} T(-m)|\mu_{m,j} - \mu_{m,j-1}| \frac{|\Gamma(m+\rho')|}{\Gamma(m+1)},
\end{eqnarray*}
which can be upper bounded as in the proof of Lemma~\ref{LeJest}.  Multiplying by the length of the contour, we  find
\begin{eqnarray*}
    \left|\int_{\Co^I} J_k(n, s)\dee{s}\right|
    = O\left( T(\rho')^{k-k_U} p^{ \epsilon (\log_{p/q} \log n)^2/2 + O((\log\log n)^{2-\delta}) } \right).
\end{eqnarray*}


We use the following standard bound on the $\Gamma$ function: for $s = \rho' + it$, provided that $|\Arg(s)|$ is less
than and bounded away from $\pi$ and $|s|$ is sufficiently large, we have
\[
    |\Gamma(s)| 
    \leq C|t|^{\rho' - 1/2} e^{-\pi|t|/2}.
\]
This is applicable on $\Co^O$, and we again use the fact that $|T(s)| \leq T(\rho')$ and $|\mu_{m,j} - \mu_{m,j-1}| \leq m$,
which yields an upper bound of the form
\begin{eqnarray*}
\left| \sum_{m \geq j} T(-m)(\mu_{m,j} - \mu_{m,j-1})  \frac{\Gamma(m+s)}{\Gamma(m+1)}  \right| 
& =& O\left( \sum_{m \geq j} T(-m) m \frac{|t|^{m+\rho' - 1/2}e^{-\pi|t|/2})}{\Gamma(m+1)} \right) \\
&=& O\left( p|t|^{\rho' + 1/2}e^{-\pi|t|/2} e^{p|t|}   \right),
\end{eqnarray*}
where we have used the inequality
\begin{eqnarray*}
    |t|^{\rho' - 1/2}e^{-\pi|t|/2} \sum_{m \geq j} \frac{m(p|t|)^{m}}{m!}
    \leq p|t|^{\rho' + 1/2}e^{-\pi|t|/2} e^{p|t|} = e^{-\Theta(|t|)},
\end{eqnarray*}
uniformly in $j$, by our choice of $|t|$.

Furthermore, since $T(\rho')< 1$ we have
\[
	\sum_{j=0}^k n^{-\rho'}T(\rho')^{k-j} = O(n^{-\rho'}) 
    = e^{O(\log n \log\log n)}.
\]
Hence, integrating this on $\Co^O$ gives
\begin{eqnarray*}
    \left| \int_{\Co^O} J_k(n, s) \dee{s} \right|
    &=& O\left( T(\rho')^{k-k_U}  e^{O(\log n \log \log n)}  e^{-\Theta( e^{(\log\log n)^{2-\delta}})} \right)\\
& =&  O\left( T(\rho')^{k-k_U} e^{-\Theta( e^{(\log\log n)^{2-\delta}})} \right).
\end{eqnarray*}
Adding these together gives
\begin{eqnarray*}
    \Po{G}_{k}(n)
    &\leq &\left| \int_{\Co^I} J_{k_U}(n, s)\dee{s} + \int_{\Co^O} J_{k_U}(n, s)\dee{s} \right| \\
    & = & O\left( T(\rho')^{k-k_U} p^{ \epsilon (\log_{p/q} \log n)^2/2 + O((\log\log n)^{2-\delta}) } \right).
\end{eqnarray*}
Similarly we get a bound for $\Po{G''}_{k}(n)$:
\begin{eqnarray*}
    \Po{G''}_{k}(n) = 
     O\left( {\rho'}^2 T(\rho')^{k-k_U} p^{ \epsilon (\log_{p/q} \log n)^2/2 + O((\log\log n)^{2-\delta}) } \right).
\end{eqnarray*}
Hence by depoissonization (see (\ref{eqDep3.3}) from Section~\ref{sec:depo}) we get
\[
\mu_{n,k} = O\left( T(\rho')^{k-k_U} p^{ \epsilon (\log_{p/q} \log n)^2/2 + O((\log\log n)^{2-\delta}) } \right)
+ O(n^{-1+\epsilon})
\]
as needed.
\end{proof}

Our original goal was to bound the tail $\Pr[H_n > k_U]$ by the following sum
which we split into two parts:
\[
    \Pr[H_n > k_U] \le  \sum_{k \geq k_U} \mu_{n,k}
    = \sum_{k = k_U}^{\ceil{(\log n)^2}} \mu_{n,k} + \sum_{k=\ceil{(\log n)^2} + 1}^n \mu_{n,k}.
\]
The initial part can be bounded using (\ref{mu_nkUBound}), and the final part we handle using (\ref{mu_mjUpperBound}) in 
Lemma~\ref{mu_mjUpperBoundLemma}.     
Indeed, since $T(\rho') < 1$ the first sum can be bounded by
\[
    \sum_{k=k_U}^{\ceil{(\log n)^2}}\mu_{n,k}
    \leq
    e^{-\Theta(\epsilon(\log\log n)^2)} .
\]
The second sum is at most
\[
    \sum_{k=\ceil{(\log n)^2}+1}^{n}  \mu_{n,k} 
    \leq ne^{-\Theta((\log n)^4)} \\
    =  e^{-\Theta((\log n)^4)}.
\]
Adding these upper bounds together shows that
$
    \Pr[H_n > k_U] = e^{-\Theta(\epsilon(\log\log n)^2)} \to 0,
$
as desired.

\subsubsection{Upper bound on the variance of the profile}

We consider now the case
\begin{equation}\label{eqkLdef}
k = k_L = \log_{1/p} n + \psi_L(n) = \log_{1/p} n + \psi(n), \ \ \ \
\psi(n)=\frac 12 \left(1 -\epsilon \right) \log_{p/q}\log n
\end{equation}
and start with an upper bound for the variance of the profile $\Var[B_{n,k}]$.
\begin{lemma}\label{VarianceLemma}
Suppose that $p>q$, that $\epsilon> 0$, and that $k_L$ is given by (\ref{eqkLdef}).
Then we have 
\begin{eqnarray}
\Var[B_{n,k}] &= 
  O\left( p^{-\epsilon (\log_{p/q} \log n)^2/2 + O((\log\log n)^{2-\delta}) } \right).
    \label{VarB_nkUBound}
\end{eqnarray}
\end{lemma}

\begin{proof}
The proof technique here is the same as for the proof of the upper bound on $\mu_{n,k}$.
    Our goal is to upper bound the expression
    \[
        \Po{V}_k(n) = \sum_{n\ge 0} \mathbb{E}[B_{n,k}^2] \frac{n^n}{n!}e^{-n} - \tilde G_k(n)^2 
        = \frac{1}{2\pi i} \int_{\rho' - i\infty}^{\rho' + i\infty} J^{(V)}_k(n, s)\dee{s},
    \]
    where 
    \begin{eqnarray*}
        J^{(V)}_k(n, s)
        = n^{-s}T(s)^k\Gamma(s+1)B_k(s), 
    \end{eqnarray*}
    and
    \begin{eqnarray*}
        B_k(s)
        = 1 - (s+1) 2^{-(s+2)} + \sum_{j=1}^k T(s)^{-j} \frac{\Me{W_{j,V}}(s)}{\Gamma(s+1)},
    \end{eqnarray*}
    with \cite{magnerPhD2015}
    \begin{eqnarray*}
        \Me{W_{j,V}}(s)
        &=& \sum_{m \geq j} \frac{\Gamma(m+s)}{m!} \left[  \right.
        \left.T(-m) ( c_{m,j} - c_{m,j-1} + \mu_{m,j} - \mu_{m,j-1}) \right.\\
            &&\left. + T(s)2^{-(s+m)} \sum_{\ell=0}^m \mu_{\ell,j-1}\mu_{m-\ell,j-1} \right. \\
            &&\left.+ 2\sum_{\ell=0}^m \mu_{\ell,j-1} \mu_{m-\ell,j-1} p^\ell q^{m-\ell}
            - 2^{-(m+s)} \sum_{\ell=0}^m \mu_{\ell,j}\mu_{m-\ell,j}
        \right] .
    \end{eqnarray*}
    As above we need a bound on the moments of $B_{m,j}$ for $m$ sufficiently
    close to $j$: for $\mu_{m,j} = \E[B_{m,j}]$, this is (\ref{mu_mjUpperBound}) in Lemma~\ref{mu_mjUpperBoundLemma}.  It turns out that $c_{m,j} = \E[B_{m,j}(B_{m,j}-1)]$ 
    satisfies a similar recurrence as $\mu_{m,j}$ (see \cite{magnerspa2015}) 
and also similar inequality: for $j\to\infty$ and $m> j$,
    \begin{eqnarray*}
        c_{m,j} \leq \frac{m!}{(m-j-1)!} p^{j^2/2 + O(j)}.
    \end{eqnarray*}
    The proof is by induction and follows along the same lines as that of the upper bound in Lemma~\ref{mu_mjUpperBoundLemma}.
    Using this, we can upper bound the inverse Mellin integral as in the upper bound for $\tilde G_k(n)$.
In particular it follows that 
\[
\tilde V_{k_L}(n) = O\left(  p^{ -\epsilon(\log_{p/q} \log n)^2/2 + O((\log\log n)^{2-\delta}) }       \right)
\]
and similarly we have 
\[
\tilde V_{k_L}''(n) = O\left( {\rho'}^2 n^{-2} p^{ -\epsilon(\log_{p/q} \log n)^2 + O((\log\log n)^{2-\delta})  }  \right),
\]
where $\rho'=-\log_{p/q}\log n + O(\log\log\log n)$.
With the help of depoissonization, see (\ref{eqVarest}), we thus obtain (\ref{VarB_nkUBound}).
\end{proof}

\subsubsection{ Lower bound on $H_n$ }
\label{H_nLowerBoundSection}

The most difficult part of the proof of Theorem~\ref{HeightTheorem} is to prove a lower bound for
the expected profile.
\begin{lemma}\label{LemunkLbound}
Suppose that $p>q$, that $\epsilon> 0$, and that $k_L$ is given by (\ref{eqkLdef}).
Then we have 
\begin{eqnarray}
\mu_{n,k_L} = \Omega\left( p^{ -\epsilon(\log_{p/q} \log n)^2/2 + 
O(\log\log n \log\log \log n) }   \right). 
    \label{eqEB_nkLBound}
\end{eqnarray}
\end{lemma}

By combining Lemma~\ref{VarianceLemma} and Lemma~\ref{LemunkLbound} it immediately follows that
\[
\Pr[H_n < k_L] \le \frac{\Var[B_{n,k_L}]}{ \mu_{n,k_L}^2}   \to 0
\]
which proves the lower bound on $H_n$.


The plan to prove Lemma~\ref{LemunkLbound} is as follows: 
we evaluate the inverse Mellin integral exactly by a residue computation.  This results in a nested
summation, which we simplify using the binomial theorem and the series of the
exponential function. From this representation we will then detect 
several terms that contribute to the leading term in the asymptotic expansion.

\begin{lemma}\label{Lelower-1}
Suppose that $\rho < 0$ but not an integer. Then we have
\begin{eqnarray}
    \Po{G}_k(n)
    = \sum_{j=0}^k \sum_{m \ge j} \kappa_{m,j}\left( \mu_{m,j}-\mu_{m,j-1}\right),
    \label{G_kExplicitFormHLowerBound}
\end{eqnarray}
where 
\begin{eqnarray}\label{eqkapparep}
    \kappa_{m,j} 
    = \frac{T(-m)n^m}{m!} 
    \sum_{\ell= (-\ceil{m+\rho} + 1) \logicor 0}^{\infty} \frac{(-n)^{\ell}}{\ell!} T(-m-\ell)^{k-j} 
\end{eqnarray}
and $x \logicor y$ denotes the maximum of $x$ and $y$.
\end{lemma}

\begin{proof}
By shifting the line of integration and collecting residues we have
\[
\frac 1{2\pi i} \int_{\rho-i\infty}^{\rho + i\infty} 
n^{-s}T(s)^{k-j} \Gamma(m+s)\, ds = 
\sum_{\ell \ge \max\{0, - m - \rho\} } 
\frac{n^{m + \ell} (-1)^\ell }{\ell!} T(-\ell- m)^{k-j},
\]
where the remaining integral after shifting by any finite amount
becomes $0$ in the limit as a result of the superexponential decay
of the $\Gamma$ function on the points $1/2 - j$ for positive integer $j$.
Hence the lemma follows.
\end{proof}

We now choose $\rho$ as $\rho = -j^*-1$ and set $j_0 = \lfloor j^* + \frac 12 \rfloor$, where
$j^*$ is the root of the equation
\begin{eqnarray}\label{eqj0def}
 (q/p)^{j*}(k_L-j^*) =  \frac {\log(1/p)}{\log(p/q)}(j^* -  \psi_L(n)),
\end{eqnarray}
where $\psi_L(n) = \frac 12 \left( 1 - \epsilon\right) \log_{p/q}\log n$.

In particular, let us define
\begin{eqnarray}
	\overline r_0 := (q/p)^{j_0} (k_L - j_0), && \overline r_1 := \frac{\log(1/p)}{\log(p/q)}(j_0 - \psi_L(n)).
\end{eqnarray}
Then it follows that
\begin{eqnarray}\label{eqj0def-app1}
 \sqrt{q/p} \overline r_1
 \le \overline r_0
 \le \sqrt{p/q} \overline r_1.
\end{eqnarray}
If $j> j_0$ and $m\ge j$ then we certainly have $(-\ceil{m+\rho} + 1) 
\logicor 0 = 0$,
whereas for $j=j_0$ we have $(-\ceil{j_0+\rho} + 1) \logicor 0 = 1$.

Asymptotically we have $j^* = \log_{p/q}\log n -\log_{p/q}\log\log n + O(1)$. Hence we also have
$j_0 = \log_{p/q}\log n -\log_{p/q}\log\log n + O(1)$ and
$\rho = - \log_{p/q}\log n -\log_{p/q}\log\log n + O(1)$.

In what follows we will encounter several different asymptotic behaviors. In particular we
will show that 
\begin{eqnarray}
\tilde G_k(n) &=& 
D(p) C_*(p) p^{j_0(j_0+1)/2} q^{j_0-1} n^{j_0} p^{j_0(k-j_0)} e^{\overline r_0} \Phi\left( \frac{\overline r_1 - \overline r_0}{\sqrt{\overline r_0}} \right) 
\label{eqasymrel}\\
&+& C_*(p) p^{j_0(j_0+1)/2} q^{j_0-1} n^{j_0} p^{j_0(k-j_0)}  \frac{{\overline r_0}^{\overline r_1}}{\Gamma(\overline r_1+1)}
\left( C(p, \overline r_0/\overline r_1, \langle \overline r \rangle) + o(1)    \right) \nonumber
\end{eqnarray}
where
$\langle x \rangle = x - \lfloor x \rfloor$ denotes the fractional part of a real number $x$, and
\begin{eqnarray} \label{eqDefAp}
D(p) = \sum_{L,M \ge 0} \xi_{L+1}\frac{(-1)^{M}}{M!} p^{((L+M)^2 + L-M)/2} q^{-L-M}
\end{eqnarray}
and $C(p,u,v)$ is a certain function in $p,u,v$ that is strictly positive (see below).  
Here and elsewhere, $\Phi$ denotes the distribution function of the normal distribution.

Since ${{\overline r_0}^{\overline r_1}}/{\Gamma(\overline r_1+1)} = O(e^{\overline r_0}/\sqrt{\overline r_0} )$.
Thus, the first term seems to be the asymptotically leading one. However, it turns out
that $D(p) \equiv 0$ (as we will prove in Section~\ref{secmiracle}) so it follows that 
\begin{eqnarray} \label{eqpositive}
\tilde G_k(n) \ge C(p)
 p^{j_0(j_0+1)/2} q^{j_0-1} n^{j_0} p^{j_0(k-j_0)}  \frac{{\overline r_0}^{\overline r_1}}{\Gamma(\overline r_1+1)}
\end{eqnarray}
for some constant $C(p) > 0$. Note also that this lower bound implies
(\ref{eqEB_nkLBound}) since by definition $\sqrt{q/p}\, \overline r_0  \le \overline r_1 \le  \sqrt{p/q}\, \overline r_0$ so that
$e^{-\overline r_0} {{\overline r_0}^{\overline r_1}}/{\Gamma(\overline r_1+1)} = e^{\Omega(\log\log n)}$.

The calculations for proving (\ref{eqasymrel}) are quite involved, in particular 
the proof of positivity of $C(p,u,v)$. So we will only present a part of the
calculations and refer for a full proof to the appendix.

In what follows we will have some  
error terms that are smaller by a factor $p^{j_0}$ or $(q/p)^{j_0}$ compared to the 
asymptotic leading term. However, it is easy to check that for $\frac 12 < p < 1$ we have $p^{j_0} = o(E)$ and $(q/p)^{j_0}= o(E)$, 
where $E: = e^{-\overline r_0} {{\overline r_0}^{\overline r_1}}/{\Gamma(\overline r_1+1)}$
so that they will not influence the asymptotic leading term.

For $j\le j_0$ and $j\le m\le j_0$ we  have
\begin{eqnarray*}
\kappa_{m,j} =  \frac{T(-m)n^m}{m!} \sum_{r=0}^{k-j} {k-j \choose r}  
p^{m(k-j-r)} q^{mr}  
\left(  e^{-np^{k-j-r}q^r} - \sum_{\ell\le j_0-m} \frac{(-n)^\ell}{\ell!} (p^{k-j-r}q^{r})^\ell   \right)
\end{eqnarray*}
and  otherwise
\begin{eqnarray*}
\kappa_{m,j} =  \frac{T(-m)n^m}{m!} \sum_{r=0}^{k-j_0} {k-j \choose r} 
p^{m(k-j-r)} q^{mr)}   e^{-np^{k-j-r}q^r}.
\end{eqnarray*}

In view of the above discussion we can thus replace the term $T(-m)$ (in $\kappa_{m,j}$) by $p^m$;
the resulting sum will be denoted by $\overline \kappa_{m,j}$.
We can also replace $\mu_{m,j} - \mu_{m,j-1}$ by
\[
\overline \nu_{m,j} := - C_*(p) m! p^{j(j-1)/2} q^{j-1} \xi_{m-j+1}  .
\]
By a careful look we thus obtain
\begin{equation}
 \Po{G}_k(n) = 
     \sum_{j=0}^k \sum_{m \ge j} \overline\kappa_{m,j} \overline \nu_{m,j}  
    + O\left( n^{j_0} T(-j_0)^{k-j_0}p^{j_0(j_0+1)/2}q^{j_0} \left( p^{j_0} + (q/p)^{j_0} \right) \right). \label{G_kExplicitFormHLowerBound-2}
\end{equation}

In order to analyze the sum representation (\ref{G_kExplicitFormHLowerBound-2}) we split it
into several parts:
\[
T_1 :=  \sum_{j> j_0} \sum_{m\ge j}  \overline\kappa_{m,j} \overline \nu_{m,j}, \quad
T_2 :=  \sum_{j\le  j_0} \sum_{m > j_0} \overline \kappa_{m,j} \overline \nu_{m,j}, \quad
T_3 :=  \sum_{j\le  j_0} \sum_{m=j}^{j_0} \overline \kappa_{m,j} \overline \nu_{m,j}. 
\]
Note that the exponential function $e^{-np^{k-j-r}q^r}
= e^{-(q/p)^{r-r_1(j)}}$ behaves completely differently
for $r\le r_1(j)$ and for $r> r_1(j)$
where $r_1(j) = (j- \psi(n))\frac{\log(1/p)}{\log(p/q)}$.
Hence it is convenient to split $T_3$ into three parts $T_{30} + T_{31}+T_{32}$, where the $T_{30}$ and $T_{31}$ correspond
to the terms with $r\le r_1(j)$ and $T_{32}$ for those with $r> r_1(j)$.
$T_{30}$ involves the exponential function $e^{-np^{k-j-r}q^r}$ whereas
$T_{31}$ takes care of the polynomial sum
$\sum_{\ell\le j_0-m} \frac{(-n)^\ell}{\ell!} (p^{k-j-r}q^{r})^\ell$.

For notational convenience we set
\begin{equation}\label{eqF0-0}
F_0 :=  p^{j_0(j_0+1)2} q^{j_0-1} n^{j_0} p^{j_0(k-j_0)} \frac{ \overline r_0^{\overline r_1}} {\Gamma(\overline r_1 +1)}.
\end{equation}

We recall that 
\[
T_1 = - C_*(p) \sum_{j> j_0} \sum_{m\ge j} p^{j(j-1)/2} q^{j-1} \xi_{m-j+1} p^m n^m
\sum_{r=0}^{k-j} {k-j \choose r} p^{m(k-j-r)} q^{mr} 
e^{-n p^{k-j-r} q^r}.
\]
We use now the substitutions $j=j_0 + J$ and $m = j+ L = j_0 + J + L$, 
where $J > 0$ and $L\ge 0$. Furthermore by using approximation 
${k-j \choose r} \sim (k-j)^r/r! \sim (k-j_0)^r/r!$ we obtain
\begin{eqnarray*}
T_1 &\sim & - C_*(p)  p^{j_0(j_0+1)2} q^{j_0-1} n^{j_0} p^{j_0(k-j_0)} 
\sum_{J > 0} \sum_{L\ge 0} p^{J (J+1)/2} q^{J} \xi_{L+1} p^L \\
&&\qquad \times \sum_{r} \frac{ {\overline r_0}^r } {r!}  (q/p)^{(L+J)(r-r_1(j))} e^{-(q/p)^{r-r_1(j)}} \\
&\sim&  - C_*(p)F_0  \cdot\sum_{J > 0} p^{J (J+1)/2} q^{J}  
\left( \frac {\overline r_0}{\overline r_1} \right)^{J \frac{\log(1/p)}{\log(p/q)}}
  \sum_{L\ge 0} 	\xi_{L+1} p^L \\
&&\qquad \times   
\sum_r (q/p)^{(L+J)(r-r_1(j))} \left( \frac {\overline r_0}{\overline r_1} \right)^{r-r_j(j)}   e^{-(q/p)^{r-r_1(j)}},
\end{eqnarray*}
where $F_0$ is given in (\ref{eqF0-0}). 
Thus, if we define (with the implicit notation $q = 1-p$)
\begin{eqnarray}
C_1(p,u,v) &=& \sum_{J > 0} p^{J (J+1)/2} q^{J}  
u^{J \frac{\log(1/p)}{\log(p/q)}}
  \sum_{L\ge 0} 	\xi_{L+1} p^L \label{eqC1puv-0}   \\
&&\qquad \times   
\sum_{R\in \mathbb{Z}} \left( (q/p)^{(L+J)} u \right)^{R - v -J\frac{\log(1/p)}{\log(p/q)}}   e^{-(q/p)^{R - v -J\frac{\log(1/p)}{\log(p/q)}}}  \nonumber
\end{eqnarray}
we obtain
\[
T_1 \sim - C_*(p)\,F_0\,  C_1\left( p, \frac{\overline r_0}{\overline r_1}, \langle \overline r_1 \rangle \right).
\]  
Note that we have substituted $r-r_1(j)$ by
\begin{eqnarray*}
r - r_1(j) &=& (r-\lfloor \overline r_1 \rfloor) - \langle \overline r_1 \rangle  + (\overline r_1 - r_1(j)) \\
&=& R - v - J\frac{\log(1/p)}{\log(p/q)}.
\end{eqnarray*}

Similarly we obtain $T_2 \sim - C_*(p)\,F_0\,  C_2\left( p, \frac{\overline r_0}{\overline r_1}, \langle \overline r_1 \rangle \right)$,
where
\begin{eqnarray}
C_2(p,u,v) &=& \sum_{J \le 0} p^{J (J+1)/2} q^{J}  
u^{J \frac{\log(1/p)}{\log(p/q)}}
  \sum_{L> -J} 	\xi_{L+1} p^L \label{eqC2puv-0}   \\
&&\qquad \times   
\sum_{R\in \mathbb{Z}} \left( (q/p)^{(L+J)} u \right)^{R - v -J\frac{\log(1/p)}{\log(p/q)}}   e^{-(q/p)^{R - v -J\frac{\log(1/p)}{\log(p/q)}}},  \nonumber
\end{eqnarray}
and
$T_{30} \sim - C_*(p)\,F_0\,  C_{30}\left( p, \frac{\overline r_0}{\overline r_1}, \langle \overline r_1 \rangle \right),$  
where
\begin{eqnarray}
C_{30}(p,u,v) &=& \sum_{J \le 0} p^{J (J+1)/2} q^{J}  
u^{J \frac{\log(1/p)}{\log(p/q)}}
  \sum_{L=0}^{-J} 	\xi_{L+1} p^L \label{eqC30puv-0}   \\
&&\qquad \times   
\sum_{R\in \mathbb{Z}, R - v -J\frac{\log(1/p)}{\log(p/q)}\le 0 } 
\left( (q/p)^{(L+J)} u \right)^{R - v -J\frac{\log(1/p)}{\log(p/q)}}   e^{-(q/p)^{R - v -J\frac{\log(1/p)}{\log(p/q)}}},  \nonumber
\end{eqnarray}
and $T_{32} \sim - C_*(p)\,F_0\,  C_{32}\left( p, \frac{\overline r_0}{\overline r_1}, \langle \overline r_1 \rangle \right)$, 
where
\begin{eqnarray}
C_{32}(p,u,v) &=& \sum_{J \le 0} p^{J (J+1)/2} q^{J}  
u^{J \frac{\log(1/p)}{\log(p/q)}}
  \sum_{L=0}^{-J} 	\xi_{L+1} p^L  \nonumber  \\
&&\qquad \times   
\sum_{R\in \mathbb{Z}, R - v -J\frac{\log(1/p)}{\log(p/q)} > 0 } 
\left( (q/p)^{(L+J)} u \right)^{R - v -J\frac{\log(1/p)}{\log(p/q)}} \label{eqC32puv-0}  \\
&& \qquad \qquad \times \left(  e^{-(q/p)^{R - v -J\frac{\log(1/p)}{\log(p/q)}}} 
- \sum_{\ell=0}^{-J-L} \frac{(-1)^\ell}{\ell!} (q/p)^{(R - v -J\frac{\log(1/p)}{\log(p/q)})\ell} 
  \right). \nonumber
\end{eqnarray}

Finally we deal with $T_{31}$.
First of all we regroup the summation by setting $m=j_0-M$, $j=j_0-M-L$, and $\ell = M-K$ which
gives
\begin{eqnarray*}
T_{31} &=& C_*(p) p^{j_0(j_0+1)/2} q^{j_0-1} n^{j_0} p^{j_0(k-j_0)}  \sum_{K\ge 0} \left( \frac qp \right)^{K\overline r_1} \\
&&\times  \sum_{L\ge 0, \, M\ge K} \xi_{L+1}\frac{(-1)^{M-K}}{(M-K)!} p^{((L+M)^2 + L-M)/2-K(L+M)} q^{-L-M} \\
&& \times \quad \sum_{r\le r_1(j_0-M-L)} {k-j_0+M+L \choose r } \left( \frac qp \right)^{(j_0-K)r}.
\end{eqnarray*} 
We single out the case $K=0$ (and consider only the sum over $K,M,r$) which 
we write as
\begin{eqnarray*}
D(p) C_*(p) \sum_{r\le \overline r_1} {k-j_0+L+M \choose r } \left( \frac qp \right)^{j_0r} + S_0,
\end{eqnarray*}
where $D(p)$ is given by (\ref{eqDefAp}) and 
\begin{eqnarray*}
S_0&:=&  - C_*(p)\sum_{L,M \ge 0} \xi_{L+1}\frac{(-1)^{M}}{M!} p^{((L+M)^2 + L-M)/2} q^{-L-M} \\
&& \qquad \times \sum_{r_1(j_0-M-L)< r \le \overline r_1} {k-j_0+L+M \choose r } \left( \frac qp \right)^{j_0r}.
\end{eqnarray*}
Note that 
\[
\sum_{r\le \overline r_1} {k-j_0+L+M \choose r } \left( \frac qp \right)^{j_0r} 
= e^{\overline r_0} \Phi\left( \frac{\overline r_1 - \overline r_0}{\sqrt{\overline r_0}} \right)
\left( 1 + O\left( \frac{\log\log n}{\log n} (L+M) \right)\right),
\]
where $\Phi$ denotes the distribution function of the normal distribution.

Thus, if we set 
\begin{eqnarray*}
S_K&=& C_*(p)p^{j_0(j_0+1)2} q^{j_0-1} n^{j_0} p^{j_0(k-j_0)}   \left( \frac qp \right)^{K\overline r_1} \\
&&\qquad \times  \sum_{L\ge 0, \, M\ge K} \xi_{L+1}\frac{(-1)^{M-K}}{(M-K)!} p^{((L+M)^2 + L-M)/2-K(L+M)} q^{-L-M} \\
&&\qquad \qquad \times \sum_{r\le r_1(j_0-M-L)} {k-j_0+M+L \choose r } \left( \frac qp \right)^{(j_0-K)r}.
\end{eqnarray*}
then we have
\[
T_{31} = D(p) C_*(p) e^{\overline r_0} \Phi\left( \frac{\overline r_1 - \overline r_0}{\sqrt{\overline r_0}} \right) (1+o(1))
-S_0 + \sum_{K\ge 1} S_K.
\]

In the same way as above we obtain 
$S_0 \sim - C_*(p)\,F_0\,  C_{31,0}\left( p, \frac{\overline r_0}{\overline r_1}, \langle \overline r_1 \rangle \right)$,
where
\begin{eqnarray}
C_{31,0}(p,u,v) &=& 
\sum_{L,M \ge 0} \xi_{L+1}\frac{(-1)^{M}}{M!} p^{((L+M)^2 + L-M)/2} q^{-L-M} \label{eqC310puv-0}     \\
&&\qquad \times 
\sum_{-(M+L) \frac{\log(1/p)}{\log(p/q)} +v \le R \le 0 } u^{R-v}.
\nonumber
\end{eqnarray}
It is also convenient to rewrite this also as a sum over $J = -M-L\le 0$ and $0\le L \le -J$:
\begin{eqnarray}
C_{31,0}(p,u,v) &=& 
\sum_{J\le 0}\sum_{L=0}^{-J} \xi_{L+1}\frac{(-1)^{-J-L}}{(-J-L)!} p^{ J(J+1)/2 + L} q^{J} \label{eqC310puv-2-0}     \\
&&\qquad \times 
\sum_{ J \frac{\log(1/p)}{\log(p/q)} +v \le R \le 0 } u^{R-v}.
\nonumber
\end{eqnarray}
For $K\ge 1$ the terms $S_K$ can be approximated by
$S_K \sim C_*(p)\,F_0\,  C_{31,K}\left( p, \frac{\overline r_0}{\overline r_1}, \langle \overline r_1 \rangle \right)$,  
where
\begin{eqnarray}
C_{31,K}(p,u,v) &=& 
\sum_{J\le -K} \sum_{L=0}^{-J-K} \xi_{L+1}\frac{(-1)^{-J-L-K}}{(-J-L-K)!} p^{J(J+1)/2 + L +JK} q^{J} \nonumber     \\
&&\qquad \times 
\sum_{R\le v+J \frac{\log(1/p)}{\log(p/q)} } 
 \left( u \left(\frac qp\right)^{-K} \right)^{R-v}.
\label{eqC31Kpuv-2-0}
\end{eqnarray}

Summing up, if we set 
\[
C(p,u,v) = - C_1(p,u,v)-C_2(p,u,v)-C_{30}(p,u,v)-C_{32}(p,u,v)-C_{31,0}(p,u,v)+ \sum_{K\ge 1} C_{31,K}(p,u,v)
\]
and by observing that $D(p) = 0$ (see Section~\ref{secmiracle}) we have:

\begin{lemma}
With the notation from above we have
\[
\tilde G_k(n) =  
C_*(p) p^{j_0(j_0+1)/2} q^{j_0-1} n^{j_0} p^{j_0(k-j_0)}  \frac{{\overline r_0}^{\overline r_1}}{\Gamma(\overline r_1+1)}
\left( C(p, \overline r_0/\overline r_1, \langle \overline r \rangle) + o(1)    \right).
\]
\end{lemma}

It remains to show that $C(p,u,v)$ is strictly positive for $\frac 12 < p < 1$, $\sqrt{q/p} \le u \le \sqrt{p/q}$, $0\le v < 1$.
Since the representation of $C(p,u,v)$ is quite involved we will use the following strategy. We do an asymptotic analysis 
for $p\to \frac 12$ and $p\to 1$ and fill out the remaining interval, $0.51 \le p \le 0.97$ by a numerical analysis
(together with upper bounds for the derivatives). Due to space limitations we present here only a short version of
the (very involved) considerations. A full version can be found in the appendix.

We start with the behavior for $p\to \frac 12$.
\begin{lemma}\label{Lepto1/2}
Set $p/q = e^{\eta}$ and $\tilde u = \frac 1\eta \log u$. 
Then for $\eta \to 0+$ (which is equivalent to $p\to \frac 12$) we have uniformly for $\tilde u \in [-\frac 12, \frac 12]$ and $v\in [0,1)$
\begin{equation}\label{eqLepto1/2}
C(p,u,v) \sim \frac 1\eta h(\tilde u),
\end{equation}
where $h(\tilde u)$ is a continuous and positive function.

In particular we have $C(p,u,v) > 0$ for $\frac 12 < p \le 0.51$.
\end{lemma}

\begin{proof}
We single out the function $C_1(p,u,v)$ and start with the sum over $R$. The first observation is that for $\eta \to 0$ we can replace the
sum by an integral, that is, we have for fixed integers $L,J$, as $\eta\to 0$,
\begin{eqnarray*}
&& \sum_{R\in \mathbb{Z}} \left( (q/p)^{(L+J)} u \right)^{R - v -J\frac{\log(1/p)}{\log(p/q)}}   e^{-(q/p)^{R - v -J\frac{\log(1/p)}{\log(p/q)}}} \\
&&\sim \int_{-\infty}^\infty 
\left( (q/p)^{(L+J)} u \right)^{t}   e^{-(q/p)^{t}} \, dt  = \frac 1\eta \int_{-\infty}^\infty e^{-\left(M - \tilde u \right) t} e^{-e^{-t}}\, dt.
\end{eqnarray*}
This also implies that the leading asymptotic term does not depend on $v$.
Further note that $\tilde M = M - \frac 1\eta \log u = L+J - \tilde u \ge \frac 12$ so that the integral converges
and by using the substitution $w = e^{-t}$ we obtain
\[
\int_{-\infty}^\infty e^{-\tilde M t} e^{-e^{-t}}\, dt 
= \int_0^\infty w^{\tilde M -1} e^{-w}\, dw = \Gamma(\tilde M).
\]
This finally shows that, as $p\to \frac 12$ (or equivalently as $\eta = \log(p/q) \to 0$),
\begin{equation}\label{eqC1asymp}
C_1(p,u,v) \sim \frac 1\eta 
\sum_{J > 0} 2^{-J (J+1)/2 - J + J\tilde u}   
  \sum_{L\ge 0} 	\xi_{L+1}(1/2)\, 2^{-L}\, \Gamma\left( J + L - \tilde u\right).
\end{equation}

Similarly we can handle the other terms and obtain the asymptotic representation (\ref{eqLepto1/2}).
Since the function $h(\tilde u)$ is explicit (as a series expansion) and continuously differentiable
in $\tilde u$ we can use a simple numerical analysis (together with upper bounds for the derivative
$h'(\tilde u)$) in order to show that $h(\tilde u) > 0$ for $\tilde u \in [-\frac 12, \frac 12]$.

Finally by taking care also on error terms (which were neglected in the above analysis) it also
follows that $C(p,u,v) > 0$ for $\frac 12 < p \le 0.51$.
\end{proof}

The situation for $p\to 1$ is more delicate in the analysis, however, positivity follows then immediately.
\begin{lemma}\label{Lepto1}
Set $\overline c(v) = \max\{v-v^2/2,(1-v^2)/2\}$. Then we have, as $p\to1$ uniformly 
for $\sqrt{q/p} \le u \le \sqrt{p/q}$, $0\le v < 1$
\begin{equation}\label{eqLepto1}
C(p,u,v) \ge \exp\left( \overline c(v) \frac{\log^2 (1-p)}{\log 1/p}(1 + o(1)) \right).
\end{equation}
In particular we have $C(p,u,v) > 0$ for $0.97\le p < 1$.
\end{lemma}

\begin{proof}
We just consider the most interesting case, namely the sum $\sum_{K\ge 1} C_{31,K}(p,u,v)$ and
assume for a moment that $v> 0$. We set
\[
I_0:= \left[ - v\left( \frac{\log q}{\log p}-1\right) , 0 \right) \cap \mathbb{Z}
\]
and for $M\ge 1$
\[
I_M:= \left[ - (v+M)\left(\frac{\log q}{\log p}-1\right) , - (v+M-1)\left(\frac{\log q}{\log p}-1\right) \right) \cap \mathbb{Z}
\]
If $J\in I_M$ we have, as $p\to 1$, 
\[
\sum_{R\le v+ J \frac{\log(1/p)}{\log(p/q)}} \left( u (q/p)^{-K} \right)^{R-v} \sim 
\left( u (q/p)^{-K} \right)^{-M-v}.
\]
Since 
\[
\sum_{L=0}^{-J-K} \xi_{L+1} p^L \frac{(-1)^{-J-K-L}}{(-J-K-L)!} = [z^{-J-K}] \prod_{j\ge 0} \frac{e^{qp^jz} -1}{qp^j z} e^{-z} = [z^{-J-K}] e^{z/2 + O(qz^2) - z}
\]
we get
\begin{eqnarray*}
C_{31,K,M}&:=& \sum_{J\in I_M,\, J \le -K} p^{J(J+1)/2 +JK}q^J  \sum_{R\le v+ J \frac{\log(1/p)}{\log(p/q)} } 
 \left( u \left(\frac qp\right)^{-K} \right)^{R-v} \\
&&\qquad \times\sum_{L=0}^{-J-K} \xi_{L+1}p^{L} \frac{(-1)^{-J-L-K}}{(-J-L-K)!}     \\
&\sim &  \sum_{J\in I_M,\, J \le -K}  p^{J(J+1)/2 +JK}q^J  
 \left( u \left(\frac qp\right)^{-K} \right)^{-M-v} [z^{-J-K}] e^{z/2 + O(qz^2) - z}
\end{eqnarray*}
and consequently if we sum over $K\ge 1$
\begin{eqnarray*}
\sum_{K\ge 1} C_{31,K,M} &\sim &  u^{-M-v}
\sum_{J\in I_M} p^{J(J+1)/2}q^J \sum_{K=1}^{-J} p^{JK}(q/p)^{K(M+v)} [z^{-J-K}] e^{z/2 + O(qz^2) - z} \\
& =  &  u^{-M-v}
\sum_{J\in I_M} p^{J(J+1)/2}q^J \sum_{K=1}^{-J} p^{JK}(q/p)^{M(1+v)} [z^{-J-K}] e^{z/2 + O(qz^2) - z}.
\end{eqnarray*}
We observe that (for $J\in I_M$) 
\begin{eqnarray*}
 &&\sum_{K=1}^{-J} p^{JK}(q/p)^{K(M+v)} [z^{-J-K}] e^{z/2 + O(qz^2) - z}
 = [z^{-J}] \frac{p^{J}(q/p)^{M+v}z}{1- p^{J}(q/p)^{M+v}z } e^{z/2 + O(qz^2) - z} \\
 &&\qquad \sim   p^{-J^2}(q/p)^{-J(M+v)} e^{z_M/2 + O(q z_M^2) - z_M},
\end{eqnarray*}
where $z_M = p^{-J}(q/p)^{-M-v}$. Note that $z_M$ varies between $1$ and $1/q$ if $J\in I_M$.
However, it will turn out that the asymptotic leading terms will come from $J$ close to 
$- (v+M)\frac{\log q}{\log p}$ which means that $z_M$ asymptotically $1$ and, thus,  the last
exponential term is asymptotically $e^{-1/2}$. The reason is that the term
\[
p^{J(J+1)/2}q^J p^{-J^2}(q/p)^{-J(M+v)} = p^{-J^2/2} q^{J(1-M-v)} p^{J(\frac 12 +M+v)}
\]
has its absolute minimum for $J$ close to $- (v+M-1)\frac{\log q}{\log p}$ and for $J\in I_M$
it gets maximal for $J$ close to $- (v+M)\frac{\log q}{\log p}$, in particular if
\[
J = J_{v,M} := - \left\lfloor (M+v) \left( \frac{\log q}{\log p} -1 \right) \right\rfloor.
\]
Thus, we obtain
\begin{eqnarray*}
\sum_{K\ge 1} C_{31,K,M} &\sim &  e^{-\frac 12} u^{-M-v} p^{-J_{v,M}^2/2} q^{J_{v,M}(1-M-v)} p^{J_{v,M}(\frac 12 +M+v)} \\
&=& e^{ \frac{\log^2 q}q \left( M+v- \frac 12 (M+v)^2 \right) + O(\log^2 q)  }. 
\end{eqnarray*}
Since $(M+v) - \frac 12 (M+v)^2 \le 0$ for $M\ge 2$ (and $0\le v < 1$) it is clear that only the first two
terms corresponding to $M=0$ and $M=1$ are relevant
Hence, we obtain
\[
\sum_{K\ge 1} C_{31,K} \sim
e^{ \frac{\log^2 (1-p)}{\log(1/p)} \left( v- \frac 12 v^2 \right) + O(\log^2 (1-p))  }
+ e^{ \frac{\log^2 (1-p)}{\log(1/p)} \frac 12 \left( 1- v^2  \right) + O(\log^2 (1-p))  }.  
\]
Acually this kind of representation also holds for $v= 0$.

The other terms can be handled in a similar way. Actually $C_1,C_2, C_{32}, C_{31,0}$ are of
smaller order, whereas $C_{30}$ has (almost) a comparable order of magnitude.

Finally, by taking error terms into account it follows that $C(p,u,v)$ is positive for
$0.97 \le p < 1$.
\end{proof}

Thus, it remains to consider $C(p,u,v)$ for $0.51 \le p \le 0.97$. As mentioned above we
do here a numerical analysis. For example, for the following sample valued we obtain:

\begin{center}
\begin{tabular}{llll}
\hline
$p$ & $u$ & $v$ & $C(p, u, v)$ \\
\hline
0.51 & 1.00 & 0.20 & 17.6603002053593  \\
0.51 & 1.00 & 0.40 & 17.6630153331822 \\
0.51 & 1.00 & 0.60 & 17.6610407898646 \\
0.51 & 1.00 & 0.80 & 17.6856832509155 \\
0.60 & 0.90 & 0.60 & 1.49524800151569 \\
0.60 & 1.00 & 0.20 & 1.08391296918222 \\
0.60 & 1.00 & 0.60 & 1.08391297098683 \\
0.60 & 1.00 & 0.80 & 1.08391297046200 \\
0.60 & 1.10 & 0.20 & 0.834656789094941 \\
0.60 & 1.20 & 0.60 & 0.673917281982084 \\
0.70 & 1.00 & 0.60 & 0.232497954955319 \\
0.80 & 1.00 & 0.60 & 0.0287161523336721 \\
0.85 & 1.00 & 0.60 & 0.00237172764900606 \\
0.93 & 1.00 & 0.60 & 1.87317294616045 $\times  10^{15}$ \\
0.97 & 0.50 & 0.60 & 9.17733198126610 $\times  10^{72}$ \\
0.97 & 1.00 & 0.60 & 6.05478107453485 $\times 10^{72}$ \\
0.97 & 5.00 & 0.60 & 2.30524156812013 $\times 10^{72}$ \\
\hline
\end{tabular}
\end{center}

A more detailed analysis can be found in the appendix. 

\subsection{Proof of Theorem~\ref{FillupTheorem}}
\label{FProof}

The analysis of $F_n$ runs along the same lines as for $H_n$. As already mentioned
we will give only a roadmap of the proof since it is actually much easier than 
that of $H_n$.

\subsubsection{Lower bound on $F_n$}

The lower bound on $F_n$ can be proven in two different ways. We can either use 
the inverse Mellin transform integral for $\tilde G_k(n)$ 
\[
k = k_L = \log_{1/q}\log n - (1+\epsilon)\log_{1/q}\log\log n
\]
evaluated at 
$\rho = \log_{p/q}\log n$. This leads to ${\rm Pr}[F_n < k] \le \mu_{n,k} \to 0$.

Alternatively we can use the correspondence between the R\'enyi process
and the random PATRICIA trie construction, along with the relationship between PATRICIA tries
and standard tries.  Because of the path compression step in the construction of a PATRICIA trie
from a trie, the fillup level for a PATRICIA trie is always greater than or equal to the fillup 
level for the associated trie.  Furthermore, it is known (see \cite{park2008}) that the fillup level
in random tries for $p > 1/2$ is, with high probability,
\begin{eqnarray*}
    \log_{1/q} n - \log_{1/q}\log\log n + o(\log\log\log n).
\end{eqnarray*} 
Thus, with high probability, this is also a lower bound for the $F_n$ that we study.

\subsubsection{Upper bound on $F_n$}

The upper bound proof for $F_n$ follows along similar lines to 
the lower bound for $H_n$.
We set
\[
    k = k_U = \log_{1/q} n - (1-\epsilon)\log_{1/q}\log\log n,
\]
and our goal is to show that $\Var[B_{n,k}] = o(\E[B_{n,k}]^2)$. 
First we get an upper bound for $\Var[B_{n,k}]$ in the same way as in the case of $H_n$
(via inverse Mellin transform and Depoissonization) of the form
\[
\Var[B_{n,k}] = O\left(q^{-\epsilon \log_{p/q}\log n \cdot \log{1/q} \log\log n (1 + o(1))}\right).
\]

In order to obtain a corresponding lower bound for $\mu_{n,k} =\E[B_{n,k}]$ we 
use again the explicit representation
\begin{equation}
    \label{PoGkFillupExpression}
    \Po{G}_k(n) = \sum_{j = 0}^k \sum_{m \geq j} \kappa_{m,j} (\mu_{m,j} - \mu_{m,j-1}),
\end{equation}
where 
\begin{eqnarray}
    \kappa_{m,j} 
    &=& \frac{T(-m)n^m}{m!} \sum_{\ell=0}^\infty \frac{(-n)^{\ell}}{\ell!} T(-m-\ell)^{k-j} \nonumber \\
    &=& \frac{T(-m)}{m!} \sum_{r=0}^{k-j} { {k-j}\choose r} (np^r q^{k-j-r})^m \exp(-np^r q^{k-j-r}). \label{eqkapparep2}
\end{eqnarray}
We note that, because $\rho > 0$, there are no contributions from poles, so that the $\ell$-sum begins with $0$, in
contrast to (\ref{eqkapparep}) which leads to the simplified form (\ref{eqkapparep2}).

Our derivation suggests that the main contribution to (\ref{PoGkFillupExpression}) comes from the terms $j = O(1)$ and $m = \rho \cdot p/q + O(1)$.  In this range,
the difference $\mu_{m,j} - \mu_{m,j-1}$ is estimable by the following 
lemma from \cite{magnerknesslszpa2014} (see part (i) of Theorem 2.2 of
that paper).
\begin{lemma}[Precise asymptotics for $\mu_{m,j}$ when $j = O(1)$ and $m\to\infty$]
    \label{MuJO1Lemma}
    For $p > q$, $m \to \infty$, and $j = O(1)$, we have
    \[
        \mu_{m,j} \sim mq^j (1 - q^j)^{m-1}.
    \]
\end{lemma}
Note, in particular, that $\mu_{m,j} - \mu_{m,j-1}$ is strictly positive in this range.  Applying this lemma, some algebra is required to
show that the contribution of the $(m,j)$th term, with $m = \rho \cdot p/q + O(1)$ and $j = O(1)$, is
\begin{equation}
    \label{DominantTermFillup}
    q^{-\epsilon \log_{p/q}\log n \cdot \log_{1/q}\log\log n (1 + o(1))}.
\end{equation}
To complete the necessary lower bound on the entire sum (\ref{PoGkFillupExpression}), we consider also the following sums:
\begin{eqnarray}
    \sum_{j=0}^{j'}\sum_{m = j}^{m'} \kappa_{m,j} (\mu_{m,j} - \mu_{m,j-1}) \quad\mbox{and}\quad
    \sum_{j > j'}\sum_{m \geq j} \kappa_{m,j} (\mu_{m,j} - \mu_{m,j-1}),
\end{eqnarray}
where $j'$ and $m'$ are sufficiently large fixed positive numbers.  We note that the terms that are not covered by any of these 
sums may be disregarded, since by Lemma~\ref{MuJO1Lemma} they are non-negative.

It may be shown that both sums are smaller than the dominant term (\ref{DominantTermFillup}) by a factor of $e^{-\Theta(\rho)}$,
both by upper bounding terms in absolute value and using the trivial bound $|\mu_{m,j} - \mu_{m,j-1}| \leq 2m$.

We thus arrive at
\begin{equation}\label{eqmunkfilluplevel}
\mu_{n,k} \ge  q^{-\epsilon \log_{p/q}\log n \cdot\log_{1/q}\log\log n (1 + o(1)}.
\end{equation}
Since this tends to $\infty$ with $n$, combining this with the upper bound for the variance 
yields the desired upper bound on $\Pr[F_n > k]$, which establishes the upper bound on $F_n$.

\section{Depoissonization}\label{sec:depo}

\subsection{Analytic Depoissonization}

The Poisson transform $\tilde G(z)$ of a sequence
$g_n$ is defined by $\tilde G(z) = \sum_{n\ge 0} g_n \frac{z^n}{n!} e^{-z}$. 
If the sequence $g_n$ is {\it smooth enough} then we usually have 
$g_n \sim \tilde G(n)$ (as $n\to\infty$) which we call {\it Depoissonization}.
In \cite{jacquetszpa1997} a theory for {\it Analytic Depoissonization} is developed. For example,
the basic theorem (Theorem 1) says that if 
\begin{eqnarray}\label{eqDep1}
|\tilde G(z)| \le B |z|^\beta
\end{eqnarray}
for $|z|> R$ and $|\arg(z)| \le \theta$ (for some $B> 0$, $R>0$, and $0< \theta < \pi/2$) 
and 
\begin{eqnarray}\label{eqDep2}
|\tilde G(z)e^z| \le A e^{\alpha|z|}
\end{eqnarray}
for $|z|> R$ and $\theta <|\arg(z)| \le \pi$ (for some $A> 0$ and $\alpha < 1$) then
\begin{eqnarray}\label{eqDep3}
g_n = \tilde G(n) + O(n^{\beta-1}).
\end{eqnarray}
Actually this expansion can be made 
more precise by taking into account derivatives of $\tilde G(z)$.
For example, we have 
\begin{eqnarray}\label{eqDep3.2}
g_n = \tilde G(n) - \frac n2 \tilde G''(n)+  O(n^{\beta-2}).
\end{eqnarray}

In \cite[Lemmas 1 and 18]{magnerspa2015} it is shown that $\tilde G_k(z) = \sum_{n\ge 0} \mu_{n,k} \frac{z^n}{n!} e^{-z}$
satisfies (\ref{eqDep1}) with $\beta = 1+ \epsilon$ for any $\epsilon> 0$ and  
(\ref{eqDep2}) for some $\alpha < 1$ uniformly for
all $k\ge 0$. Thus, it follows uniformly for all $k\ge 0$ 
\begin{eqnarray}\label{eqDep3.3}
\mu_{n,k} = \tilde G_k(n)  - \frac n2 \tilde G_k''(n)+  O(n^{\epsilon-1}).
\end{eqnarray}
The estimate (\ref {eqDep3}) is not sufficient for our purposes 
(it only works if $\mu_{n,k}$ grows at least polynomially as in the 
{\it central range}). For the boundary region, where 
$k \sim \log_{1/p} n$ or $k \sim \log_{1/q} n$ we have to use (\ref{eqDep3.3}) which means 
that we have to deal with 
derivatives of $\tilde G_k(z)$, too.

\subsection{Poisson Variance}

Next we discuss how the variance of a random variable can be handled with the help
of the Poisson transform. First we assume that $\tilde G(z)$ is the Poisson transform
of the expected values $\mu_n = \E[X_n]$ or a sequence of random variables.
Furthermore we set 
\[
\tilde V(z) = \sum_{n\ge 0} \E[X_n^2]  \frac{z^n}{n!} e^{-z} - \tilde G(z)^2
\]
which we denote the Poisson variance. This is not the Poisson transform of the variance.
However, since we usually have $\E[X_n^2] \sim V(n) + G(n)^2$ and $\E[X_n] \sim G(n)$
it is expected that $\Var[X_n] \sim V(n)$. Actually this can be made precise with the
help of (\ref{eqDep3.2}). Suppose that $\tilde G(z)$ and  $\tilde V(z)$ satisfy the
property (\ref{eqDep1}) and that $\tilde G(z)$ and  $\tilde V(z) + \tilde G(z)^2$ the
property (\ref{eqDep2}). Then it follows that 
\[
\E[X_n] = \tilde G(n) - \frac n2 \tilde G''(n)+  O(n^{\beta-2})
\]
and
\[
\E[X_n^2] = \tilde V(n) + \tilde G(n)^2 - \frac n2 \tilde V''(n)
- n (\tilde G'(n))^2 - n \tilde G(n)\tilde G''(n) +  O(n^{\beta-2})
\]
from which it follows that 
\begin{eqnarray}
\Var[X_n] &= \tilde V(n) - \frac n2 \tilde V''(n) - n (\tilde G'(n))^2 +
\frac 14 n^2  (\tilde G''(n))^2 \nonumber \\
&+ O(n^{2\beta -4}) +  O( n^{\beta-2} \tilde G(n)) +  O( n^{\beta} \tilde G''(n)).
\label{eqVarest}
\end{eqnarray}
In particular in our case we know that the Poisson transform 
$\tilde G_k(z)$ (of the sequence $\mu_{n,k} = \E[B_{n,k}]$) and the 
corresponding Poisson variance $\tilde V_k(z)$ satisfy the assumptions
for $\beta =1 + \epsilon$ (for every fixed $\epsilon> 0$), see
\cite{magnerspa2015}. Thus we also obtain (\ref{eqVarest}) in the 
present context.

\section{An Unexpected Identity}\label{secmiracle}

In this final section we prove that $D(p) = 0$ which seems to be a new 
(and unexpected) identity.\footnote{The following simple proof is due to 
Gleb Pogudin (Univ. Linz).}

\begin{lemma}\label{LeDp=0}
Suppose that $|p|<1$ and $q = 1-p$ and set
\begin{eqnarray} \label{eqDefAp-0}
D(p) = \sum_{L,M \ge 0} \xi_{L+1}\frac{(-1)^{M}}{M!} p^{((L+M)^2 + L-M)/2} q^{-L-M},
\end{eqnarray}
where $\xi_{\ell} = \xi_\ell(p)$ is recursively defined by $\xi_1 = 1$ and 
    \begin{equation}\label{eqexrec2-0}
      \xi_\ell
        = q^{-1}p^{\ell} \sum_{J=1}^\ell \frac{\xi_{\ell+1-J}}{J!} (q/p)^{J}.
     \end{equation}
Then
\begin{equation}\label{eqLeDp=0}
D(p) =  0.
\end{equation}
\end{lemma}

\begin{proof}
By setting $L+M = n$ then we can rewrite $D(p)$ as 
\[
D(p) = \sum_{n\ge 0} p^{ n\choose 2} \sum_{L=0}^{n} \xi_{L+1} (p/q)^L \frac{(-1)^{(n-L)}}{(n-L)!} q^{-(n-L)}.
\]
Since the recurrence (\ref{eqexrec2-0}) for $\xi_\ell$ can be rewritten to 
\[
X(z) = \sum_{L\ge 0} \xi_{L+1} z^L = \prod_{j\ge 0} \frac{e^{qp^j z} -1 }{qp^j z}
\]
we thus obtain
\[
D(p) = \sum_{n\ge 0} p^{ n\choose 2} [z^n] X((p/q)z) e^{-z/q} = 
\sum_{n\ge 0} p^{ n\choose 2} [z^n] \prod_{j\ge 0} \frac{e^{(p-1)p^j z} - e^{-p^j z }}{p^{j+1} z}.
\]
Hence, if we set $f(z) = \frac 1{pz} \left( e^{(p-1)z} - e^{-z} \right)$, 
$F(z) = f(z) f(pz) f(p^2z)\cdots$,  and $F_n = [z^n] F(z)$ 
then $D(p) = 0$ is equivalent to $\sum_{n\ge 0} F_n p^{n\choose 2} = 0$.

We next set $g(z) = e^{-z}$, $h(z) = (e^z-1)/z$, and $q(z) = (1-e^{-z})/z$. Then
we have $f(z) = g(z)h(pz)$ and $q(z) = g(z)h(z)$ which implies the representation
\[
F(z) = \prod_{j\ge 0} g(p^j z) h(p^{j+1} z) = g(z) \prod_{j\ge 1} g(p^j z) h(p^{j} z) = g(z) \prod_{j\ge 1} q(p^j z).
\]
Hence, if we set  $Q(z) = q(z) q(pz) q(p^2z)\cdots$,  and $Q_n = [z^n] Q(z)$ then we also have
\[
F(z) = g(z) Q(pz) = (1-zq(z)) Q(pz) = Q(pz) - z Q(z) = \sum_{n\ge 0} Q_n(p^n z - z^{n+1}).
\]
So, finally, if we use the substitution $z^n \mapsto p^{n\choose 2}$ and the property ${ n+1 \choose 2} = {n\choose 2} + n$,
we immediately see that every summand vanishes. This proves $D(p) = 0$.
\end{proof}

%


\appendix

\def\Co{\mathcal{C}}
\def\Arg{\mathrm{Arg}}

\section{Extension of Lemma~\ref{KnesslLemma}}
\label{KnesslLemmaProof}
We prove here an extended version of Lemma~\ref{KnesslLemma}
and provide a more detailed analysis of the quantities $\xi_{\ell}(n)$ and $\xi_\ell$.

\begin{lemma}
    \label{KnesslLemma1}
    \label{KNESSLLEMMA1}         
    \label{mu_mjUpperBoundLemma1}
    Let $p \geq q$.
    For $n\to\infty$ and $1\le k < n$ with $\log^2 (n-k) = o(k)$,
$$
        \mu_{n,k}
        = (n-k)^{3/2 + \frac{\log q}{\log p}}
\frac{n!}{(n-k)!}p^{k^2/2 + k/2}q^{k} \cdot
\exp\left( -\frac{\log^2(n-k)}{2\log(1/p)} \right)\Theta(1).
$$
\end{lemma}

    We first work out the case $p> q$. The case $p=q = \frac 12$ is slightly easier and
    will be discussed below.
   
    We recall the definition of the quantity
$
        S_\ell(n) = \mu_{n,n-\ell}.
$
With this notation the quantities $\xi_\ell(n)$ are defined by
    \begin{equation}
        S_\ell(n) = n! C_*(p) p^{(n-\ell)^2/2 + (n-\ell)/2}q^{n-\ell} \xi_{\ell}(n).
        \label{S_ellDefinition}
    \end{equation}
where
$$
        C_*(p) = \prod_{j=2}^\infty (1-p^j - q^j)^{-1} \cdot (1 + (q/p)^{j-2}).
$$
By defintion and simple computation 
 we have $\xi_{1}(n) \sim 1$ and $\xi_{2}(n) \sim 1/2$.

    Our task now is to determine the asymptotic behavior of 
$\xi_{\ell}(n)$ as $\ell\to\infty$.  From the recurrence (\ref{muRecurrence}) 
for $\mu_{n,k}$, we can derive a corresponding one for $\xi_\ell(n)$:
    \begin{equation}
        \xi_{\ell}(n)(1-p^n-q^n) = 
     \sum_{J=1}^\ell \frac{\xi_{\ell+1-J}(n-J)}{J!}q^{-1}p^{\ell-n}(p^{n-J}q^{J} 
+ p^{J}q^{n-J}).
    \label{eqxielln-rec}
    \end{equation}
    Note that this recurrence uniquely defines $\xi_\ell(n)$ given 
$\xi_1(1) = 1/C_*(p)$.  First the recurrence provides $\xi_1(n)$ for all $n\ge 2$. 
Then we get $\xi_2(3)$  (recall that $\xi_2(1) = \xi_2(2) = 0$) 
and recursively $\xi_2(n)$ for $n\ge 4$; etc.
    
    From this recurrence it follows immediatly that for each $\ell\ge 1$ the limit
    $\xi_\ell = \lim_{n\to\infty} \xi_\ell(n)$
    exists. Clearly this limits satisfy the recurrence
$$
        \xi_\ell 
        = q^{-1}p^{\ell} \sum_{J=1}^\ell \frac{\xi_{\ell+1-J}}{J!} (q/p)^{J}
$$
    which we will analyze separately in the sequel.
 
\subsection{A-priori bounds for $\xi_\ell(n)$}  

    We will first prove an a-priori bound for $\xi_\ell(n)$ and also an error bound for
    the difference $\xi_\ell(n) - \xi_\ell$:
    \begin{equation}
    \xi_\ell(n) = O\left( \frac 1{\ell !} \right), \qquad 
    \xi_\ell(n) = \xi_\ell + O\left( \frac{p^{n-\ell} + 
(q/p)^{n-\ell}}{(\ell-1)!} \right),
    \label{eqxielln-est}
    \end{equation}
    where the implied constants depend on $p$.
    Both inequalities can be shown by induction by applying it recursively to 
    (\ref{eqxielln-rec}). Note first that (\ref{eqxielln-rec}) is only relevant
    for $n> \ell$ since $\xi_\ell(n) = 0$ for $n\le \ell$.
    
    We choose $\ell_0 \ge 1$ in a way that 
\[
\frac{p^{\ell_0}/q(e^{q/p}-1) + (q/p)^{\ell_0}/q(e^{p/q}-1)}{1-p^{\ell_0}-q^{\ell_0}} 
\le 1.
\]
and let $C_0$ be an upper bound for $\xi_\ell(n)$ for $\ell < \ell_0$ and 
all $n \ge 1$.  Then it follows by induction that $\xi_\ell(n) \le C_0$ 
for all $\ell \ge 1$ and $n\ge 1$.
We just have to observe that for $n> \ell \ge \ell_0$
\begin{eqnarray*}
\xi_\ell(n) &\le& \frac{C_0}{1-p^n-q^n} \left( \sum_{J=1}^{\ell} 
\frac{p^{\ell-J}q^{J-1}}{J!} + 
\sum_{J=1}^{\ell} \frac{p^{\ell+J-n}q^{n-J-1}}{J!}  \right) \\
&\le & \frac{C_0}{1-p^n-q^n} \left( \frac{p^\ell}q (e^{q/p}-1) + 
\frac 1q \left( \frac qp \right)^n  (e^{p/q}-1) \right) \\
&\le & \frac{C_0}{1-p^{\ell_0}-q^{\ell_0}} \left( \frac{p^{\ell_0}}q (e^{q/p}-1) + 
\frac 1q\left( \frac qp \right)^{\ell_0}  (e^{p/q}-1) \right) \\
&\le C_0.
\end{eqnarray*}

Next we prove the first inequality of (\ref{eqxielln-est}). 
Here we fix $\ell_1$ in a way that
\[
\frac 1{(\ell_1+1)(1-p^{\ell_1}-q^{\ell_1})} \left( \frac 1{pq} + 
\frac 1{q^2} \right) \le 1
\]
and set $C_1 = C_0 \ell_1!$. Then we automatically have 
$\xi_\ell(n) \le C_1/\ell!$ for $\ell \le \ell_1$ and
$n\ge 1$.  Furthermore we obtain by induction for $n>\ell > \ell_1$
\begin{eqnarray*}
\xi_\ell(n) &\le& \frac{C_1}{1-p^n-q^n} \left( \sum_{J=1}^{\ell} 
\frac{p^{\ell-J}q^{J-1}}{J!(\ell+1-J)!} + 
\sum_{J=1}^{\ell} \frac{p^{\ell+J-n}q^{n-J-1}}{J!(\ell+1-J)!}  \right) \\
&\le & \frac{C_1}{(\ell+1)!(1-p^n-q^n)} \\
&\times &\left( \frac 1{pq} \sum_{J=0}^{\ell+1} 
{\ell+1 \choose J} p^{\ell+1-J}q^{J} + \frac{(q/p)^{n-\ell}}{q^2} 
\sum_{J=0}^{\ell+1}   {\ell+1 \choose J} p^{J}q^{\ell+1-J}  \right) \\
&\le & \frac{C_1}{\ell!} \frac 1{(\ell_1+1)(1-p^{\ell_1}-q^{\ell_1})} 
\left( \frac 1{pq} + \frac{1}{q^2} \right) \\
&\le & \frac{C_1}{\ell!}.
\end{eqnarray*}

Finally we deal with the second inequality of (\ref{eqxielln-est}). 
Since 
\[
\xi_1(n) = \prod_{j> n} \frac{1-p^j-q^j}{1 + (q/p)^{j-2}}
\]
it is certainly true for $\ell =1$. Now it is an easy exercise to verify it 
for $\ell=2$, $\ell = 3$ etc.
by adapting possibly the implicit constant for each $\ell$.
For sufficiently large $\ell \ge \ell_2$ we can do a common inductive step 
because of the following calculations:
\begin{eqnarray*}
&&|\xi_\ell(n) - \xi_\ell| \\
&\le &
\left| \frac 1{1-p^n-q^n}\sum_{J=1}^\ell \frac{\xi_{\ell+1-J}(n-J)}{J!}q^{-1}
p^{\ell-n}(p^{n-J}q^{J} + p^{J}q^{n-J})  
- q^{-1}p^{\ell} \sum_{J=1}^\ell \frac{\xi_{\ell+1-J}}{J!} (q/p)^{J}  \right| \\
&\le& \left| \frac 1{1-p^n-q^n}\sum_{J=1}^\ell \frac{|\xi_{\ell+1-J}(n-J)-
\xi_{\ell+1-J}|}{J!}q^{-1}p^{\ell-n}(p^{n-J}q^{J} + p^{J}q^{n-J}) \right| \\
&+& \left| \frac 1{1-p^n-q^n}\sum_{J=1}^\ell \frac{\xi_{\ell+1-J}}{J!}p^{\ell-J}q^{J-1} 
- \sum_{J=1}^\ell \frac{\xi_{\ell+1-J}}{J!} p^{\ell-J}q^{J-1}   \right| \\
&+& \left| \frac 1{1-p^n-q^n}\sum_{J=1}^\ell 
\frac{\xi_{\ell+1-J}}{J!}p^{J+\ell-n}q^{n-J-1} \right| \\
 &=& D_1 + D_2 + D_3.
\end{eqnarray*}
By assumption  we have 
\[
|\xi_{\ell+1-J}(n-J)-\xi_{\ell+1-J}| \le C(p^{n-\ell-1} + (q/p)^{n-\ell-1})/(\ell-J)!,
\] 
where we can assume without loss of generality that $C\ge 1$. 
So we can estimate $D_1$ by
\begin{eqnarray*}
D_1 &\le&  \frac C{1-p^n-q^n}\sum_{J=1}^\ell \frac{p^{n-\ell-1} + 
(q/p)^{n-\ell-1}}{J!(\ell-J)!}q^{-1}p^{\ell-n}(p^{n-J}q^{J} + p^{J}q^{n-J}) \\
&\le & \frac C{\ell!(1-p^n-q^n)} \left( \frac{p^{n-\ell}}{pq} + 
\frac{p(q/p)^{n-\ell}}{q^2} \right) 
\left(  1  + (q/p)^{n-\ell} \right) \\
&\le& \frac{C p^{n-\ell}}{(\ell-1)!} \frac 2{pq(1-p^n-q^n)\ell} + 
\frac{C (q/p)^{n-\ell}}{(\ell-1)!} \frac {2p}{q^2(1-p^n-q^n)\ell}.
\end{eqnarray*}
Furthermore by using the inequality $\xi_\ell \le C_1/\ell!$ and 
the assumption $C\ge 1$ we obtain
$$
D_2 = \frac{p^n+q^n}{1-p^n-q^n} \xi_\ell 
\le \frac{C p^{n-\ell}}{(\ell-1)!} \frac {2C_1}{(1-p^n-q^n)\ell} 
$$
and
\begin{eqnarray*}
D_3 &\le& \frac {C_1}{1-p^n-q^n} \sum_{J=1}^\ell 
\frac{p^{-n} q^{n-1} }{J!(\ell-J)!} p^J q^{\ell-J} \\
&\le& {C_1 (q/p)^n}{\ell! q(1-p^n-q^n)} \\
&\le & \frac{C (q/p)^{n-\ell}}{(\ell-1)!} \frac {C_1}{q(1-p^n-q^n)\ell}.
\end{eqnarray*}
Hence, if $\ell \ge l_2$, where $\ell_2$ satisfies
\[
\frac 1{\ell_2(1-p^{\ell_2}-q^{\ell_2})} \left( \frac 2{pq} + 2C_1 \right) \le 1 
\]
and
\[
\frac 1{\ell_2(1-p^{\ell_2}-q^{\ell_2})} \left( \frac {2p}{q^2} + 
\frac{C_1}q \right) \le 1 ,
\]
we obtain the second inequality of (\ref{eqxielln-est}) 
for all $\ell\ge 1$ and $n\ge 1$.

\subsection{Asymptotics for $\xi_\ell$}

    In order to obtain asymptotics for $\xi_\ell$ we use the 
Poisson transform of $\ell! \xi_{\ell}$ 
(\emph{not} of $\xi_{\ell}$), which we denote by 
    $\Po{\xi}(z)$.  The functional equation is
$$
        \Po{\xi}(z) = \Po{\xi}(pz) \frac{1-e^{-qz}}{pqz}.
$$
    This may be iterated and  produces the explicit formula
    \begin{equation}
        \Po{\xi}(z) = z\prod_{j=0}^\infty \left( \frac{1 - e^{-qp^jz}}{qp^jz} \right),
        \label{PoXiFormula}
    \end{equation}
    which also shows that $\Po{\xi}(z)$ is entire.  Cauchy's integral formula then gives
$$
        \ell!\cdot \xi_{\ell}
        = \frac{1}{2\pi i} \oint_{\Co} \frac{e^{z}}{z^\ell}
\prod_{j=0}^\infty \left( \frac{1 - e^{-qp^jz}}{qp^jz} \right)\dee{z},
$$
    for a simple, closed contour $\Co$ encircling the origin. 
Actually we use the circle contour $|z| = \ell$.
    It follows then in precisely the same way as in  
\cite{jacquetszpa1997} that for $\ell \to \infty$,
    \begin{equation} 
        \xi_{\ell}
        \sim \frac{\Po{\xi}(\ell)}{\ell!} = \frac{1}{(\ell-1)!} 
        \exp\left[ \sum_{j=0}^\infty 
\log\left( \frac{1-e^{-qp^j\ell}}{qp^j\ell} \right) \right].
        \label{sum1}
    \end{equation}
    
    Our next task is to justify the this step which can be seen 
as a depoissonization step.  Let us first study the asymptotic 
behavior of $\Po{\xi}(\ell)$
   
   
   We reindex the sum by setting $j = \floor{\log_{1/p} \ell} + J$, so 
that, for a fixed $J$, $p^j\ell = O(1)$ as $\ell\to\infty$.
    Then $j = \log_{1/p}\ell + J - \fracpart{\log_{1/p}(\ell)}$, 
where $\fracpart{x}$ denotes the fractional
    part of $x$.  Defining $\alpha_{\ell} = \fracpart{\log_{1/p}\ell}$, we get
$$
              \Po{\xi}(\ell)= \ell \exp\left[
            \sum_{J=-\floor{\log_{1/p}\ell}}^{\infty} 
\log\left( \frac{1-e^{-qp^Jp^{-\alpha_\ell}}}{qp^Jp^{-\alpha_\ell}} \right).
        \right]
$$
    The sum then becomes
    \begin{eqnarray}
        &&\sum_{J=-\floor{\log_{1/p}\ell}}^{\infty} 
\log\left( \frac{1-e^{-qp^Jp^{-\alpha_\ell}}}{qp^Jp^{-\alpha_\ell}} \right)
        =  \nonumber \\
        &&\sum_{J=0}^\infty \log \left( 
\frac{1-e^{-qp^Jp^{-\alpha_\ell}}}{qp^Jp^{-\alpha_\ell}} \right)
        + \sum_{J=1}^\floor{\log_{1/p} \ell} \left[ 
\log(1-e^{-qp^{-J}p^{-\alpha_\ell}}) - \log q + \alpha_\ell \log p + 
J\log p \right] \nonumber \\
        &\sim& \floor{\log_{1/p} \ell}[\alpha_\ell \log p - \log q] + 
\floor{\log_{1/p} \ell} (\floor{\log_{1/p}\ell} + 1) 
\frac{1}{2} \log p \label{nonSummation} \\
        &+& \sum_{J=1}^\infty \log(1-e^{-qp^{-J}p^{-\alpha_\ell}}) + 
\sum_{J=0}^\infty \log \left( \frac{1-e^{-qp^Jp^{-\alpha_\ell}}}{qp^Jp^{-\alpha_\ell}} 
\right). \nonumber
    \end{eqnarray}
    The expression (\ref{nonSummation}) can be rewritten as
    \begin{eqnarray*}
        &&(\log_{1/p} \ell - \alpha_\ell)(\alpha_\ell \log p - \log q) + 
\frac{1}{2}\log p(\log_{1/p} \ell - \alpha_\ell) (\log_{1/p} \ell + 1 - \alpha_\ell) \\
        &=& \frac{1}{2}(\log_{1/p} \ell)^2 \log p + (\log_{1/p} \ell)
(-\log q + \frac{1}{2}\log p) - \alpha^2_\ell \log p \\
&+& \frac{1}{2}(\log p)\alpha_\ell(\alpha_\ell - 1) 
        + \alpha_\ell \log q,
    \end{eqnarray*}
    so that, finally,
         \begin{eqnarray*}
        \Po{\xi}(\ell)& =& \ell \prod_{J=0}^\infty \left( 
\frac{1 - e^{-qp^Jp^{-\alpha_\ell}}}{qp^Jp^{-\alpha_\ell}} \right) 
        \prod_{J=1}^\infty (1-e^{-qp^{-J}p^{-\alpha_\ell}}) \\
        & &\qquad \times \exp\left[ -\frac{\log^2 \ell}{2\log(1/p)} \right] 
\ell^{1/2+\log q / \log p}
        \exp\left[-\frac{\alpha_\ell(\alpha_\ell + 1)}{2}\log p \right] 
e^{-\alpha_\ell \log q} \\
        &=&  \ell^{3/2+\log q / \log p} \exp\left[ -\frac{\log^2 \ell}{2\log(1/p)} 
\right] \Theta(1).
    \end{eqnarray*}


It remains to check proper growth conditions on $\Po{\xi}(z)$ which can be directly 
used to justify the depoissonization step. Actually we show that we have similiar
properties as (\ref{eqDep1}) and (\ref{eqDep2}). First we show that 
\begin{equation}\label{eqDep1-new}
\Po{\xi}(z) = O\left(  |z|^{3/2+\log q / \log p}
        \exp\left[ -\frac{\log^2 |z|}{2\log(1/p)} \right] \right) 
\end{equation}
uniformly for $\Re(z) \ge \eta$, where $\eta> 0$ is fixed. 
We set $j' = \lfloor \log |z|/\log(1/p) \rfloor$. Since 
$(1- e^{-w})/w = 1 + O(w)$ it directly follows that
\[
\prod_{j=j'+1}^{\infty} \left( \frac{1 - e^{-qp^jz}}{qp^jz} \right) = O(1).
\]
Furthermore, if $\Re(z) \ge \eta$ it follows that 
$|1 - e^{-qp^jz}| \le 1 + e^{-qp^j \eta}$ which implies that 
$\prod_{j=0}^{j'} \left( {1 - e^{-qp^jz}} \right) = O(1)$.
Hence we obtain (\ref{eqDep1-new})).
\[
\Po{\xi}(z) = O\left( |z|\, q^{-j'-1} p^{-j'(j'+1)/2} |z|^{-j'}\right) 
= O\left(  |z|^{3/2+\log q / \log p}
        exp\left[ -\frac{\log^2 |z|}{2\log(1/p)} \right]   \right).
\]
It also follows (by considering Cauchy's formula as in \cite{jacquetszpa1997}) that
all derivatives have similar extimates:
\begin{equation}\label{eqDep1-new-der}
\Po{\xi}^{(k)}(z) = O\left( k! |z|^{3/2+\log q / \log p-k}
        \exp\left[ -\frac{\log^2 |z|}{2\log(1/p)} \right] \right) 
\end{equation}
uniformly for $\Re(z) \ge \eta$.

It remains to prove a condition of the form (\ref{eqDep2}):
    \begin{equation}\label{eqDep2-new}
        |e^{z}\Po{\xi}(z)|
        \leq Ce^{\alpha|z|},
    \end{equation}
    for some positive constants $C$ and $\alpha$ with 
$\alpha < 1$, for $\theta\le |\Arg(z)| \le \pi$
    We will choose $\theta$ such that $\cos(\theta) < 1/2$.
    which we require in order to prove the desired bound for 
$z$ outside the cone but with $\Re(z) > 0$.
This can be proved following the steps of \cite{magnerspa2015}, and we leave
details for the reader.`

Finally by using (\ref{eqDep1-new}) and (\ref{eqDep2-new}) 
together with the method used in 
\cite{jacquetszpa1997} we obtain that 
\begin{eqnarray*}
\ell! \xi_\ell &=& \Po{\xi}(\ell) + 
O\left(  \ell^{1/2+\log q / \log p}
        \exp\left[ -\frac{\log^2 \ell}{2\log(1/p)} \right] \right) \\
        &= &\Po{\xi}(\ell) \left( 1 + O(\ell^{-1}) \right).
\end{eqnarray*}        
This completes the depoissonization proof for $\Po{\xi}(z)$.


Summing up we obtain
\begin{eqnarray*}
\xi_\ell(n) &= &\frac{\Po{\xi}(\ell)}{\ell!}  \left( 1 + O(\ell^{-1}) \right)+ 
O\left( \frac{p^{n-\ell} + (q/p)^{n-\ell}}{(\ell-1)!} \right) \\
&=& \frac{\Po{\xi}(\ell)}{\ell!} \Theta(1)
\end{eqnarray*}
if $\log^2 \ell = o(n-\ell)$. 
Hence, plugging this into (\ref{S_ellDefinition}) and setting $k=n-\ell$, we arrive at  
    \begin{eqnarray*}
        \mu_{n,k} =
        \frac{n!}{(n-k)!} p^{k^2/2 + k/2}q^k (n-k)^{3/2+\frac{\log q}{\log p}} 
            \exp\left( -\frac{\log^2(n-k)}{2\log(1/p)}  \right) \Theta(1),
    \end{eqnarray*}
    if $\log^2(n-k)) = o(k)$.

\subsection{The Symmetric case $p = q = \frac 12$}

If $p = q = \frac 12$ we can process almost in the same way as above. 
Most importantly the recurrence for $\xi_\ell(n)$ simplifies to 
\[
\xi_\ell(n)(1 - 2^{1-n}) = \sum_{J=1}^\ell \frac{\xi_{\ell+1-J}}{J!} 2^{2-\ell}.
\]
Similarly to the above it now follows that 
$\xi_\ell(n) = \xi_\ell + O(2^{-(n-\ell)}/(\ell-1)!)$,
where $\xi_\ell = \lim_{n\to\infty} \xi_\ell(n)$. 
Furthermore the Poisson generating function
$\Po{\xi}(z)$ of $\ell! \xi_\ell$ is now given by
\[
\Po{\xi}(z) = z^2 \prod_{j=1}^\infty \frac{1-e^{-z/2^j}}{z/2^j}.
\]
By using the same techniques aa above it, thus, follows that 
\[
\xi_\ell = \frac{\ell^{5/2}}{\ell !}  
\exp\left( -\frac{\log^2\ell}{2\log 2}  \right) \Theta(1).
\]
and consequently
    \begin{eqnarray*}
        \mu_{n,k} =
         \frac{n!}{(n-k)!} 2^{-k^2/2 - 3k/2} (n-k)^{5/2} 
            \exp\left( -\frac{\log^2(n-k)}{2\log 2}  \right) \Theta(1),
    \end{eqnarray*}
    if $\log^2(n-k)) = o(k)$. This is consistent with the case $p > q$.

\subsection{Uniform bounds for $p\to 1$}.

We finally provide upper bounds for $\xi_\ell$ that are uniform as $p\to 1$.
(Actually we show more or less that (\ref{sum1}) gives also an upper bound.)

For convenience we work with $\xi_{\ell+1}$ instead of $\xi_\ell$ and recall that
the generating function is given by
\[
X(z) = \sum_{\ell \ge 0} \xi_\ell z^\ell = \prod_{j\ge 0} \frac{e^{qp^jz} -1}{qp^j z}.
\]
First suppose that $|z| \le 1/(1-p) = 1/q$ Then we have $|qp^j z|\le 1$ for all $j\ge 0$
and by using the approximation $e^x = 1 + x + x^2/2 + O(x^3)$ we directly obtain
the uniform representation
\begin{equation}\label{eqxiappr0}
X(z) = \sum_{\ell \ge 0} \xi_\ell z^\ell = e^{z/2 + O(qz^2)}. \qquad  (|qz| \le 1)
\end{equation}
This representation can be also used to obtain asymptotics for $\xi_{\ell+1} = \xi_{\ell+1}(p)$
if $\ell \le 1/(1-p) = 1/q$. We just apply a usual saddle point asymptotics on the circle
$|z| = \ell$ on the Cauchy integral
\[
\xi_{\ell +1} = \frac 1{2\pi i} \int_{|z|= \ell} X(z) z^{-ell-1}\, dz
\]
to obtain
\begin{equation}\label{eqxiappr}
\xi_{\ell+1}(p) = \frac 1{2^{\ell(1+O(q\ell)))} \ell!} \qquad (\ell \le 1/q).
\end{equation}
In particular if $\ell$ is fixed then (\ref{eqxiappr}) implies that $\xi_{\ell+1}(1) = 1/(2^\ell \ell!)$.

If $|zq|\ge 1$ then the situation is different. Here we restrict to the case of positive real $z$,
since this is sufficient to obtain upper bounds. We define $j_0$ by
\[
j_0 = \left\lceil \frac{\log(q z)}{\log(1/p)} \right\rceil
\]
and distinguish between the case $j\ge j_0$ and $j< j_0$. In the first case we $qp^j z \le 1$ so that
\[
\prod_{j\ge j_0} \frac{e^{qp^jz} -1}{qp^j z} = e^{p^{j_0} z/2 + O(p^{2j_0} q z^2) } = e^{O(1/q)}.
\]
We note that the implicit constant can be chosen to be arbitrarily close to $\frac 12 + \frac 16 = \frac 23$.
For the remaining product we have
\[
\prod_{j < j_0} \frac{e^{qp^jz} -1}{qp^j z} \le \prod_{j < j_0} \frac{e^{qp^jz}}{qp^j z} \le \frac{ e^{z(1-p^{j_0})}}{(qz)^{j_0} p^{j_0(j_0-1)/2}  } \le 
e^{z - \frac 1q - \frac{\log^2(qz)}{2\log(1/p)}}
\]
Thus we obtain 
\begin{equation}\label{eqxiappr2}
X(z) \le e^{z - \frac{\log^2(qz)}{2\log(1/p)}}. \qquad  (qz > 1).
\end{equation}
From this estimate we obtain, by the way, a uniform estimate for 
\begin{equation}\label{eqxiappr3}
\xi_{\ell+1} \le \frac {X(\ell)}{\ell^\ell} \le \frac{C}{\ell!}
\end{equation}
uniformly for $\frac 12 \le p \le 1$.


\section{Precise Analysis of the Lower Bound}

We present here a detailed and precise analysis for the lower bound of $\tilde G_k(n)$, 
where $k = \log_{1/p} \log n + \Psi_L(n)$ with $\Psi_L(n) = \frac 12 (1- \epsilon) \log_{p/q} \log n$.

We recall that $j_0 = \lfloor j^* +\frac 12\rfloor$ is defined as the closest integer to $j^*$, where $j^*$ is the solution of the
equation
\[
(q/p)^{j^*} (k - j^*) = \frac{\log (1/p)}{\log(p/q)} (j^* - \Psi_L(n)).
\]
Furthermore we set
\[
\overline r_0 = (q/p)^{j_0} (k - j_0) \quad \mbox{and} \quad \overline r_1 =  \frac{\log (1/p)}{\log(p/q)} (j_0 - \Psi_L(n)).
\]
and
\[
r_1(j) =  \frac{\log (1/p)}{\log(p/q)} (j - \Psi_L(n)) = \overline r_1 + \frac{\log (1/p)}{\log(p/q)}(j-j_0).
\]
Note that $\sqrt{q/p} \le \overline r_0 / \overline r_1 \le \sqrt{p/q}$ and that 
\[
n p^{k-j} (q/p)^{r_1(j)} = p^{\Psi_L(n) - j} (q/p)^{r_1(j)} = 1.
\]
We will also make use of the following abbreviation
\begin{equation}\label{eqF0}
F_0 :=  p^{j_0(j_0+1)2} q^{j_0-1} n^{j_0} p^{j_0(k-j_0)} \frac{ \overline r_0^{\overline r_1}} {\Gamma(\overline r_1 +1)}.
\end{equation}

Next we recall that we represent $\tilde G_k(n)$ (see (\ref{G_kExplicitFormHLowerBound-2})) by 
\[
 \Po{G}_k(n) = 
     \sum_{j=0}^k \sum_{m \ge j} \overline\kappa_{m,j} \overline \nu_{m,j}  
    + O\left( n^{j_0} T(-j_0)^{k-j_0}p^{j_0(j_0+1)/2}q^{j_0} \left( p^{j_0} + (q/p)^{j_0} \right) \right),
\]
where
\[
\overline \nu_{m,j} := - C_*(p) m! p^{j(j-1)/2} q^{j-1} \xi_{m-j+1}  .
\]
and 
\begin{eqnarray*}
\overline \kappa_{m,j} =  \frac{p^m n^m}{m!} \sum_{r=0}^{k-j} {k-j \choose r}  
p^{m(k-j-r)} q^{mr}  
\left(  e^{-np^{k-j-r}q^r} - \sum_{\ell\le j_0-m} \frac{(-n)^\ell}{\ell!} (p^{k-j-r}q^{r})^\ell   \right)
\end{eqnarray*}
for $j\le j_0$ and $j\le m\le j_0$ and otherwise
\begin{eqnarray*}
\overline \kappa_{m,j} =  \frac{p^m n^m}{m!} \sum_{r=0}^{k-j_0} {k-j \choose r} 
p^{m(k-j-r)} q^{mr)}   e^{-np^{k-j-r}q^r}.
\end{eqnarray*}

We also split the above into several parts:
\[
T_1 :=  \sum_{j> j_0} \sum_{m\ge j}  \overline\kappa_{m,j} \overline \nu_{m,j}, \quad
T_2 :=  \sum_{j\le  j_0} \sum_{m > j_0} \overline \kappa_{m,j} \overline \nu_{m,j}, \quad
T_3 :=  \sum_{j\le  j_0} \sum_{m=j}^{j_0} \overline \kappa_{m,j} \overline \nu_{m,j}. 
\]
Moreover we note that the exponential function $e^{-np^{k-j-r}q^r}
= e^{-(q/p)^{r-r_1(j)}}$ behaves completely differently 
for $r\le r_1(j)$ and for $r> r_1(j)$
where $r_1(j) = (j- \psi_(n))\frac{\log(1/p)}{\log(p/q)}$.
Hence it is convenient to split $T_3$ into three parts $T_{30} + T_{31}+T_{32}$, where the $T_{30}$ and $T_{31}$ correspond
to the terms with $r\le r_1(j)$ and $T_{32}$ for those with $r> r_1(j)$.
$T_{30}$ involves the exponential function $e^{-np^{k-j-r}q^r}$ whereas
$T_{31}$ takes care of the polynomial sum
$\sum_{\ell\le j_0-m} \frac{(-n)^\ell}{\ell!} (p^{k-j-r}q^{r})^\ell$.

What remains is a more detailed analysis of the sums $T_1,T_2,T_{30}, T_{31}, T_{32}$.

\subsection{Representation of the terms $T_1,T_2,T_{30}, T_{31}, T_{32}$}

\subsubsection{The term $T_1$.}

We recall that 
\begin{eqnarray*}
T_1 &=& \sum_{j> j_0} \sum_{m\ge j} \overline\kappa_{m,j} \overline \nu_{m,j} \\
&=& - C_*(p) \sum_{j> j_0} \sum_{m\ge j} p^{j(j-1)/2} q^{j-1} \xi_{m-j+1} p^m n^m
\sum_{r=0}^{k-j} {k-j \choose r} p^{m(k-j-r)} q^{mr} 
e^{-n p^{k-j-r} q^r}
\end{eqnarray*}
We use now the substitutions $j=j_0 + J$ and $m = j+ L = j_0 + J + L$, 
where $J > 0$ and $L\ge 0$. Furthermore by using approximation 
${k-j \choose r} \sim (k-j)^r/r! \sim (k-j_0)^r/r!$ we obtain
\begin{eqnarray*}
T_1 &\sim & - C_*(p)  p^{j_0(j_0+1)2} q^{j_0-1} n^{j_0} p^{j_0(k-j_0)} 
\sum_{J > 0} \sum_{L\ge 0} p^{J (J+1)/2} q^{J} \xi_{L+1} p^L \\
&&\qquad \times \sum_{r} \frac{ {\overline r_0}^r } {r!}  (q/p)^{(L+J)(r-r_1(j))} e^{-(q/p)^{r-r_1(j)}} \\
&\sim&  - C_*(p)F_0  \cdot\sum_{J > 0} p^{J (J+1)/2} q^{J}  
\left( \frac {\overline r_0}{\overline r_1} \right)^{J \frac{\log(1/p)}{\log(p/q)}}
  \sum_{L\ge 0} 	\xi_{L+1} p^L \\
&&\qquad \times   
\sum_r (q/p)^{(L+J)(r-r_1(j))} \left( \frac {\overline r_0}{\overline r_1} \right)^{r-r_j(j)}   e^{-(q/p)^{r-r_1(j)}},
\end{eqnarray*}
where $F_0$ is given in (\ref{eqF0}). 
Thus, if we define (with the implicit notation $q = 1-p$)
\begin{eqnarray}
C_1(p,u,v) &=& \sum_{J > 0} p^{J (J+1)/2} q^{J}  
u^{J \frac{\log(1/p)}{\log(p/q)}}
  \sum_{L\ge 0} 	\xi_{L+1} p^L \label{eqC1puv}   \\
&&\qquad \times   
\sum_{R\in \mathbb{Z}} \left( (q/p)^{(L+J)} u \right)^{R - v -J\frac{\log(1/p)}{\log(p/q)}}   e^{-(q/p)^{R - v -J\frac{\log(1/p)}{\log(p/q)}}}  \nonumber
\end{eqnarray}
we obtain
\[
T_1 \sim - C_*(p)\,F_0\,  C_1\left( p, \frac{\overline r_0}{\overline r_1}, \langle \overline r_1 \rangle \right),
\]  
where $\langle x \rangle = x - \lfloor x \rfloor$ denotes the fractional part of a real number $x$.
Note that we have substituted $r-r_1(j)$ by
\begin{eqnarray*}
r - r_1(j) &=& (r-\lfloor \overline r_1 \rfloor) - \langle \overline r_1 \rangle  + (\overline r_1 - r_1(j)) \\
&=& R - v - J\frac{\log(1/p)}{\log(p/q)}
\end{eqnarray*}

\subsubsection{The term $T_2$.}

The term $T_2$ can be handled almost in the same way as $T_1$.
By using the representation
\begin{eqnarray*}
T_2 &=& \sum_{j\le j_0} \sum_{m > j_0} \overline\kappa_{m,j} \overline \nu_{m,j} \\
&=& - C_*(p) \sum_{j\le j_0} \sum_{m > j_0} p^{j(j-1)/2} q^{j-1} \xi_{m-j+1} p^m n^m
\sum_{r=0}^{k-j} {k-j \choose r} p^{m(k-j-r)} q^{mr} 
e^{-n p^{k-j-r} q^r}
\end{eqnarray*}
and the same substitutions as above, $j = j_0+J$, $m = j+ L = j_0 + J + L$, 
where $J \le 0$ and $L\ge -J$, we obtain
\[
T_2 \sim - C_*(p)\,F_0\,  C_2\left( p, \frac{\overline r_0}{\overline r_1}, \langle \overline r_1 \rangle \right),
\]  
where
\begin{eqnarray}
C_2(p,u,v) &=& \sum_{J \le 0} p^{J (J+1)/2} q^{J}  
u^{J \frac{\log(1/p)}{\log(p/q)}}
  \sum_{L> -J} 	\xi_{L+1} p^L \label{eqC2puv}   \\
&&\qquad \times   
\sum_{R\in \mathbb{Z}} \left( (q/p)^{(L+J)} u \right)^{R - v -J\frac{\log(1/p)}{\log(p/q)}}   e^{-(q/p)^{R - v -J\frac{\log(1/p)}{\log(p/q)}}}.  \nonumber
\end{eqnarray}

\subsubsection{The term $T_{30}$.}

Next we consider $T_{30}$ that is given by
\begin{eqnarray*}
T_{30}
&=& - C_*(p) \sum_{j\le j_0} \sum_{m= j}^{j_0} p^{j(j-1)/2} q^{j-1} \xi_{m-j+1} p^m n^m
\sum_{r\le r_j(j)}^{k-j} {k-j \choose r} p^{m(k-j-r)} q^{mr} 
e^{-n p^{k-j-r} q^r}.
\end{eqnarray*}
Here we obtain (again with the same substitutions as above)
\[
T_{30} \sim - C_*(p)\,F_0\,  C_{30}\left( p, \frac{\overline r_0}{\overline r_1}, \langle \overline r_1 \rangle \right),
\]  
where
\begin{eqnarray}
C_{30}(p,u,v) &=& \sum_{J \le 0} p^{J (J+1)/2} q^{J}  
u^{J \frac{\log(1/p)}{\log(p/q)}}
  \sum_{L=0}^{-J} 	\xi_{L+1} p^L \label{eqC30puv}   \\
&&\qquad \times   
\sum_{R\in \mathbb{Z}, R - v -J\frac{\log(1/p)}{\log(p/q)}\le 0 } 
\left( (q/p)^{(L+J)} u \right)^{R - v -J\frac{\log(1/p)}{\log(p/q)}}   e^{-(q/p)^{R - v -J\frac{\log(1/p)}{\log(p/q)}}}.  \nonumber
\end{eqnarray}

\subsubsection{The term $T_{32}$.}

The term $T_{32}$ that is given by
\begin{eqnarray*}
T_{32}
&=& - C_*(p) \sum_{j\le j_0} \sum_{m= j}^{j_0} p^{j(j-1)/2} q^{j-1} \xi_{m-j+1} p^m n^m
\sum_{r > r_j(j)}^{k-j} {k-j \choose r} p^{m(k-j-r)} q^{mr} \\
&& \qquad \qquad \times \left( e^{-n p^{k-j-r} q^r} - \sum_{\ell\le j_0-m} \frac{(-n p^{k-j-r} q^r)^\ell}{\ell !}   \right)
\end{eqnarray*}
Here we get
\[
T_{32} \sim - C_*(p)\,F_0\,  C_{32}\left( p, \frac{\overline r_0}{\overline r_1}, \langle \overline r_1 \rangle \right),
\]  
where
\begin{eqnarray}
C_{32}(p,u,v) &=& \sum_{J \le 0} p^{J (J+1)/2} q^{J}  
u^{J \frac{\log(1/p)}{\log(p/q)}}
  \sum_{L=0}^{-J} 	\xi_{L+1} p^L \label{eqC32puv}   \\
&&\qquad \times   
\sum_{R\in \mathbb{Z}, R - v -J\frac{\log(1/p)}{\log(p/q)} > 0 } 
\left( (q/p)^{(L+J)} u \right)^{R - v -J\frac{\log(1/p)}{\log(p/q)}}  \\
&& \qquad \qquad \times \left(  e^{-(q/p)^{R - v -J\frac{\log(1/p)}{\log(p/q)}}} 
- \sum_{\ell=0}^{-J-L} \frac{(-1)^\ell}{\ell!} (q/p)^{(R - v -J\frac{\log(1/p)}{\log(p/q)})\ell} 
  \right)  \nonumber
\end{eqnarray}

\subsubsection{The term $T_{31}$.}

The term $T_{31}$ is the most interesting one.
It can be written as
\[
T_{31} = \sum_{K\ge 0} S_K.
\]
where we have to distinguish between the term $S_0$ and the terms $S_K$ for $K\ge 1$,

The term $S_0$ is given by 
\begin{eqnarray*}
S_0&=&  - C_*(p)p^{j_0(j_0+1)2} q^{j_0-1} n^{j_0} p^{j_0(k-j_0)} \sum_{L,M \ge 0} \xi_{L+1}\frac{(-1)^{M}}{M!} p^{((L+M)^2 + L-M)/2} q^{-L-M} \\
&& \qquad \times \sum_{r_1(j_0-M-L)< r \le \overline r_1} {k-j_0+L+M \choose r } \left( \frac qp \right)^{j_0r}.
\end{eqnarray*}
Since
\[
\sum_{r_1(j_0-M-L)< r \le \overline r_1} {k-j_0+L+M \choose r } \left( \frac qp \right)^{j_0r} 
\sim \frac{ \overline r_0^{\overline r_1}} {\Gamma(\overline r_1 +1)} 
\sum_{-(M+L) \frac{\log(1/p)}{\log(p/q)} \le r-\overline r_1 \le 0 } \left( \frac{\overline r_0}{\overline r_1} \right)^{r-\overline r_1}
\]
we, thus, obtain the representation
\[
S_0 \sim - C_*(p)\,F_0\,  C_{31,0}\left( p, \frac{\overline r_0}{\overline r_1}, \langle \overline r_1 \rangle \right),
\]  
where
\begin{eqnarray}
C_{31,0}(p,u,v) &=& 
\sum_{L,M \ge 0} \xi_{L+1}\frac{(-1)^{M}}{M!} p^{((L+M)^2 + L-M)/2} q^{-L-M} \label{eqC310puv}     \\
&&\qquad \times 
\sum_{-(M+L) \frac{\log(1/p)}{\log(p/q)} +v \le R \le 0 } u^{R-v}.
\nonumber
\end{eqnarray}
It is also convenient to rewrite this also as a sum over $J = -M-L\le 0$ and $0\le L \le -J$:
\begin{eqnarray}
C_{31,0}(p,u,v) &=& 
\sum_{J\le 0}\sum_{L=0}^{-J} \xi_{L+1}\frac{(-1)^{-J-L}}{(-J-L)!} p^{ J(J+1)/2 + L} q^{J} \label{eqC310puv-2}     \\
&&\qquad \times 
\sum_{ J \frac{\log(1/p)}{\log(p/q)} +v \le R \le 0 } u^{R-v}.
\nonumber
\end{eqnarray}

For $K\ge 1$ the terms $S_K$ are given by 
\begin{eqnarray*}
S_K&=& C_*(p)p^{j_0(j_0+1)2} q^{j_0-1} n^{j_0} p^{j_0(k-j_0)}   \left( \frac qp \right)^{K\overline r_1} \\
&&\qquad \times  \sum_{L\ge 0, \, M\ge K} \xi_{L+1}\frac{(-1)^{M-K}}{(M-K)!} p^{((L+M)^2 + L-M)/2-K(L+M)} q^{-L-M} \\
&&\qquad \qquad \times \sum_{r\le r_1(j_0-M-L)} {k-j_0+M+L \choose r } \left( \frac qp \right)^{(j_0-K)r}.
\end{eqnarray*}
The sum over $r$ together with the factor $\left( q/p \right)^{K\overline r_1}$ can be approximated by
\begin{eqnarray*}
&& \left( \frac qp \right)^{K\overline r_1} \sum_{r\le r_1(j_0-M-L)} {k-j_0+M+L \choose r } \left( \frac qp \right)^{(j_0-K)r} \\
&&\sim \frac{ \overline r_0^{\overline r_1}} {\Gamma(\overline r_1 +1)} 
\sum_{r\le \overline r_1 - (M+L)  \frac{\log(1/p)}{\log(p/q)} } 
\left( \frac{\overline r_0}{\overline r_1} \right)^{r-\overline r_1} \left( \frac qp \right)^{-K(r-\overline r_1)}.
\end{eqnarray*}
Thus, we have for $K\ge 1$
\[
S_K \sim C_*(p)\,F_0\,  C_{31,K}\left( p, \frac{\overline r_0}{\overline r_1}, \langle \overline r_1 \rangle \right),
\]  
where
\begin{eqnarray}
C_{31,K}(p,u,v) &=& 
\sum_{L\ge 0, \, M\ge K} \xi_{L+1}\frac{(-1)^{M-K}}{(M-K)!} p^{((L+M)^2 + L-M)/2-K(L+M)} q^{-L-M}  \nonumber     \\
&&\qquad \times 
\sum_{R\le v - (M+L)  \frac{\log(1/p)}{\log(p/q)} } 
 \left( u \left(\frac qp\right)^{-K} \right)^{R-v}.
\label{eqC31Kpuv}
\end{eqnarray}
As above we can rewrite this as a sum over $J\le -K$ and $0\le L \le -J-K$:
\begin{eqnarray}
C_{31,K}(p,u,v) &=& 
\sum_{J\le -K} \sum_{L=0}^{-J-K} \xi_{L+1}\frac{(-1)^{-J-L-K}}{(-J-L-K)!} p^{J(J+1)/2 + L +JK} q^{J} \nonumber     \\
&&\qquad \times 
\sum_{R\le v+J \frac{\log(1/p)}{\log(p/q)} } 
 \left( u \left(\frac qp\right)^{-K} \right)^{R-v}.
\label{eqC31Kpuv-2}
\end{eqnarray}

\subsection{Asymptotic analysis for $p\to \frac 12$.}

It $p\to \frac 12 $ then $q/p \to 1$ and we can write
\[
\frac qp = e^{-\eta},
\]
where $\eta \to 0$. Note that $p> \frac 12$ implies that $\eta > 0$.

Furthermore, we use the abbrevation $u = \overline r_0/\overline r_1$ and $v = \{ \overline r_1\}$.
Note that $ e^{\eta/2}= \sqrt{q/p}\le u \le \sqrt{p/q} = e^{\eta/2}$ or equivalently $-\frac 12 \le \frac 1\eta \log u \le \frac 12$.
We will therefore also use the abbreviation $\tilde u = \frac 1\eta \log u$

Finally we mention that $\xi_\ell = \xi_\ell(p)$ depend smoothly on $p \in [\frac 12, 1]$.
Furthermore we have a uniform upper bound $|\xi_\ell(p)|\le C/(\ell-1)!$.

\subsubsection{The term $T_1$.}

Since the factors $C_*(p)$ and $F_0$ are present in all terms it suffices
to study the sum $C_1(p,u,v)$ (defined in (\ref{eqC1puv})) in order to study the 
asymptotic behavior of $T_1$. 

We start with the sum over $R$. The first observation is that for $\eta \to 0$ we can replace the
sum by an integral, that is, we have for fixed integers $L,J$, as $\eta\to 0$,
\begin{eqnarray*}
&& \sum_{R\in \mathbb{Z}} \left( (q/p)^{(L+J)} u \right)^{R - v -J\frac{\log(1/p)}{\log(p/q)}}   e^{-(q/p)^{R - v -J\frac{\log(1/p)}{\log(p/q)}}} \\
&&\sim \int_{-\infty}^\infty 
\left( (q/p)^{(L+J)} u \right)^{t}   e^{-(q/p)^{t}} \, dt \\
&& = \frac 1\eta \int_{-\infty}^\infty e^{-\left(M - \tilde u \right) t} e^{-e^{-t}}\, dt.
\end{eqnarray*}
This also implies that the leading asymptotic term does not depend on $v = \{ \overline r_1\}$.
Further note that $\tilde M = M - \frac 1\eta \log u = L+J - \tilde u \ge \frac 12$ so that the integral converges
and by using the substitution $w = e^{-t}$ we obtain
\[
\int_{-\infty}^\infty e^{-\tilde M t} e^{-e^{-t}}\, dt 
= \int_0^\infty w^{\tilde M -1} e^{-w}\, dw = \Gamma(\tilde M).
\]
This finally shows that, as $p\to \frac 12$ (or equivalently as $\eta = \log(p/q) \to 0$),
\begin{equation}\label{eqC1asymp}
C_1(p,u,v) \sim \frac 1\eta 
\sum_{J > 0} 2^{-J (J+1)/2 - J + J\tilde u}   
  \sum_{L\ge 0} 	\xi_{L+1}(1/2)\, 2^{-L}\, \Gamma\left( J + L - \tilde u\right).
\end{equation}

\subsubsection{The term $T_2$.}

Here we analyze the term  $C_2(p,u,v)$ (defined in (\ref{eqC2puv})) which looks very similar
to $C_1(p,u,v)$, and in fact we can do almost the same considerations. First of all we again
have that $J+L - \frac 1\eta \log u > 0$. Thus, the appearing sums and integrals are convergent.
This finally shows that, as $p\to \frac 12$ (or equivalently as $\eta = \log(p/q) \to 0$),
\begin{equation}\label{eqC2asymp}
C_2(p,u,v) \sim \frac 1\eta 
\sum_{J \le 0} 2^{-J (J+1)/2 - J + J\tilde u}   
  \sum_{L > -J} 	\xi_{L+1}(1/2)\, 2^{-L}\, \Gamma\left( J + L - \tilde u\right).
\end{equation}

\subsubsection{The term $T_{30}$.}

The representation of $C_{30}(p,u,v)$ (given in (\ref{eqC30puv})) is quite similar to
$C_1(p,u,v)$ or $C_2(p,u,v)$. The main difference is that $L+J \le 0$ and that the sum over $R$ can now be approximated by
the incomplete Gamma function:
\begin{eqnarray*}
&& \sum_{R\in \mathbb{Z}, R - v -J\frac{\log(1/p)}{\log(p/q)}\le 0 } 
\left( (q/p)^{(L+J)} u \right)^{R - v -J\frac{\log(1/p)}{\log(p/q)}}   e^{-(q/p)^{R - v -J\frac{\log(1/p)}{\log(p/q)}}}  \\
&& \qquad \qquad \sim \int_{-\infty}^0 
\left( (q/p)^{(L+J)} u \right)^{t}   e^{-(q/p)^{t}} \, dt \\
&&  \qquad \qquad = \frac 1\eta \int_{-\infty}^0 e^{-\left(L+J - \tilde u\right) t} e^{-e^{-t}}\, dt \\
&& \qquad \qquad  = \frac 1\eta \int_1^\infty w^{L+J - \tilde u} e^{-w}\, dw \\
&&  \qquad \qquad  = \frac 1\eta \Gamma( L+J - \tilde u, 1),
\end{eqnarray*}
where 
\[
\Gamma(s,z) := \int_z^\infty w^{s-1} e^{-w}\, dw.
\]
This leads to
\begin{equation}\label{eqC30asymp}
C_{30}(p,u,v) \sim \frac 1\eta 
\sum_{J \le 0} 2^{-J (J+1)/2 - J + J\tilde u}   
  \sum_{L =0}^{-J} 	\xi_{L+1}(1/2)\, 2^{-L}\, \Gamma\left( J + L - \tilde u, 1\right).
\end{equation}

\subsubsection{The term $T_{32}$.}

We start again with the analysis of the $R$-sum. We note that $L+J\le 0$ so that we obtain
\begin{eqnarray*}
&&\sum_{R\in \mathbb{Z}, R - v -J\frac{\log(1/p)}{\log(p/q)} > 0 } 
\left( (q/p)^{(L+J)} u \right)^{R - v -J\frac{\log(1/p)}{\log(p/q)}}  \\
&& \qquad \qquad \times \left(  e^{-(q/p)^{R - v -J\frac{\log(1/p)}{\log(p/q)}}} 
- \sum_{\ell=0}^{-J-L} \frac{(-1)^\ell}{\ell!} (q/p)^{(R - v -J\frac{\log(1/p)}{\log(p/q)})\ell} 
  \right)  \\
&& =   \sum_{R\in \mathbb{Z}, R - v -J\frac{\log(1/p)}{\log(p/q)} > 0 } 
\left( (q/p)^{(L+J-\tilde u)} \right)^{R - v -J\frac{\log(1/p)}{\log(p/q)}}  \\
&& \qquad \qquad \times  \sum_{\ell\ge -J-L+1} \frac{(-1)^\ell}{\ell!} (q/p)^{(R - v -J\frac{\log(1/p)}{\log(p/q)})\ell} \\
&&\sim \sum_{\ell\ge -J-L+1} \frac{(-1)^\ell}{\ell!} 
\int_0^\infty e^{-\eta (\ell + J+L - \tilde u) t}\, dt \\
&& =  \sum_{\ell\ge -J-L+1} \frac{(-1)^\ell}{\ell!}  \frac 1{\ell + J+L - \tilde u}
\end{eqnarray*}
This leads to 
\begin{equation}\label{eqC32asymp}
C_{32}(p,u,v) \sim \frac 1\eta 
\sum_{J \le 0} 2^{-J (J+1)/2 - J + J\tilde u}   
  \sum_{L =0}^{-J} 	\xi_{L+1}(1/2)\, 2^{-L} \sum_{\ell\ge -J-L+1} \frac{(-1)^\ell}{\ell!}  \frac 1{\ell + J+L - \tilde u}.
\end{equation}

\subsubsection{The term $T_{31}$.}

We consider first the behavior of $C_{31,0}(p,u,v)$ (that is given in (\ref{eqC310puv})).
Since
\[
\sum_{-(M+L) \frac{\log(1/p)}{\log(p/q)} +v \le R \le 0 } u^{R-v}
\sim \frac{1 - 2^{-(M+L)\tilde u}}{\eta\, \tilde u}
\]
we directly obtain
\begin{equation}\label{eqC310asymp}
C_{31,0}(p,u,v) \sim \frac 1\eta 
\sum_{L,M \ge 0} \xi_{L+1}(1/2) \frac{(-1)^{M}}{M!} 2^{-(L+M)^2/2 +L/2 + 3M/2}
\frac{1 - 2^{-(M+L)\tilde u}}{\tilde u}.
\end{equation}
As above it also convenient to rewrite the resulting sum into a sum in 
$J = -L-M \le 0$ and $0\le L \le -J$:
\begin{equation}\label{eqC310asymp-2}
C_{31,0}(p,u,v) \sim \frac 1\eta 
\sum_{J\le 0,\, 0\le L \le -J} \xi_{L+1}(1/2) \frac{(-1)^{-J-L}}{(-J-L)!} 2^{-J(J+1)/2 -J-L}
\frac{1 - 2^{J\tilde u}}{\tilde u}.
\end{equation}

Next suppose that $K\ge 1$. The terms $C_{31,K}(p,u,v)$ are given in (\ref{eqC310puv}).
Since 
\[
\sum_{R\le - (M+L)  \frac{\log(1/p)}{\log(p/q)} } 
 \left( u \left(\frac qp\right)^{-K} \right)^{R-v}
\sim \frac{p^{K(M+L)} p^{(M+L)\tilde u}}{\eta(K+\tilde u)}
\]
it follows that 
\begin{eqnarray}
&& C_{31,K}(p,u,v)  \label{eqC31Kasymp} \\
&& \sim \frac 1{\eta(K+\tilde u)}
\sum_{L\ge 0, \, M\ge K} \xi_{L+1}(1/2) \frac{(-1)^{M-K}}{(M-K)!} 2^{-(L+M)^2/2 + L/2+3M/2- (M+L)\tilde u }  \nonumber
\end{eqnarray}
and consequently
\begin{eqnarray}
&& \sum_{K\ge 1} C_{31,K}(p,u,v)  \label{eqC31asymp} \\
&& \sim \frac 1{\eta} \sum_{L\ge 0, \, M\ge 1} \xi_{L+1}(1/2) \, 2^{-(L+M)^2/2 + L/2 + 3M/2- (M+L)\tilde u } 
\sum_{K=1}^M \frac{(-1)^{M-K}}{(M-K)!(K+\tilde u)} \nonumber
\end{eqnarray}
Of course we can rewrite the resulting sum as a sum over $J,L$:
\begin{eqnarray}
&& \sum_{K\ge 1} C_{31,K}(p,u,v)  \label{eqC31asymp-2} \\
&& \sim \frac 1{\eta}\sum_{J\le -1, 0\le L\le -J-1} \xi_{L+1}(1/2)\, 2^{-J(J+1)/2 -J -L  +J\tilde u } 
\sum_{K=1}^{-J-L} \frac{(-1)^{-J-L-K}}{(-J-L-K)!(K+\tilde u)}  \nonumber
\end{eqnarray}

\subsubsection{The full behavior for $p\to \frac 12$.}

Summing up it follows that
\begin{eqnarray*}
\tilde G_k(n) &=& T_1 + T_2 + T_{30} + T_{32} + \sum_{K\ge 0} S_K  \\ 
 &\sim &  C_*(p)\,F_0\, ( - C_1 - C_2 - C_{30} - C_{32} - C_{31,0} + \sum_{K\ge 1} C_{31,K} ),
\end{eqnarray*}
where, as $p\to \frac 12$ or equivalently as $\eta \to 0$, all the terms
$C_1(p,u,v)$, $C_2(p,u,v)$, $C_{30}(p,u,v)$, $C_{32}(p,u,v)$, $C_{31,K}(p,u,v)$ ($K\ge 0$) behave
as $1/\eta$ with some factor that only depends on $\tilde u$. This means that we have 
\[
- C_1 - C_2 - C_{30} - C_{32} - C_{31,0} + \sum_{K\ge 1} C_{31,K} \sim \frac 1\eta h_1(\tilde u) 
\]
for some explicit function $h_1(\tilde u)$ that collects from the asymptotic relations
(\ref{eqC1asymp}), (\ref{eqC2asymp}), (\ref{eqC30asymp}), (\ref{eqC32asymp}), (\ref{eqC310asymp}), (\ref{eqC31asymp}).

It remains to check that $h_1(\tilde u)$ stays positive for all $\tilde u \in [-\frac 12, \frac 12]$.
It is clear that the dependence on $\tilde u$ is smooth in all appearing terms and that the negative derivative 
with respect to $\tilde u$ can be uniformly upper bounded, as seen in the 
following
table of derivatives.

\begin{center}
\begin{tabular}{|l|l|l|l|} \hline
$\tilde{u}$ & $h_1(\tilde{u})$ & $h_1(\tilde{u}) / h'_1(\tilde{u})$ & $h'_1(\tilde{u})$  \\ \hline
$-0.50$ & $1.37683018271327$ & $-0.722028017914153$ & $-1.90689301322511$ \\ 
$-0.40$ & $1.20276152989834$ & $-0.760013160991751$ & $-1.58255355516324$ \\ 
$-0.30$ & $1.05800806833013$ & $-0.802220048867141$ & $-1.31885019555944$ \\ 
$-0.20$ & $0.937149181875061$ & $-0.849393914518373$ & $-1.10331515902895$ \\
$-0.10$ & $0.835870082265573$ & $-0.902466357406207$ & $-0.926206362603876$ \\ 
$0.00$ & $0.358367943474688$ & $0.000915198138561305$ & $391.574161239056$ \\ 
$0.10$ & $0.678937477362699$ & $-1.03136470454952$ & $-0.658290393657834$ \\ 
$0.20$ & $0.618287879529247$ & $-1.11069248821028$ & $-0.556668822461859$ \\ 
$0.30$ & $0.566972485392761$ & $-1.20324708128446$ & $-0.471202045042585$ \\ 
$0.40$ & $0.523532363681955$ & $-1.31263743584976$ & $-0.398840037152404$ \\ 
$0.50$ & $0.486782979369433$ & $-1.44391680699806$ & $-0.337126749276828$ \\ \hline
\end{tabular}
\end{center}

Hence it sufficient to check positivity in a sufficiently fine grid which can be easily performed.
The following table gives some sample values:

\begin{center}
\begin{tabular}{|l|l|l|} \hline
$p$ & $\tilde{u}$ & $h_1(\tilde{u})$ \\
\hline
$0.50$ & $-0.50$ & $1.37683018271327$ \\ 
$0.50$ & $-0.45$ & $1.28574151187623$ \\ 
$0.50$ & $-0.40$ & $1.20276152989834$ \\ 
$0.50$ & $-0.35$ & $1.12708836544424$ \\ 
$0.50$ & $-0.30$ & $1.05800806833013$ \\ 
$0.50$ & $-0.25$ & $0.994884277261959$ \\
$0.50$ & $-0.20$ & $0.937149181875062$ \\
$0.50$ & $-0.15$ & $0.884295608451989$ \\
$0.50$ & $-0.10$ & $0.835870082265572$ \\
$0.50$ & $-0.05$ & $0.791466739676032$ \\
$0.50$ & $0.00$ & $0.580594753668194$ \\ 
$0.50$ & $0.05$ & $0.713309765274110$ \\ 
$0.50$ & $0.10$ & $0.678937477362699$ \\ 
$0.50$ & $0.15$ & $0.647342275661044$ \\ 
$0.50$ & $0.20$ & $0.618287879529247$ \\ 
$0.50$ & $0.25$ & $0.591561730562133$ \\ 
$0.50$ & $0.30$ & $0.566972485392761$ \\ 
$0.50$ & $0.35$ & $0.544347799045552$ \\ 
$0.50$ & $0.40$ & $0.523532363681955$ \\ 
$0.50$ & $0.45$ & $0.504386172111908$ \\ 
$0.50$ & $0.50$ & $0.486782979369433$ \\
\hline
\end{tabular}
\end{center}
\bigskip

\subsection{Asymptotic Analysis for $p\to 1$.}

Before we start with the analysis of the terms $T_1$, $T_2$ etc. we 
recall that
\[
\xi_{L+1}(1) = \frac 1{2^L L!}.
\]
and that 
\[
\sum_{L\ge 0} \xi_{L+1} z^L = e^{z/2 + O(qz^2)}, \qquad |z|\le 1/q
\]
which implies that 
\[
\xi_{L+1}(p) = \frac 1{2^{L(1+O(qL))} L!}.
\]
On the other hand we have a uniform estimate of the form
\[
\xi_{L+1} \le \frac C{L!}
\]

Next we observe that, as $p\to 1$, 
\[
\left( \frac pq \right)^{\frac{\log(1/p)}{\log(p/q)}} = 1 + O(q) \to 1.
\]
Hence, it also follows that
\[
u^{\frac{\log(1/p)}{\log(p/q)}} = 1 + O(q) \to 1.
\]

The strategy for the analysis in the case $p\to 1$ is to 
show that the terms $C_1$, $C_2$, $C_{31,0}$, and $C_{32}$ are
(at most) of order $e^{O(1/q)}$, whereas the $C_{30}$ and the sum $\sum_{K\ge 1} C_{31,0}$ are at
least of order  $e^{c/q  \log^2 q}$ for some positive constant $c> 0$. These terms have to 
be evaluated with more care.

\subsubsection{The term $T_1$.}

We consider the sum $C_1(p,u,v)$ (defined in (\ref{eqC1puv})) and start with the term $J=1$.
Then for every fixed $L\ge 0$ the sum over $R$ is dominated by (at most) two terms corresponding
to $R = 0$ and $R=1$:
\begin{eqnarray*}
&& \sum_{R\in \mathbb{Z}} \left( (q/p)^{(L+1)} u \right)^{R - v -J\frac{\log(1/p)}{\log(p/q)}}   e^{-(q/p)^{R - v -J\frac{\log(1/p)}{\log(p/q)}}}  \\
&& \sim \left( (q/p)^{(L+1)} u \right)^{- v -\frac{\log(1/p)}{\log(p/q)}}   e^{-(q/p)^{- v -\frac{\log(1/p)}{\log(p/q)}}} \\
&& + \left( (q/p)^{(L+1)} u \right)^{1 - v -J\frac{\log(1/p)}{\log(p/q)}}   e^{-(q/p)^{1 - v -\frac{\log(1/p)}{\log(p/q)}}}  \\
&& \sim \left( (q/p)^{(L+1)}u \right)^{- v }   e^{-(q/p)^{- v}} + \left( (q/p)^{(L+1)}u\right)^{1- v }   e^{-(q/p)^{1- v}}.
\end{eqnarray*}
Since (for $\delta \in \{0,1\}$) 
\begin{eqnarray*}
\sum_{L\ge 0} \xi_{L+1} p^L (q/p)^{(\delta-v)L} &= & e^{ \frac 12 p (q/p)^{\delta-v} + O( q (q/p)^{2(\delta -v)} ) }\\
&= & e^{\frac 12 q^{(\delta-v)} + O(q^{1+\delta-v} ) } 
\end{eqnarray*}
we thus obtain that term corresponding $J=1$ in $C_1(p,u,v)$ is asymptotically given by
\[
q (qu)^{-v} e^{-\frac 12 q^{-v}}  + q (qu)^{1-v} e^{-\frac 12 q^{1-v}} 
\]
For $J> 1$ the computations are almost the same but the (initial) factor $q^J$ makes them negligible 
compared to this term. As above we represent $u$ as $u = (p/q)^{\tilde u} \sim q^{-\tilde u}$
Thus, we finally have, as $p\to 1$,
\begin{equation}\label{eqC1asymp2}
C_1(p,u,v) \sim q^{1 -v(1-\tilde u)} e^{-\frac 12 q^{-v}}  + q^{1+ (1- v)(1-\tilde u)} e^{-\frac 12 q^{1-v}} 
\end{equation}
This term is trivially bounded.
 
\subsubsection{The term $T_2$.}  

The main difference between the terms $C_1(p,u,v)$ and  $C_2(p,u,v)$ (defined in (\ref{eqC1puv}) and (\ref{eqC2puv}))
is that $C_2(p,u,v)$ sums over $J\le 0$ and the {\it leading factor} is now $q^J$ that tends to infinity (for $J < 0$ as $p\to 1$).

Next we consider the sum over $R$ and observe that we (uniformly) have 
\begin{eqnarray*}
&&\sum_{R\in \mathbb{Z}} \left( (q/p)^{(L+J)} u \right)^{R - v -J\frac{\log(1/p)}{\log(p/q)}}   e^{-(q/p)^{R - v -J\frac{\log(1/p)}{\log(p/q)}}} \\
&&\qquad = O\left( \frac{(L + J-1)!}{\log(1/q)} \right).
\end{eqnarray*}
Furthermore since $\xi_{L+1} \le C/ L!$ and 
\[
\sum_{L > -J} \frac{(L+J-1)!}{L!} = O \left( \frac 1{(L+1)!} \right)
\]
it is sufficient to consider the sum
\[
\frac 1{\log(1/p)} \sum_{J\le 0} p^{J(J+1)/2}q^J \frac 1{(L+1)!}
\]
It is an easy exercise that the term $p^{J(J+1)/2}q^J /{(L+1)!}$ is maximal for $J \sim A/\log p \sim - A/q$ for some constant $A> 0$
and that it is of order  $e^{O(1/q)}$. Hence it follows that $C_2(p,u,v)$ is also upper bounded by $e^{O(1/q)}$, too.

\subsubsection{The term $T_{30}$.}

We first suppose that $v> 0$ and we also consider (first) the interval
\[
J\in \left[ - v\left( \frac{\log q}{\log p}-1\right) , 0 \right] \cap \mathbb{Z}
\]
for which the term corresponding to $R=0$ dominates the sum (note that $L+J\le 0$);
\begin{eqnarray*}
&&\sum_{R\in \mathbb{Z}, R - v -J\frac{\log(1/p)}{\log(p/q)}\le 0 } 
\left( (q/p)^{(L+J)} u \right)^{R - v -J\frac{\log(1/p)}{\log(p/q)}}   e^{-(q/p)^{R - v -J\frac{\log(1/p)}{\log(p/q)}}}  \\
&& \sim 
\left( (q/p)^{(L+J)} u \right)^{- v -J\frac{\log(1/p)}{\log(p/q)}}   e^{-(q/p)^{- v -J\frac{\log(1/p)}{\log(p/q)}}} 
\end{eqnarray*}
It is also clear that this term gets largest when $J$ equals
\[
J = J_{v,0} := - \left\lfloor v \left( \frac{\log q}{\log p} -1 \right) \right\rfloor.
\]
Interestingly the result will depend heavily on the rounding error
\[
\delta = - v \left( \frac{\log q}{\log p} -1 \right) - J_{v,0}.
\]
In particular the sum over $R$ is then asymptotically given by
\[
\left( (q/p)^{(L+J_{v,0})} u \right)^{- v -J_{v,0}\frac{\log(1/p)}{\log(p/q)}}   e^{-(q/p)^{- v -J_{v,0}\frac{\log(1/p)}{\log(p/q)}}} 
\sim e^{-1} (q/p)^{-\delta(J_{v,0}+L)}.
\]
Thus, we have to study next the following sum	
\begin{eqnarray*}
e^{-1} \sum_{L=0}^{-J_{v,0}} \xi_{L+1} p^L (q/p)^{-\delta(J_{v,0}+L)} &= & e^{-1} 
[z^{-J_{v,0}}] \frac{ e^{pz/2 + O(qz^2)} }{1 - (q/p)^\delta z} \\
&\sim & e^{-1} (q/p)^{-\delta J_{v,0}} e^{p(p/q)^\delta + O(q (p/q)^{2\delta})}.
\end{eqnarray*}
Finally, we observe that the (total) sum over  $J \in \left[ - v\left( \frac{\log q}{\log p}-1\right) , 0 \right]$ 
is dominated by the term corresponding to $J_{v,0}$ and is given by
\begin{eqnarray*}
&& e^{-1} (q/p)^{-\delta J_{v,0}} e^{p(p/q)^\delta + O(q (p/q)^{2\delta})} (q^{1/2} u)^{-v} p^{J_{v,0}^2/2} q^{J_{v,0}} \\
&& \sim e^{-1} (q^{1/2} u)^{-v} q^{-v^2+v} e^{(p/q)^\delta (1+o(1))} e^{(v(1-\delta)-v^2/2 \frac{\log^2 q}{\log(1/p)}} q^{\delta(2-3v)}
\end{eqnarray*}
For example, if $\delta = 0$ then we obtain (asymptotically)
\[
\sim e^{-1/2} (q^{1/2} u)^{-v} q^{-v^2+v} e^{(v-v^2/2) \frac{\log^2 q}{\log(1/p)}}.
\]

Next we consider 
\[
J\in \left[ - (v+1)\left( \frac{\log q}{\log p}-1\right) , (v+1)\left( \frac{\log q}{\log p}-1\right) \right) \cap \mathbb{Z}.
\]
In this range the term $p^{J^2/2}q^J$ attains its maximum if $J$ is close to $J_0 := -\frac{\log q}{\log p}$ and it will turn
out to be the essential range in this case. However, in order to be precise we set $J = \kappa J_0$ with $\kappa \in (v,v+1)$.
In this case the sum over $R$ is dominated by the term related to $R = -1$:
\[
\left( (q/p)^{L+\kappa J_0}u\right)^{-1-v+\kappa} e^{-(q/p)^{-1-v+\kappa}} (1+o(1))
\]
Thus, we are led to analyze the sum
\begin{eqnarray*}
&&\sum_{L=0}^{-\kappa J_0} \xi_{L+1} p^L (q/p)^{(-1-v+\kappa)(L+\kappa J_0)} = \\
&& [z^{-\kappa J_0}] X(z) \frac 1{1- (q/p)^{1+v-\kappa} z} = (q/p)^{(-1-v+\kappa)\kappa J_0} e^{q^{-1-v+\kappa}/2 + O(q^{1-2(1+v-\kappa)})}
\end{eqnarray*}
and consequently the term
\begin{eqnarray*}
&& p^{\kappa J_0(\kappa J_0+1)/2}q^{\kappa J_0} u^{\kappa J_0 \frac{\log p}{\log q}} (q/p)^{(-1-v+\kappa)\kappa J_0} 
u^{-1-v+\kappa} e^{q^{-1-v+\kappa}/2} \\
&& \sim e^{(\kappa - \kappa^2/2 - v) \frac{\log^2 q}{\log(1/p)} }  q^{-\kappa(\frac 32 + v - \kappa)} u^{-1-v} e^{q^{-1-v+\kappa}/2}
\end{eqnarray*}
which gets maximal for $\kappa = 1$ and has a local behavior of the form
\[
e^{(\frac 12 - v) \frac{\log^2 q}{\log(1/p)} }  q^{-\frac 12 -v} u^{-1-v} e^{q^{-v}/2}
e^{- \frac {\log(1/p)}2 (J - J_0)^2 }.
\]
Thus, summing over $J$ in this range we finally get
\[
\frac 1{\sqrt{2\pi \log(1/p)}} e^{(\frac 12 - v) \frac{\log^2 q}{\log(1/p)} }  q^{-\frac 12 -v} u^{-1-v} e^{q^{-v}/2}
\]

It is an easy exercise to show that the contributions coming from $J < (v+1)J_0$ are negligible compared to these terms.
Hence, we obtain for $v> 0$ as $p\to 1$
\begin{eqnarray*}
C_{30}(p,u,v) & \sim & e^{-1} (q^{1/2} u)^{-v} q^{-v^2+v} e^{(p/q)^\delta (1+o(1))} e^{(v(1-\delta)-v^2/2) \frac{\log^2 q}{\log(1/p)}} q^{\delta(2-3v)} \\
&~~+&  \frac 1{\sqrt{2\pi}} e^{(\frac 12 - v) \frac{\log^2 q}{\log(1/p)} }  q^{-1 -v} u^{-1-v} e^{q^{-v}/2}.
\end{eqnarray*}

Finally the case  $v= 0$ (or if $v$ is close to zero) can be handled in the same way, however, only the second term survives and so we get
\[
C_{30}(p,u,v) \le  \frac {1+o(1)}{\sqrt{2\pi}} e^{(\frac 12 - v) \frac{\log^2 q}{\log(1/p)} }  q^{-1 -v} u^{-1-v} e^{q^{-v}/2}.
\]

\subsubsection{The term $T_{32}$.}

The analysis of $C_{32}(p,u.v)$ is relatively easy. 
If $R - v -J\frac{\log(1/p)}{\log(p/q)}>0$ then
\[
(q/p)^{R - v -J\frac{\log(1/p)}{\log(p/q)}} \le 1
\] 
which implies that
\[
 e^{-(q/p)^{R - v -J\frac{\log(1/p)}{\log(p/q)}}} 
- \sum_{\ell=0}^{-J-L} \frac{(-1)^\ell}{\ell!} (q/p)^{(R - v -J\frac{\log(1/p)}{\log(p/q)})\ell} 
= O\left( \frac{(q/p)^{(R - v -J\frac{\log(1/p)}{\log(p/q)})(-J-L+1)} }{(-J-L+1)! }  \right).
\]
We first suppose that $v> 0$ and concentrate on $J\in (vJ_0,0]$. Then the relevant term in 
the sum over $R$ is the one with $R = 1$. Thus, the sum over $R$ is dominated (up to a constant) by
\[
\left( (q/p)^{(L+J)} u \right)^{1 - v -J\frac{\log(1/p)}{\log(p/q)}} 
\frac{(q/p)^{(R - v -J\frac{\log(1/p)}{\log(p/q)})(-J-L+1)} }{(-J-L+1)! } 
= \frac{((q/p)u)^{1 - v -J\frac{\log(1/p)}{\log(p/q)}} }{(-J-L+1)! }.
\]
Next, by using the upper bound $\xi_{L+1} \le C/L!$ it follows that 
\[
\sum_{L=0}^{-J} \xi_{L+1}p^L \frac{1}{(-J-L+1)! } 
\le \frac {2^{-J}}{(-J+1)!}.
\]
Thus, we are led to consider the sum
\[
\sum_{J\in (vJ_0,0]} p^{J(J+1)/2} q^J u^{J \frac{\log(1/p)}{\log(p/q)}} \frac {2^{-J}}{(-J+1)!}
((q/p)u)^{1 - v -J\frac{\log(1/p)}{\log(p/q)}}.
\]
As in the case of the analysis of $C_2(p,u,v)$ is follows that this sum is upper bounded
by $e^{O(1/q)}$ since the most significant term appears for $J \sim A/\log p$ (for some constant $A> 0$).
Finally it is an easy exercise to show that the sum over $J \le J_0 v$ is negligible.

If $v= 0$ then a similar analysis applies and we (certainly) have an upper bound of the form $e^{O(1/q)}$.

\subsubsection{The term $T_{31}$.}

We set
\[
I_0:= \left[ - v\left( \frac{\log q}{\log p}-1\right) , 0 \right) \cap \mathbb{Z}
\]
and for $M\ge 1$
\[
I_M:= \left[ - (v+M)\left(\frac{\log q}{\log p}-1\right) , - (v+M-1)\left(\frac{\log q}{\log p}-1\right) \right) \cap \mathbb{Z}
\]

First we consider the term $C_{31,0}(p,u,v)$ given in (\ref{eqC310puv-2}).
If $v> 0$ and $J\in I_0$ then $v + J \frac{\log(1/p)}{\log(p/q)} > 0$ which implies that 
the sum over $R$ is empty. Thus, we can assume that $J < - v\left( \frac{\log q}{\log p}-1\right)$.

By using the upper bound $\xi_{L+1} \le C/L!$ we obtain the upper bound
\[
\left|\sum_{L=0}^{-J} \xi_{L+1}\frac{(-1)^{-J-L}}{(-J-L)!} \right| \le C \frac {2^{-J}}{(-J)!}.
\]
Recall that the term
\[
p^{ J(J+1)/2} q^{J} \frac {2^{-J}}{(-J)!}
\]
is maximal for $J$ close to $J_{\max} = A/\log p$ (for some constant $A>0$) and the maximum value is or order $e^{O(1/q)}$.
Hence, if $v > A/\log(1/q)$ then $J_{\max} \in I_0$ and consequently we can upper bound $C_{31,0}$ trivially by
$e^{O(1/q)}$. Conversely, if $v \le A/\log(1/q)$ then $J_{\max}\in I_1$ and we also get an upper bound of 
the form $e^{O(1/q)}$. Note that the appearing $R$-sums are negligible compared to the leading term $e^{O(1/q)}$.

Finally if $v= 0$, then we have $J_{\max} \in I_1$ and again we get an upper bound of the form $e^{O(1/q)}$.

Next suppose that $K\ge 1$. 
Here we note that for $J\in I_M$ we have, as $p\to 0$, 
\[
\sum_{R\le v+ J \frac{\log(1/p)}{\log(p/q)}} \left( u (q/p)^{-K} \right)^{R-v} \sim 
\left( u (q/p)^{-K} \right)^{-M-v}.
\]
Actually we could also work with an error term. However, in order to make the
following computations more transparent we concentrate on the leading term.
Since 
\[
\sum_{L=0}^{-J-K} \xi_{L+1} p^L \frac{(-1)^{-J-K-L}}{(-J-K-L)!} = [z^{-J-K}] e^{z/2 + O(qz^2) - z}
\]
we get
\begin{eqnarray*}
C_{31,K,M}&:=& \sum_{J\in I_M,\, J \le -K} p^{J(J+1)/2 +JK}q^J  \sum_{R\le v+ J \frac{\log(1/p)}{\log(p/q)} } 
 \left( u \left(\frac qp\right)^{-K} \right)^{R-v} \\
&&\qquad \times\sum_{L=0}^{-J-K} \xi_{L+1}p^{L} \frac{(-1)^{-J-L-K}}{(-J-L-K)!}     \\
&\sim &  \sum_{J\in I_M,\, J \le -K}  p^{J(J+1)/2 +JK}q^J  
 \left( u \left(\frac qp\right)^{-K} \right)^{-M-v} [z^{-J-K}] e^{z/2 + O(qz^2) - z}
\end{eqnarray*}
and consequently if we sum over $K\ge 1$
\begin{eqnarray*}
\sum_{K\ge 1} C_{31,K,M} &\sim &  u^{-M-v}
\sum_{J\in I_M} p^{J(J+1)/2}q^J  \\
&&\qquad \times \sum_{K=1}^{-J} p^{JK}(q/p)^{K(M+v)} [z^{-J-K}] e^{z/2 + O(qz^2) - z} \\
& =  &  u^{-M-v}
\sum_{J\in I_M} p^{J(J+1)/2}q^J  \\
&&\qquad \times      \sum_{K=1}^{-J} p^{JK}(q/p)^{M(1+v)} [z^{-J-K}] e^{z/2 + O(qz^2) - z}.
\end{eqnarray*}
We observe that (for $J\in I_M$) 
\begin{eqnarray*}
 &&\sum_{K=1}^{-J} p^{JK}(q/p)^{K(M+v)} [z^{-J-K}] e^{z/2 + O(qz^2) - z}
 = [z^{-J}] \frac{p^{J}(q/p)^{M+v}z}{1- p^{J}(q/p)^{M+v}z } e^{z/2 + O(qz^2) - z} \\
 &&\qquad \sim   p^{-J^2}(q/p)^{-J(M+v)} e^{z_M/2 + O(q z_M^2) - z_M},
\end{eqnarray*}
where $z_M = p^{-J}(q/p)^{-M-v}$. Note that $z_M$ varies between $1$ and $1/q$ if $J\in I_M$.
However, it will turn out that the asymptotic leading terms will come from $J$ close to 
$- (v+M)\frac{\log q}{\log p}$ which means that $z_M$ asymptotically $1$ and, thus,  the last
exponential term is asymptotically $e^{-1/2}$. The reason is that the term
\[
p^{J(J+1)/2}q^J p^{-J^2}(q/p)^{-J(M+v)} = p^{-J^2/2} q^{J(1-M-v)} p^{J(\frac 12 +M+v)}
\]
has its absolute minimum for $J$ close to $- (v+M-1)\frac{\log q}{\log p}$ and for $J\in I_M$
it gets maximal for $J$ close to $- (v+M)\frac{\log q}{\log p}$, in particular if
\[
J = J_{v,M} := - \left\lfloor (M+v) \left( \frac{\log q}{\log p} -1 \right) \right\rfloor.
\]
Thus, we obtain
\begin{eqnarray*}
\sum_{K\ge 1} C_{31,K,M} &\sim &  e^{-\frac 12} u^{-M-v} p^{-J_{v,M}^2/2} q^{J_{v,M}(1-M-v)} p^{J_{v,M}(\frac 12 +M+v)} \\
&=& e^{ \frac{\log^2 q}q \left( M+v- \frac 12 (M+v)^2 \right) + O(\log^2 q)  }. 
\end{eqnarray*}
Since $(M+v) - \frac 12 (M+v)^2 \le 0$ for $M\ge 2$ (and $0\le v < 1$) it is clear that only the first two
terms corresponding to $M=0$ and $M=1$ are relevant.
Hence, we obtain as a crude estimate
\[
\sum_{K\ge 1} C_{31,K} 
e^{ \frac{\log^2 q}q \left( v- \frac 12 v^2 \right) + O(\log^2 q)  }
+ e^{ \frac{\log^2 q}q \frac 12 \left( 1- v^2  \right) + O(\log^2 q)  }.  
\]
or as a more precise one
\[
\sum_{K\ge 1} C_{31,K} 
\sim 
e^{-\frac 12} (q^{1/2} u)^{-v} q^{-v^2} p^{-J_{v,0}^2/2} q^{J_{v,0}(1-v)} + 
e^{-\frac 12}  q^{-(1+v)(\frac 32 +v) } u^{-1-v} p^{-J_{v,1}^2/2} q^{-J_{v,0}v}
\]
In what follows we will have to study the precise behavior of the first summand. 
We set (as in the analysis of $T_{30}$ 
\[
\delta = - v \left( \frac{\log q}{\log p} -1 \right) - J_{v,0}.
\]
then we have
\[
e^{-\frac 12} (q^{1/2} u)^{-v} q^{-v^2} p^{-J_{v,0}^2/2} q^{J_{v,0}(1-v)} 
\sim e^{-\frac 12} (q^{1/2} u)^{-v} q^{-v^2+v+\delta}e^{(v-v^2/2) \frac{\log^2 q}{\log(1/p)}}.
\]

If $v=0$ or close to $0$ then the second term dominates and we get a lower bound of the form
\[
\sum_{K\ge 1} C_{31,K} 
\ge e^{-\frac 12}  q^{-(1+v)(\frac 32 +v) } u^{-1-v} p^{-J_{v,1}^2/2} q^{-J_{v,0}v}
\]

\subsubsection{The full behavior for $p\to 1$.}

The most significant terms are $T_{30}$ and $T_{31}$ (or $C_{30}$ and $\sum_{K\ge 1} C_{31,K}$). 
So we just have to concentrate on them.
Both consist of two contributions, where the first one is dominating for larger $v$ and
the second one for smaller $v$.

Suppose first that $\frac 12 \le v < 1$. Then the corresponding terms in 
$C_{30}$ and $\sum_{K\ge 1} C_{31,K}$ are
\[
e^{-1} (q^{1/2} u)^{-v} q^{-v^2+v} e^{(p/q)^\delta (1+o(1))} e^{(v(1-\delta)-v^2/2) \frac{\log^2 q}{\log(1/p)}} q^{\delta(2-3v)}
\]
and
\[
e^{-\frac 12} (q^{1/2} u)^{-v} q^{-v^2+v+\delta}e^{(v-v^2/2) \frac{\log^2 q}{\log(1/p)}}
\]
where
\[
\delta = - v \left( \frac{\log q}{\log p} -1 \right) - J_{v,0}.
\]
and
\[
J_{v,0} = - \left\lfloor v \left( \frac{\log q}{\log p} -1 \right) \right\rfloor.
\]
In particular we have $0\le \delta < 1$.

If $\delta > 0$ then it is clear that the term corresponding to $\sum_{K\ge 1} C_{31,K}$
dominates the other one. Thus, 
$T_{31}$ is larger than that of $T_{30}$ which proves positivity in this case.
If $\delta = 0$ then both terms coincide. However, by looking at second order terms
(which we have not worked out here) it again follows that $T_{31}$ dominates.

If $0\le v \le \frac 12$ then we have to look at the terms
\[
\frac 1{\sqrt{2\pi}} e^{(\frac 12 - v) \frac{\log^2 q}{\log(1/p)} }  q^{-1 -v} u^{-1-v} e^{q^{-v}/2}
\]
and
\[
e^{-\frac 12}  q^{-(1+v)(\frac 32 +v) } u^{-1-v} p^{-J_{v,1}^2/2} q^{-J_{v,0}v}
\]
If we replace $J_{v,0}$ by $- \left\lfloor v \left( \frac{\log q}{\log p}-1\right)\right\rfloor$ then 
$p^{-J_{v,1}^2/2} q^{-J_{v,0}v}$ rewrites to $e^{ (1-v^2) \log^2 q/(2 \log(1/p)}$
which is definitely smaller than $e^{(\frac 12 - v) \frac{\log^2 q}{\log(1/p)} }$ if $v> 0$.
If $v=0$ then the subexponential growth of the first term is of order $u^{-1} q^{-1}$ and that of
the second term of order $u^{-1} q^{-3/2}$. Hence again the second term dominates and so 
we get positivity as $p\to 1$ in this case.

\subsection{Proof of Positivity}

In the previous sections we have shown that $\tilde G_k(n) \gg F_0$ 
if $p$ is sufficiently close to $1/2$ or sufficiently close to $1$.
Actually we can do the above analysis even more precisely by giving 
error terms which rigorously proves the lower bound $\tilde G_k(n) \gg F_0$ 
for $0.5 < p \le 0.51$ and for $0.97 \le p < 1$.

Thus it remains to consider the interval $0.51\le p \le 0.97$. 
In this interval we know that
\[
\tilde G_k(n) \sim   C_*(p)\,F_0\, ( - C_1 - C_2 - C_{30} - C_{32} - C_{31,0} + \sum_{K\ge 1} C_{31,K} ),
\]
where the terms $C_1(p,u,v)$, $C_2(p,u,v)$, $C_{30}(p,u,v)$, $C_{32}(p,u,v)$, $C_{31,K}(p,u,v)$ ($K\ge 0$),
in which $u = \overline r_0 / \overline r_1 \in [\sqrt{q/p},\sqrt{p/q}]$ and $v = \langle \overline r_1 \rangle$,
are explicitly given in (\ref{eqC1puv}), (\ref{eqC2puv}), (\ref{eqC30puv}), (\ref{eqC32puv}), (\ref{eqC310puv}), (\ref{eqC31Kpuv}).

In order to make numerical calculations we first replace the infinite sums of 
\[
 - C_1 - C_2 - C_{30} - C_{32} - C_{31,0} + \sum_{K\ge 1} C_{31,K}
\]
by finite sums, that is we just have to consider $|J|\le J_0$, $L\le L_0$, $K\le K_0$, and $R \le R_0$ where 
\[ 
J_0 = 35 , \quad L_0 = 70, \quad K_0 = 80, \quad R_0 = 95.
\] 
The error that we make can be uniformly bounded by $10^{-5}$.

All the terms that appear in the remaining finite sum are continuous in the parameters $p$, $u$, $v$; however there is
a discontinuity in the terms that appear if $R- v- J \frac{\log(1/p)}{\log(p/q)} = 0$. These are only finitely many such cases for $R$.

Hence, in order to show positivity of $- C_1 - C_2 - C_{30} - C_{32} - C_{31,0} + \sum_{K\ge 1} C_{31,K} := C(p, u, v)$ it is 
sufficient to check it in a sufficiently fine three-dimensional grid and a proper analysis of the partial derivatives
and by considering discontinuities if $R- v- J \frac{\log(1/p)}{\log(p/q)} = 0$.

Below we give a table of partial derivatives with respect to $p, u, v$.

\begin{center}
\begin{tabular}{|l|l|l|l|l|l|l|} \hline
$p$ & $u$ & $v$ & $\frac{C(p,u,v)}{\| \nabla C(p,u,v) \|_1 }$ & $\frac{\partial C}{\partial p}$ & $\frac{\partial C}{\partial u}$ & $\frac{\partial C}{\partial v}$  \\ \hline
$0.60$ & $0.8573214$ & $0.200$ & $0.0397029$ & $-37.3216755$ & $-6.87205829924586$ & $-1.70174985214544 \times 10^{-9}$ \\ 
$0.60$ & $0.8573214$ & $0.400$ & $0.0397066$ & $-37.3175544$ & $-6.87205829422766$ & $8.17124146124115 \times 10^{-9}$ \\ 
$0.60$ & $0.8573214$ & $0.600$ & $0.0397105$ & $-37.3133045$ & $-6.87205830278614$ & $6.75015598972095 \times 10^{-9}$ \\ 
$0.60$ & $0.9573214$ & $0.200$ & $0.0477100$ & $-21.9600978$ & $-3.88671232717819$ & $1.68487446217114 \times 10^{-9}$ \\ 
$0.60$ & $0.9573214$ & $0.400$ & $0.0477178$ & $-21.9558812$ & $-3.88671232690196$ & $6.52766729558607 \times 10^{-9}$ \\ 
$0.60$ & $0.9573214$ & $0.600$ & $0.0477256$ & $-21.9516277$ & $-3.88671233292026$ & $2.35500507983488 \times 10^{-9}$ \\ 
$0.60$ & $1.0573214$ & $0.200$ & $0.0556009$ & $-14.2980109$ & $-2.37660414946284$ & $3.01625391330163 \times 10^{-9}$ \\ 
$0.60$ & $1.0573214$ & $0.400$ & $0.0556152$ & $-14.2937099$ & $-2.37660415074092$ & $4.62252458532930 \times 10^{-9}$ \\ 
$0.60$ & $1.0573214$ & $0.600$ & $0.0556294$ & $-14.2894568$ & $-2.37660415496421$ & $-1.60316204755873 \times 10^{-10}$ \\ 
$0.60$ & $1.1573214$ & $0.200$ & $0.0626269$ & $-10.1863196$ & $-1.54288334226127$ & $3.33222338610994 \times 10^{-9}$ \\ 
$0.60$ & $1.1573214$ & $0.400$ & $0.0626503$ & $-10.1819434$ & $-1.54288334394570$ & $2.95541369155217 \times 10^{-9}$ \\ 
$0.60$ & $1.1573214$ & $0.600$ & $0.0626730$ & $-10.1776934$ & $-1.54288334671548$ & $-1.49635859258979 \times 10^{-9}$ \\ 
$0.70$ & $0.7419408$ & $0.200$ & $0.0466821$ & $-7.02015816$ & $-0.941410951563526$ & $0.00277949304106073$ \\ 
$0.70$ & $0.7419408$ & $0.400$ & $0.0468213$ & $-7.00927750$ & $-0.941036859551048$ & $0.00326076664425301$ \\ 
$0.70$ & $0.7419408$ & $0.600$ & $0.0469950$ & $-6.99412985$ & $-0.941080960885188$ & $0.00352113957369227$ \\ 
$0.70$ & $0.8419408$ & $0.200$ & $0.0492989$ & $-5.33811883$ & $-0.631417261109490$ & $0.00300199019243053$ \\ 
$0.70$ & $0.8419408$ & $0.400$ & $0.0495040$ & $-5.32611794$ & $-0.631168855463216$ & $0.00332417515469530$ \\ 
$0.70$ & $0.8419408$ & $0.600$ & $0.0497253$ & $-5.31304507$ & $-0.631258543609903$ & $0.00339509555136175$ \\ 
$0.70$ & $0.9419408$ & $0.200$ & $0.0514611$ & $-4.23520180$ & $-0.447473694530132$ & $0.00317392708737430$ \\ 
$0.70$ & $0.9419408$ & $0.400$ & $0.0517361$ & $-4.22295039$ & $-0.447305509108986$ & $0.00334798714618501$ \\ 
$0.70$ & $0.9419408$ & $0.600$ & $0.0520044$ & $-4.21164153$ & $-0.447402921736284$ & $0.00328784937675408$ \\ 
$0.70$ & $1.0419408$ & $0.200$ & $0.0532287$ & $-3.47206308$ & $-0.330624920881206$ & $0.00330789550107013$ \\ 
$0.70$ & $1.0419408$ & $0.400$ & $0.0535756$ & $-3.45998730$ & $-0.330507053556417$ & $0.00335245633964476$ \\ 
$0.70$ & $1.0419408$ & $0.600$ & $0.0538907$ & $-3.45007283$ & $-0.330597031821256$ & $0.00320084691018963$ \\ 
$0.70$ & $1.1419408$ & $0.200$ & $0.0546466$ & $-2.92151577$ & $-0.252543040933695$ & $0.00341430828054712$ \\ 
$0.70$ & $1.1419408$ & $0.400$ & $0.0550660$ & $-2.90980583$ & $-0.252456509284293$ & $0.00334836078108580$ \\ 
$0.70$ & $1.1419408$ & $0.600$ & $0.0554287$ & $-2.90095183$ & $-0.252534083156064$ & $0.00313143305508135$ \\ 
$0.70$ & $1.2419408$ & $0.200$ & $0.0557535$ & $-2.51064207$ & $-0.198272264594124$ & $0.00350094009471391$ \\ 
$0.70$ & $1.2419408$ & $0.400$ & $0.0562453$ & $-2.49936355$ & $-0.198205173826516$ & $0.00334130490875495$ \\ 
$0.70$ & $1.2419408$ & $0.600$ & $0.0566567$ & $-2.49129520$ & $-0.198269800136153$ & $0.00307611303140831$ \\ 
$0.70$ & $1.3419408$ & $0.200$ & $0.0565844$ & $-2.19510393$ & $-0.159366055272336$ & $0.00357337397538515$ \\ 
$0.70$ & $1.3419408$ & $0.400$ & $0.0571484$ & $-2.18425793$ & $-0.159310991691086$ & $0.00333409085406799$ \\ 
$0.70$ & $1.3419408$ & $0.600$ & $0.0576098$ & $-2.17675830$ & $-0.159363902270115$ & $0.00303165893500434$ \\ 
$0.70$ & $1.4419408$ & $0.200$ & $0.0571731$ & $-1.94654812$ & $-0.130808089307877$ & $0.00363555773552626$ \\ 
$0.70$ & $1.4419408$ & $0.400$ & $0.0578086$ & $-1.93610396$ & $-0.130760449668088$ & $0.00332801239988356$ \\ 
$0.70$ & $1.4419408$ & $0.600$ & $0.0583219$ & $-1.92900552$ & $-0.130803443699534$ & $0.00299540897730211$ \\ \hline
\end{tabular}
\end{center}

\begin{center}
\begin{tabular}{|l|l|l|l|l|l|l|} \hline
$p$ & $u$ & $v$ & $\frac{C(p,u,v)}{\| \nabla C(p,u,v) \|_1 }$ & $\frac{\partial C}{\partial p}$ & $\frac{\partial C}{\partial u}$ & $\frac{\partial C}{\partial v}$  \\ \hline
$0.80$ & $0.65$ & $0.200$ & $0.0248967$ & $-1.46956589776437$ & $-0.119459648828979$ & $0.0600241100414678$ \\ 
$0.80$ & $0.65$ & $0.400$ & $0.0290076$ & $-1.49636288594479$ & $-0.140905496493815$ & $-0.00268991132656993$ \\
$0.80$ & $0.65$ & $0.600$ & $0.0231229$ & $-1.58402037698124$ & $-0.134370873496437$ & $-0.0531169107773621$ \\ 
$0.80$ & $0.75$ & $0.200$ & $0.0237794$ & $-1.18806547561690$ & $-0.0793871424775716$ & $0.0485832481444959$ \\ 
$0.80$ & $0.75$ & $0.400$ & $0.0276804$ & $-1.20264660480984$ & $-0.0935985673038431$ & $-0.00629235525195782$ \\
$0.80$ & $0.75$ & $0.600$ & $0.0215052$ & $-1.26861837550507$ & $-0.0871968711919635$ & $-0.0437779185560316$ \\
$0.80$ & $0.85$ & $0.200$ & $0.0225818$ & $-0.992277554530574$ & $-0.0563859753697216$ & $0.0402965617070095$ \\ 
$0.80$ & $0.85$ & $0.400$ & $0.0262502$ & $-0.998636356683846$ & $-0.0661789200080420$ & $-0.00814619885147749$ \\ 
$0.80$ & $0.85$ & $0.600$ & $0.0198854$ & $-1.05095720444126$ & $-0.0603108685055531$ & $-0.0365389341325795$ \\ 
$0.80$ & $0.95$ & $0.200$ & $0.0213276$ & $-0.848335491866692$ & $-0.0420003183307927$ & $0.0341005124937510$ \\ 
$0.80$ & $0.95$ & $0.400$ & $0.0247639$ & $-0.849069631954080$ & $-0.0489517334756329$ & $-0.00902477816566716$ \\ 
$0.80$ & $0.95$ & $0.600$ & $0.0182823$ & $-0.892403384526119$ & $-0.0437039697658292$ & $-0.0308048695885077$ \\ 
$0.80$ & $1.05$ & $0.200$ & $0.0200260$ & $-0.738163948966530$ & $-0.0324361927255268$ & $0.0293485365574497$ \\ 
$0.80$ & $1.05$ & $0.400$ & $0.0232414$ & $-0.735010372125089$ & $-0.0374844296402443$ & $-0.00933903507416289$ \\ 
$0.80$ & $1.05$ & $0.600$ & $0.0166997$ & $-0.772209946390490$ & $-0.0328328708576464$ & $-0.0261777524244167$ \\ 
$0.80$ & $1.15$ & $0.200$ & $0.0186804$ & $-0.651184626946133$ & $-0.0257793005573603$ & $0.0256278809160904$ \\ 
$0.80$ & $1.15$ & $0.400$ & $0.0216908$ & $-0.645332988952418$ & $-0.0295103956773346$ & $-0.00931950613392019$ \\ 
$0.80$ & $1.15$ & $0.600$ & $0.0151354$ & $-0.678217958622440$ & $-0.0253954332265494$ & $-0.0223840630155792$ \\
$0.80$ & $1.25$ & $0.200$ & $0.0172908$ & $-0.580782272038505$ & $-0.0209780958471129$ & $0.0226643319933828$ \\ 
$0.80$ & $1.25$ & $0.400$ & $0.0201143$ & $-0.573064137057600$ & $-0.0237727883387606$ & $-0.00910101858409007$ \\ 
$0.80$ & $1.25$ & $0.600$ & $0.0135846$ & $-0.602835170141702$ & $-0.0201279084848238$ & $-0.0192308538302655$ \\ 
$0.80$ & $1.35$ & $0.200$ & $0.0158557$ & $-0.522603802636468$ & $-0.0174163026258611$ & $0.0202695748470205$ \\ 
$0.80$ & $1.35$ & $0.400$ & $0.0185105$ & $-0.513610382057550$ & $-0.0195302017260701$ & $-0.00876498700819184$ \\ 
$0.80$ & $1.35$ & $0.600$ & $0.0120405$ & $-0.541082442332197$ & $-0.0162926045703671$ & $-0.0165787548596086$ \\ 
$0.80$ & $1.45$ & $0.200$ & $0.0143716$ & $-0.473664745896940$ & $-0.0147138076727060$ & $0.0183106129867383$ \\ 
$0.80$ & $1.45$ & $0.400$ & $0.0168756$ & $-0.463822268983449$ & $-0.0163233325025658$ & $-0.00836177299845531$ \\ 
$0.80$ & $1.45$ & $0.600$ & $0.0104955$ & $-0.489561844233322$ & $-0.0134373556761602$ & $-0.0143249441002524$ \\ 
$0.80$ & $1.55$ & $0.200$ & $0.0128332$ & $-0.431847222870374$ & $-0.0126263003039639$ & $0.0166912413419595$ \\ 
$0.80$ & $1.55$ & $0.400$ & $0.0152035$ & $-0.421466433564888$ & $-0.0138561256477487$ & $-0.00792306972385859$ \\ 
$0.80$ & $1.55$ & $0.600$ & $0.0089402$ & $-0.445875381153371$ & $-0.0112737912587590$ & $-0.0123921213806000$ \\ 
$0.80$ & $1.65$ & $0.200$ & $0.0112335$ & $-0.395602281798801$ & $-0.0109913985113508$ & $0.0153404116502998$ \\ 
$0.80$ & $1.65$ & $0.400$ & $0.0134857$ & $-0.384911693686263$ & $-0.0119314480429011$ & $-0.00746903628368045$ \\ 
$0.80$ & $1.65$ & $0.600$ & $0.0073638$ & $-0.408280137456529$ & $-0.00961191538806361$ & $-0.0107211906978932$ \\ 
$0.80$ & $1.75$ & $0.200$ & $0.0095627$ & $-0.363765195132260$ & $-0.00969821012120065$ & $0.0142046951836505$ \\ 
$0.80$ & $1.75$ & $0.400$ & $0.0117111$ & $-0.352933892131091$ & $-0.0104145054677929$ & $-0.00701254332113876$ \\ 
$0.80$ & $1.75$ & $0.600$ & $0.0057536$ & $-0.375474526805419$ & $-0.00832317642363023$ & $-0.00926629099495813$ \\ 
$0.80$ & $1.85$ & $0.200$ & $0.0078081$ & $-0.335436119925703$ & $-0.00866939342358819$ & $0.0132432673112248$ \\ 
$0.80$ & $1.85$ & $0.400$ & $0.0098652$ & $-0.324589674384868$ & $-0.00921116843244363$ & $-0.00656177159186200$ \\ 
$0.80$ & $1.85$ & $0.600$ & $0.0040938$ & $-0.346460554688122$ & $-0.00731875624637723$ & $-0.00799135079532221$ \\ 
$0.80$ & $1.95$ & $0.200$ & $0.0059525$ & $-0.309899919159307$ & $-0.00785021919114115$ & $0.0124244891566150$ \\ 
$0.80$ & $1.95$ & $0.400$ & $0.0079291$ & $-0.299131685380871$ & $-0.00825469764720310$ & $-0.00612184077297684$ \\ 
$0.80$ & $1.95$ & $0.600$ & $0.0023648$ & $-0.320451686228296$ & $-0.00653637328440482$ & $-0.00686765290325297$ \\ \hline
\end{tabular}
\end{center}

Supposing that on a three-dimensional grid of values, we find that the minimum value for $C(p, u, v)$ is some $\delta > 0$, and that the sum of the partial derivatives is bounded above in absolute value by some $D$.  Then,
by the multivariate form of Taylor's theorem, it is sufficient for the minimum distance between adjacent grid points to
be at most $|\delta / D|$ in order for us to conclude that $ C(p, u, v)$ is positive over its entire domain.
In particular, we observe numerically that $|\delta / D|$, when $\delta/D$ is negative, is bounded below by approximately $0.001$.

The following table gives some sample values:

\begin{center}
\begin{tabular}{|l|l|l|l|} \hline
$p$ & $u$ & $v$ & $C(p, u, v)$ \\
\hline
$0.51$ & $0.980196058819607$ & $0.400$ & $29.9108624383664$ \\ 
$0.53$ & $0.941696582148512$ & $0.400$ & $10.0526228338624$ \\ 
$0.53$ & $0.991696582148512$ & $0.400$ & $5.88642519475933$ \\ 
$0.53$ & $1.04169658214851$ & $0.400$ & $3.93987124883441$ \\ 
$0.57$ & $0.868553950490285$ & $0.400$ & $3.52426761973916$ \\ 
$0.57$ & $0.918553950490285$ & $0.400$ & $2.69816719720437$ \\ 
$0.57$ & $0.968553950490285$ & $0.400$ & $2.13795667104722$ \\ 
$0.57$ & $1.01855395049029$ & $0.400$ & $1.74575802150213$ \\ 
$0.57$ & $1.06855395049029$ & $0.400$ & $1.46363184340248$ \\ 
$0.57$ & $1.11855395049029$ & $0.400$ & $1.25604047995371$ \\
$0.61$ & $0.799590058902111$ & $0.400$ & $1.77277314367910$ \\ 
$0.61$ & $0.849590058902111$ & $0.400$ & $1.46518133627806$ \\ 
$0.61$ & $0.899590058902111$ & $0.400$ & $1.23455159510288$ \\ 
$0.61$ & $0.949590058902111$ & $0.400$ & $1.05777607664529$ \\ 
$0.61$ & $0.999590058902111$ & $0.400$ & $0.919664990997429$ \\
$0.61$ & $1.04959005890211$ & $0.400$ & $0.809941704967424$ \\ 
$0.61$ & $1.09959005890211$ & $0.400$ & $0.721477392051825$ \\ 
$0.61$ & $1.14959005890211$ & $0.400$ & $0.649215937996586$ \\ 
$0.61$ & $1.19959005890211$ & $0.400$ & $0.589499096328785$ \\ 
$0.61$ & $1.24959005890211$ & $0.400$ & $0.539631267061691$ \\ 
$0.65$ & $0.733799385705343$ & $0.400$ & $0.982651282114738$ \\
$0.65$ & $0.783799385705343$ & $0.400$ & $0.839124203522317$ \\
$0.65$ & $0.833799385705343$ & $0.400$ & $0.727100276432193$ \\
$0.65$ & $0.883799385705343$ & $0.400$ & $0.637940561184182$ \\
$0.65$ & $0.933799385705343$ & $0.400$ & $0.565793172609293$ \\
$0.65$ & $0.983799385705343$ & $0.400$ & $0.506567343973066$ \\
$0.65$ & $1.03379938570534$ & $0.400$ & $0.457330788676269$ \\ 
$0.65$ & $1.08379938570534$ & $0.400$ & $0.415936629134123$ \\ 
$0.65$ & $1.13379938570534$ & $0.400$ & $0.380783719541390$ \\ 
$0.65$ & $1.18379938570534$ & $0.400$ & $0.350658093707974$ \\ 
$0.65$ & $1.23379938570534$ & $0.400$ & $0.324625450046327$ \\ 
$0.65$ & $1.28379938570534$ & $0.400$ & $0.301956613678465$ \\ 
$0.65$ & $1.33379938570534$ & $0.400$ & $0.282074753241104$ \\ \hline
\end{tabular}
\end{center}

\begin{center}
\begin{tabular}{|l|l|l|l|} \hline
$p$ & $u$ & $v$ & $C(p, u, v)$ \\
\hline
$0.69$ & $0.670280062599836$ & $0.400$ & $0.550461162927249$ \\
$0.69$ & $0.720280062599836$ & $0.400$ & $0.476560662585850$ \\
$0.69$ & $0.770280062599836$ & $0.400$ & $0.418085884501238$ \\
$0.69$ & $0.820280062599836$ & $0.400$ & $0.370869235479415$ \\
$0.69$ & $0.870280062599836$ & $0.400$ & $0.332104766553185$ \\
$0.69$ & $0.920280062599836$ & $0.400$ & $0.299828522402899$ \\
$0.69$ & $0.970280062599836$ & $0.400$ & $0.272625293598979$ \\
$0.69$ & $1.02028006259984$ & $0.400$ & $0.249451145554923$ \\ 
$0.69$ & $1.07028006259984$ & $0.400$ & $0.229520731342754$ \\ 
$0.69$ & $1.12028006259984$ & $0.400$ & $0.212233117795916$ \\ 
$0.69$ & $1.17028006259984$ & $0.400$ & $0.197121539103930$ \\ 
$0.69$ & $1.22028006259984$ & $0.400$ & $0.183818532413344$ \\ 
$0.69$ & $1.27028006259984$ & $0.400$ & $0.172031239042804$ \\ 
$0.69$ & $1.32028006259984$ & $0.400$ & $0.161523580019521$ \\ 
$0.69$ & $1.37028006259984$ & $0.400$ & $0.152103170273795$ \\ 
$0.69$ & $1.42028006259984$ & $0.400$ & $0.143611550335309$ \\ 
$0.69$ & $1.47028006259984$ & $0.400$ & $0.135916766083619$ \\ 
$0.73$ & $0.608163640559537$ & $0.400$ & $0.293753647044937$ \\
$0.73$ & $0.658163640559537$ & $0.400$ & $0.254521707499460$ \\
$0.73$ & $0.708163640559537$ & $0.400$ & $0.223617820853438$ \\
$0.73$ & $0.758163640559537$ & $0.400$ & $0.198689182699194$ \\
$0.73$ & $0.808163640559537$ & $0.400$ & $0.178200553380755$ \\
$0.73$ & $0.858163640559537$ & $0.400$ & $0.161099962106000$ \\
$0.73$ & $0.908163640559537$ & $0.400$ & $0.146640237980822$ \\
$0.73$ & $0.958163640559537$ & $0.400$ & $0.134275597770952$ \\
$0.73$ & $1.00816364055954$ & $0.400$ & $0.123598129339214$ \\ 
$0.73$ & $1.05816364055954$ & $0.400$ & $0.114297028223827$ \\ 
$0.73$ & $1.10816364055954$ & $0.400$ & $0.106131513996562$ \\ 
$0.73$ & $1.15816364055954$ & $0.400$ & $0.0989123176383000$ \\
$0.73$ & $1.20816364055954$ & $0.400$ & $0.0924887213278005$ \\
$0.73$ & $1.25816364055954$ & $0.400$ & $0.0867392964169298$ \\
$0.73$ & $1.30816364055954$ & $0.400$ & $0.0815651633709091$ \\
$0.73$ & $1.35816364055954$ & $0.400$ & $0.0768850068928071$ \\
$0.73$ & $1.40816364055954$ & $0.400$ & $0.0726313343834271$ \\
$0.73$ & $1.45816364055954$ & $0.400$ & $0.0687476288039068$ \\
$0.73$ & $1.50816364055954$ & $0.400$ & $0.0651861534672111$ \\
$0.73$ & $1.55816364055954$ & $0.400$ & $0.0619062371781993$ \\
$0.73$ & $1.60816364055954$ & $0.400$ & $0.0588729160730193$ \\ \hline
\end{tabular}
\end{center}

\begin{center}
\begin{tabular}{|l|l|l|l|} \hline
$p$ & $u$ & $v$ & $C(p, u, v)$ \\
\hline
$0.77$ & $0.546535725000021$ & $0.400$ & $0.145533658462583$ \\ 
$0.77$ & $0.596535725000021$ & $0.400$ & $0.124327954267287$ \\ 
$0.77$ & $0.646535725000021$ & $0.400$ & $0.107945853943789$ \\ 
$0.77$ & $0.696535725000021$ & $0.400$ & $0.0949108647641168$ \\
$0.77$ & $0.746535725000021$ & $0.400$ & $0.0843029431585052$ \\
$0.77$ & $0.796535725000021$ & $0.400$ & $0.0755136117642792$ \\
$0.77$ & $0.846535725000021$ & $0.400$ & $0.0681223833348810$ \\
$0.77$ & $0.896535725000021$ & $0.400$ & $0.0618286193771169$ \\
$0.77$ & $0.946535725000021$ & $0.400$ & $0.0564114016933814$ \\
$0.77$ & $0.996535725000021$ & $0.400$ & $0.0517046696768304$ \\
$0.77$ & $1.04653572500002$ & $0.400$ & $0.0475811789582821$ \\ 
$0.77$ & $1.09653572500002$ & $0.400$ & $0.0439418020113962$ \\ 
$0.77$ & $1.14653572500002$ & $0.400$ & $0.0407081902919728$ \\ 
$0.77$ & $1.19653572500002$ & $0.400$ & $0.0378176207163818$ \\ 
$0.77$ & $1.24653572500002$ & $0.400$ & $0.0352193008652932$ \\ 
$0.77$ & $1.29653572500002$ & $0.400$ & $0.0328716716127033$ \\ 
$0.77$ & $1.34653572500002$ & $0.400$ & $0.0307404059773191$ \\ 
$0.77$ & $1.39653572500002$ & $0.400$ & $0.0287969028851336$ \\ 
$0.77$ & $1.44653572500002$ & $0.400$ & $0.0270171384844815$ \\ 
$0.77$ & $1.49653572500002$ & $0.400$ & $0.0253807795447578$ \\ 
$0.77$ & $1.54653572500002$ & $0.400$ & $0.0238704914680881$ \\ 
$0.77$ & $1.59653572500002$ & $0.400$ & $0.0224713924978488$ \\ 
$0.77$ & $1.64653572500002$ & $0.400$ & $0.0211706188873642$ \\ 
$0.77$ & $1.69653572500002$ & $0.400$ & $0.0199569750386104$ \\ 
$0.77$ & $1.74653572500002$ & $0.400$ & $0.0188206491934295$ \\ 
$0.77$ & $1.79653572500002$ & $0.400$ & $0.0177529799785283$ \\ 
$0.81$ & $0.484322104837853$ & $0.400$ & $0.0760870936626361$ \\
$0.81$ & $0.534322104837853$ & $0.400$ & $0.0675499639584132$ \\
$0.81$ & $0.584322104837853$ & $0.400$ & $0.0605278017367255$ \\
$0.81$ & $0.634322104837853$ & $0.400$ & $0.0546912979618810$ \\
$0.81$ & $0.684322104837853$ & $0.400$ & $0.0497865208840267$ \\
$0.81$ & $0.734322104837853$ & $0.400$ & $0.0456211260755595$ \\
$0.81$ & $0.784322104837853$ & $0.400$ & $0.0420491692950691$ \\
$0.81$ & $0.834322104837853$ & $0.400$ & $0.0389589632508773$ \\
$0.81$ & $0.884322104837853$ & $0.400$ & $0.0362641022489925$ \\
$0.81$ & $0.934322104837853$ & $0.400$ & $0.0338969726150253$ \\
$0.81$ & $0.984322104837853$ & $0.400$ & $0.0318040708586409$ \\
$0.81$ & $1.03432210483785$ & $0.400$ & $0.0299426063306214$ \\ 
$0.81$ & $1.08432210483785$ & $0.400$ & $0.0282780131573759$ \\ 
$0.81$ & $1.13432210483785$ & $0.400$ & $0.0267821087966240$ \\ 
$0.81$ & $1.18432210483785$ & $0.400$ & $0.0254317162361133$ \\ 
$0.81$ & $1.23432210483785$ & $0.400$ & $0.0242076219697651$ \\ 
$0.81$ & $1.28432210483785$ & $0.400$ & $0.0230937796962394$ \\ 
$0.81$ & $1.33432210483785$ & $0.400$ & $0.0220766957561978$ \\ 
$0.81$ & $1.38432210483785$ & $0.400$ & $0.0211449503674146$ \\ 
$0.81$ & $1.43432210483785$ & $0.400$ & $0.0202888213490553$ \\ 
$0.81$ & $1.48432210483785$ & $0.400$ & $0.0194999859288600$ \\ 
$0.81$ & $1.53432210483785$ & $0.400$ & $0.0187712825802748$ \\  \hline
\end{tabular}
\end{center}

\begin{center}
\begin{tabular}{|l|l|l|l|} \hline
$p$ & $u$ & $v$ & $C(p, u, v)$ \\
\hline
$0.81$ & $1.58432210483785$ & $0.400$ & $0.0180965194068392$ \\ 
$0.81$ & $1.63432210483785$ & $0.400$ & $0.0174703189090337$ \\ 
$0.81$ & $1.68432210483785$ & $0.400$ & $0.0168879914158993$ \\ 
$0.81$ & $1.73432210483785$ & $0.400$ & $0.0163454312572640$ \\ 
$0.81$ & $1.78432210483785$ & $0.400$ & $0.0158390311067080$ \\ 
$0.81$ & $1.83432210483785$ & $0.400$ & $0.0153656109424389$ \\ 
$0.81$ & $1.88432210483785$ & $0.400$ & $0.0149223588338145$ \\ 
$0.81$ & $1.93432210483785$ & $0.400$ & $0.0145067813642754$ \\ 
$0.81$ & $1.98432210483785$ & $0.400$ & $0.0141166619432482$ \\ 
$0.81$ & $2.03432210483785$ & $0.400$ & $0.0137500256186627$ \\ 
$0.85$ & $0.420084025208403$ & $0.400$ & $0.0186691810013144$ \\
$0.85$ & $0.470084025208403$ & $0.400$ & $0.0162646973694791$ \\
$0.85$ & $0.520084025208403$ & $0.400$ & $0.0143518941913499$ \\
$0.85$ & $0.570084025208403$ & $0.400$ & $0.0128045473320526$ \\
$0.85$ & $0.620084025208403$ & $0.400$ & $0.0115330738808552$ \\
$0.85$ & $0.670084025208403$ & $0.400$ & $0.0104734953492880$ \\
$0.85$ & $0.720084025208403$ & $0.400$ & $0.00957942237801035$ \\
$0.85$ & $0.770084025208403$ & $0.400$ & $0.00881663551626843$ \\
$0.85$ & $0.820084025208403$ & $0.400$ & $0.00815946033617365$ \\
$0.85$ & $0.870084025208403$ & $0.400$ & $0.00758832384963171$ \\
$0.85$ & $0.920084025208403$ & $0.400$ & $0.00708808343006240$ \\
$0.85$ & $0.970084025208403$ & $0.400$ & $0.00664686534582870$ \\
$0.85$ & $1.02008402520840$ & $0.400$ & $0.00625524408678757$ \\ 
$0.85$ & $1.07008402520840$ & $0.400$ & $0.00590565337915905$ \\ 
$0.85$ & $1.12008402520840$ & $0.400$ & $0.00559195686946623$ \\ 
$0.85$ & $1.17008402520840$ & $0.400$ & $0.00530913114198484$ \\ 
$0.85$ & $1.22008402520840$ & $0.400$ & $0.00505302824785758$ \\ 
$0.85$ & $1.27008402520840$ & $0.400$ & $0.00482019595983729$ \\ 
$0.85$ & $1.32008402520840$ & $0.400$ & $0.00460774000566744$ \\ 
$0.85$ & $1.37008402520840$ & $0.400$ & $0.00441321739799605$ \\ 
$0.85$ & $1.42008402520840$ & $0.400$ & $0.00423455315103638$ \\ 
$0.85$ & $1.47008402520840$ & $0.400$ & $0.00406997450318158$ \\ 
$0.85$ & $1.52008402520840$ & $0.400$ & $0.00391795871382783$ \\ 
$0.85$ & $1.57008402520840$ & $0.400$ & $0.00377719110929320$ \\ 
$0.85$ & $1.62008402520840$ & $0.400$ & $0.00364653129327053$ \\ 
$0.85$ & $1.67008402520840$ & $0.400$ & $0.00352498564643611$ \\ 
$0.85$ & $1.72008402520840$ & $0.400$ & $0.00341168480917986$ \\ 
$0.85$ & $1.77008402520840$ & $0.400$ & $0.00330586520067300$ \\ 
$0.85$ & $1.82008402520840$ & $0.400$ & $0.00320685363931261$ \\ 
$0.85$ & $1.87008402520840$ & $0.400$ & $0.00311405461434333$ \\ 
$0.85$ & $1.92008402520840$ & $0.400$ & $0.00302693959019962$ \\ 
$0.85$ & $1.97008402520840$ & $0.400$ & $0.00294503804752821$ \\ 
$0.85$ & $2.02008402520840$ & $0.400$ & $0.00286792988845264$ \\ 
$0.85$ & $2.07008402520840$ & $0.400$ & $0.00279523901235734$ \\ 
$0.85$ & $2.12008402520840$ & $0.400$ & $0.00272662785891953$ \\ 
$0.85$ & $2.17008402520840$ & $0.400$ & $0.00266179269510758$ \\ 
$0.85$ & $2.22008402520840$ & $0.400$ & $0.00260045962977529$ \\ 
$0.85$ & $2.27008402520840$ & $0.400$ & $0.00254238115849148$ \\ 
$0.85$ & $2.32008402520840$ & $0.400$ & $0.00248733315356731$ \\ 
$0.85$ & $2.37008402520840$ & $0.400$ & $0.00243511230246440$ \\ \hline
\end{tabular}
\end{center}

\begin{center}
\begin{tabular}{|l|l|l|l|} \hline
$p$ & $u$ & $v$ & $C(p, u, v)$ \\
\hline
$0.89$ & $0.351561524655326$ & $0.400$ & $0.00145292282104492$ \\
$0.89$ & $0.401561524655326$ & $0.400$ & $0.00118577480316162$ \\
$0.89$ & $0.451561524655326$ & $0.400$ & $0.000993013381958008$ \\
$0.89$ & $0.501561524655326$ & $0.400$ & $0.000848412513732910$ \\
$0.89$ & $0.551561524655326$ & $0.400$ & $0.000735878944396973$ \\
$0.89$ & $0.601561524655326$ & $0.400$ & $0.000646829605102539$ \\
$0.89$ & $0.651561524655326$ & $0.400$ & $0.000574707984924316$ \\
$0.89$ & $0.701561524655326$ & $0.400$ & $0.000515460968017578$ \\
$0.89$ & $0.751561524655326$ & $0.400$ & $0.000466108322143555$ \\
$0.89$ & $0.801561524655326$ & $0.400$ & $0.000424087047576904$ \\
$0.89$ & $0.851561524655326$ & $0.400$ & $0.000387966632843018$ \\
$0.89$ & $0.901561524655326$ & $0.400$ & $0.000357031822204590$ \\
$0.89$ & $0.951561524655326$ & $0.400$ & $0.000330328941345215$ \\
$0.89$ & $1.00156152465533$ & $0.400$ & $0.000306487083435059$ \\ 
$0.89$ & $1.05156152465533$ & $0.400$ & $0.000285655260086060$ \\ 
$0.89$ & $1.10156152465533$ & $0.400$ & $0.000267118215560913$ \\ 
$0.89$ & $1.15156152465533$ & $0.400$ & $0.000250488519668579$ \\ 
$0.89$ & $1.20156152465533$ & $0.400$ & $0.000235617160797119$ \\ 
$0.89$ & $1.25156152465533$ & $0.400$ & $0.000222235918045044$ \\ 
$0.89$ & $1.30156152465533$ & $0.400$ & $0.000210016965866089$ \\ 
$0.89$ & $1.35156152465533$ & $0.400$ & $0.000198870897293091$ \\ 
$0.89$ & $1.40156152465533$ & $0.400$ & $0.000188708305358887$ \\ 
$0.89$ & $1.45156152465533$ & $0.400$ & $0.000179469585418701$ \\ 
$0.89$ & $1.50156152465533$ & $0.400$ & $0.000171035528182983$ \\ 
$0.89$ & $1.55156152465533$ & $0.400$ & $0.000163167715072632$ \\ 
$0.89$ & $1.60156152465533$ & $0.400$ & $0.000155717134475708$ \\ 
$0.89$ & $1.65156152465533$ & $0.400$ & $0.000149309635162354$ \\ 
$0.89$ & $1.70156152465533$ & $0.400$ & $0.000143021345138550$ \\ 
$0.89$ & $1.75156152465533$ & $0.400$ & $0.000137194991111755$ \\ 
$0.89$ & $1.80156152465533$ & $0.400$ & $0.000131785869598389$ \\ 
$0.89$ & $1.85156152465533$ & $0.400$ & $0.000126823782920837$ \\ 
$0.89$ & $1.90156152465533$ & $0.400$ & $0.000121921300888062$ \\ 
$0.89$ & $1.95156152465533$ & $0.400$ & $0.000117585062980652$ \\ 
$0.89$ & $2.00156152465533$ & $0.400$ & $0.000113427639007568$ \\ 
$0.89$ & $2.05156152465533$ & $0.400$ & $0.000109493732452393$ \\ 
$0.89$ & $2.10156152465533$ & $0.400$ & $0.000105798244476318$ \\ 
$0.89$ & $2.15156152465533$ & $0.400$ & $0.000102311372756958$ \\ 
$0.89$ & $2.20156152465533$ & $0.400$ & $0.0000989884138107300$ \\
$0.89$ & $2.25156152465533$ & $0.400$ & $0.0000958889722824097$ \\
$0.89$ & $2.30156152465533$ & $0.400$ & $0.0000930279493331909$ \\
$0.89$ & $2.35156152465533$ & $0.400$ & $0.0000902861356735229$ \\
$0.89$ & $2.40156152465533$ & $0.400$ & $0.0000875443220138550$ \\
$0.89$ & $2.45156152465533$ & $0.400$ & $0.0000850409269332886$ \\
$0.89$ & $2.50156152465533$ & $0.400$ & $0.0000825822353363037$ \\
$0.89$ & $2.55156152465533$ & $0.400$ & $0.0000803321599960327$ \\
$0.89$ & $2.60156152465533$ & $0.400$ & $0.0000781416893005371$ \\
$0.89$ & $2.65156152465533$ & $0.400$ & $0.0000760406255722046$ \\
$0.89$ & $2.70156152465532$ & $0.400$ & $0.0000741034746170044$ \\ \hline
\end{tabular}
\end{center}

\begin{center}
\begin{tabular}{|l|l|l|l|} \hline
$p$ & $u$ & $v$ & $C(p, u, v)$ \\
\hline
$0.89$ & $2.75156152465532$ & $0.400$ & $0.0000721588730812073$ \\
$0.89$ & $2.80156152465532$ & $0.400$ & $0.0000703781843185425$ \\
$0.93$ & $0.274351630584367$ & $0.400$ & $5.08388570706080 \times 10^{18}$ \\
$0.93$ & $0.324351630584367$ & $0.400$ & $4.02164584183275 \times 10^{18}$ \\
$0.93$ & $0.374351630584367$ & $0.400$ & $3.29029412228838 \times 10^{18}$ \\
$0.93$ & $0.424351630584367$ & $0.400$ & $2.76064216553067 \times 10^{18}$ \\
$0.93$ & $0.474351630584367$ & $0.400$ & $2.36203204644596 \times 10^{18}$ \\
$0.93$ & $0.524351630584367$ & $0.400$ & $2.05283796276531 \times 10^{18}$ \\
$0.93$ & $0.574351630584367$ & $0.400$ & $1.80707983941192 \times 10^{18}$ \\
$0.93$ & $0.624351630584367$ & $0.400$ & $1.60777544694030 \times 10^{18}$ \\
$0.93$ & $0.674351630584367$ & $0.400$ & $1.44339556221551 \times 10^{18}$ \\
$0.93$ & $0.724351630584367$ & $0.400$ & $1.30586151661967 \times 10^{18}$ \\
$0.93$ & $0.774351630584367$ & $0.400$ & $1.18935865797586 \times 10^{18}$ \\
$0.93$ & $0.824351630584367$ & $0.400$ & $1.08960438065609 \times 10^{18}$ \\
$0.93$ & $0.874351630584367$ & $0.400$ & $1.00338075849318 \times 10^{18}$ \\
$0.93$ & $0.924351630584367$ & $0.400$ & $9.28227105214006 \times 10^{17}$ \\
$0.93$ & $0.974351630584367$ & $0.400$ & $8.62232402787871 \times 10^{17}$ \\
$0.93$ & $1.02435163058437$ & $0.400$ & $8.03891902568485 \times 10^{17}$ \\ 
$0.93$ & $1.07435163058437$ & $0.400$ & $7.52006018232347 \times 10^{17}$ \\ 
$0.93$ & $1.12435163058437$ & $0.400$ & $7.05607724938857 \times 10^{17}$ \\ 
$0.93$ & $1.17435163058437$ & $0.400$ & $6.63909564478263 \times 10^{17}$ \\ 
$0.93$ & $1.22435163058437$ & $0.400$ & $6.26264382665951 \times 10^{17}$ \\ 
$0.93$ & $1.27435163058437$ & $0.400$ & $5.92135844931128 \times 10^{17}$ \\ 
$0.93$ & $1.32435163058437$ & $0.400$ & $5.61076019967838 \times 10^{17}$ \\ 
$0.93$ & $1.37435163058437$ & $0.400$ & $5.32708143146652 \times 10^{17}$ \\ 
$0.93$ & $1.42435163058437$ & $0.400$ & $5.06713224060254 \times 10^{17}$ \\ 
$0.93$ & $1.47435163058437$ & $0.400$ & $4.82819540336550 \times 10^{17}$ \\ 
$0.93$ & $1.52435163058437$ & $0.400$ & $4.60794321949106 \times 10^{17}$ \\ 
$0.93$ & $1.57435163058437$ & $0.400$ & $4.40437114631546 \times 10^{17}$ \\ 
$0.93$ & $1.62435163058437$ & $0.400$ & $4.21574442377283 \times 10^{17}$ \\ 
$0.93$ & $1.67435163058437$ & $0.400$ & $4.04055483735086 \times 10^{17}$ \\ 
$0.93$ & $1.72435163058437$ & $0.400$ & $3.87748545677684 \times 10^{17}$ \\ 
$0.93$ & $1.77435163058437$ & $0.400$ & $3.72538169700794 \times 10^{17}$ \\ 
$0.93$ & $1.82435163058437$ & $0.400$ & $3.58322742657385 \times 10^{17}$ \\ 
$0.93$ & $1.87435163058437$ & $0.400$ & $3.45012513242157 \times 10^{17}$ \\ 
$0.93$ & $1.92435163058437$ & $0.400$ & $3.32527936550119 \times 10^{17}$ \\ 
$0.93$ & $1.97435163058437$ & $0.400$ & $3.20798285548969 \times 10^{17}$ \\ 
$0.93$ & $2.02435163058437$ & $0.400$ & $3.09760480929080 \times 10^{17}$ \\ 
$0.93$ & $2.07435163058437$ & $0.400$ & $2.99358100573060 \times 10^{17}$ \\ 
$0.93$ & $2.12435163058437$ & $0.400$ & $2.89540537512205 \times 10^{17}$ \\ 
$0.93$ & $2.17435163058437$ & $0.400$ & $2.80262281222224 \times 10^{17}$ \\ 
$0.93$ & $2.22435163058437$ & $0.400$ & $2.71482301836819 \times 10^{17}$ \\ 
$0.93$ & $2.27435163058437$ & $0.400$ & $2.63163520611826 \times 10^{17}$ \\ 
$0.93$ & $2.32435163058437$ & $0.400$ & $2.55272352970900 \times 10^{17}$ \\ 
$0.93$ & $2.37435163058437$ & $0.400$ & $2.47778312871162 \times 10^{17}$ \\ 
$0.93$ & $2.42435163058437$ & $0.400$ & $2.40653669169914 \times 10^{17}$ \\ 
$0.93$ & $2.47435163058437$ & $0.400$ & $2.33873146248732 \times 10^{17}$ \\  \hline
\end{tabular}
\end{center}

\begin{center}
\begin{tabular}{|l|l|l|l|} \hline
$p$ & $u$ & $v$ & $C(p, u, v)$ \\
\hline
$0.93$ & $2.52435163058437$ & $0.400$ & $2.27413662434387 \times 10^{17}$ \\ 
$0.93$ & $2.57435163058437$ & $0.400$ & $2.21254100805868 \times 10^{17}$ \\ 
$0.93$ & $2.62435163058437$ & $0.400$ & $2.15375107839294 \times 10^{17}$ \\ 
$0.93$ & $2.67435163058437$ & $0.400$ & $2.09758916054065 \times 10^{17}$ \\ 
$0.93$ & $2.72435163058437$ & $0.400$ & $2.04389187412840 \times 10^{17}$ \\ 
$0.93$ & $2.77435163058437$ & $0.400$ & $1.99250874717724 \times 10^{17}$ \\ 
$0.93$ & $2.82435163058437$ & $0.400$ & $1.94330098653627 \times 10^{17}$ \\ 
$0.93$ & $2.87435163058437$ & $0.400$ & $1.89614038471800 \times 10^{17}$ \\ 
$0.97$ & $0.50$ & $0.60$ & $9.17733198126610 \times 10^{72}$ \\
$0.97$ & $1.00$ & $0.60$ & $6.05478107453485 \times 10^{72}$ \\ 
$0.97$ & $3.00$ & $0.60$ & $3.13202840384780 \times 10^{72}$ \\
$0.97$ & $5.00$ & $0.60$ & $2.30524156812013 \times 10^{72}$ \\
\hline
\end{tabular}
\end{center}

\bigskip

This finally proves the lower bound $\tilde G_k(n) \gg F_0$  for every $p\in (\frac 12, 1)$.

\label{lastpage}

\end{document}